\documentclass[11pt]{article} 
\usepackage[english]{babel}
\usepackage{graphicx,amsmath,amssymb,amsthm}
\usepackage[all]{xy}
\usepackage[utf8x]{inputenc}
\newtheorem{theorem}{Theorem}
\newtheorem{proposition}{Proposition}
\newtheorem{corollary}{Corollary}
\newtheorem{definition}{Definition}
\newtheorem{example}{Example}
\newtheorem{remark}{Remark}
\usepackage[all]{xy}
\usepackage{youngtab}
\usepackage{float}
\restylefloat{table}
\def\P{I\!\!P}

\def\O{{\cal O}}
\def\CC{{\mathbb C}}
\def\ZZ{{\mathbb Z}}
\def\rig#1{\smash{ \mathop{\longrightarrow}
    \limits^{#1}}}

\def\lef#1{\smash{ \mathop{\longleftarrow}
    \limits^{#1}}}

\begin{document}

\title{Five Lectures on Projective Invariants}
\date{}
\author{Giorgio Ottaviani\footnote{member of GNSAGA-INDAM}}              
\maketitle

%

\begin{abstract}
 We introduce invariant rings for forms (homogeneous polynomials) and for $d$ points on the projective space, from the point of view of representation theory. 
We discuss several examples, addressing some computational issues. We introduce the graphical algebra for the invariants of $d$ points on the line.
This is an expanded version of the notes for the School on Invariant Theory and Projective Geometry, 
Trento, September 17-22, 2012.
\\

\end{abstract}
{Le teorie vanno e vengono ma le formule restano.\footnotemark[2]}\footnotetext[2]{The theories may come and go but the formulas remain} {\it G.C. Rota}
\tableofcontents
\section{Introduction and first examples}
\subsection{What is invariant theory, and the more modest aim of these lectures}

Invariant theory is a classical and superb chapter of mathematics.
It can be pursued from many points of view, and there are several excellent introductions to the subject (\cite{Br, DK, Do2, Ho, KP, KR, Muk, Olv, Pr, PV, Stu} and many others).

Given a group $G$ acting on a variety $X$, we want to describe the invariant subring $A(X)^G$ inside the coordinate ring $A(X)$.
This framework is very general, in the spirit of Klein's Erlangen Program.

Most of the classical work on the topic was done on invariants of forms and invariants of sets  of points.
In the case of invariant of forms, $V$ is a complex vector space, $G=SL(V)$ and $X$
is the natural embedding of $\P V$ in $\P S^dV$, which is called the $d$-Veronese variety
($S^d V$ is the $d$-th symmetric power of $V$).
In the case of invariants of points, there are two
interesting situations. When the points are {\it ordered}, we have $G=SL(V)$ acting on $d$ copies of $\P V$,
that is on the Segre variety $\P V\times\ldots\times\P V$.
When the points are {\it unordered}, we quotient 
the Segre variety by the symmetric group $\Sigma_d$.
Note that when $\dim V=2$, $d$ unordered points are described in equivalent way
by a homogeneous polynomial of degree $d$ in two variables,
and we reduce again to the case of forms, here $X$ is the rational normal curve.
Note that, in the dual description, points correspond to hyperplanes
and we get the well known ``arrangement of hyperplanes''.

In XIX century invariants were constructed by means of the two fundamental theorems,
that we review in \S \ref{twoffforms} for invariants of forms and in \S \ref{sectionFTpoints} for invariants of points.
The First Fundamental Theorem (1FT for short) claims that all the invariants can be constructed by a clever combinatorial procedure
called ``symbolic representation'' (see \S \ref{symbolicrep}).
Our approach is close to \cite{Pr, KR}, where the 1FT  is obtained as a consequence of Schur-Weyl duality.
The classical literature is rich of interesting examples and computations. As main textbooks from the classical period we recommend
\cite{GY}. Also several parts from \cite{Sal, EC} are developed in the setting of invariant theory.
Hilbert lectures \cite{Hilb} deserve a special mention. They are the translation of handwritten notes
of a course held by Hilbert in 1897 at G\"ottingen. The reading of such master work is particularly congenial to understand the modern development
of invariant theory.
The ``symbolic representation'' of XIX century was a hidden way to introduce 
 Schur functors and representation theory, not yet having a formal setting for  them.

In the case of forms, the invariants of degree $m$ for $d$-forms are the $SL(V)$-invariant subspace of
$S^m(S^dV)$. The decomposition of this space as a sum of irreducible representations is a difficult problem called {\it plethysm}. Although
there are algorithms computing this decomposition, for any $m$, $d$, 
a simple description is missing.

The case $\dim V=2$ is quite special. The 
elements of $S^dV$ are called binary forms of degree $d$ 
(ternary forms correspond to $\dim V=3$ and so on) and most  of classical results regarded this case.
The common zero locus of all the invariant functions is called the ``nullcone'', and it coincides
with the nilpotent cone of Manivel's lectures \cite{Man}
when $X$ is the Lie algebra of $G$ and the action is the adjoint one.

In these lectures we try to give some tools
to apply and use projective invariants, possibly in related 
fields, involving algebraic geometry or commutative algebra. This aim 
should  guide the volunteered reader into useful and beautiful mathematics.

There are several approaches to invariant theory. Some examples look ``natural'' and ``easy'' just from one point of view, while they look more sophisticated from other points of view.
So it is important to have a plurality of descriptions for the invariants, and to look for several
examples. During lectures, understanding is more important than efficiency, so we may prove the same result more than once, from
different points of view. 

Computers opened a new era in invariant theory.
Anyway, there are basic cases where the needed computations  
are out of reach, even with the help of a computer,
and even more  if the computer is used in a naive way.
We will try to sketch some computational tricks that we found useful in mathematical practice.
Our  basic computational sources are \cite{Stu,DK}.

I am indebted to M. Bolognesi, C. Ciliberto, A. Conca, I. Dolgachev, D. Faenzi, J.M. Landsberg, L. Manivel, L. Oeding, C. Ritzenthaler, E. Sernesi, B. Sturmfels and J. Weyman for several discussions.
The \S \ref{gherardelli} is due to an unpublished idea of my diploma advisor, F. Gherardelli.
These notes are an expanded version of the notes prepared for the School on Invariant Theory and Projective Geometry held in
Trento, in September 2012, organized by CIRM. I wish to thank the CIRM and 
V. Baldoni, G. Casnati, C. Fontanari, F. Galluzzi, R. Notari, F. Vaccarino  for the wonderful organization.
I wish to thank all the participants for their interest, it is amazing that a question left open during the course, about
L\"uroth quartics, has been immediately attacked and essentially solved \cite{BLRS}.

\subsection{The Veronese variety and its equations}\label{veronesesubsection}
Let $V$ be a (complex) vector space of dimension $n+1$.
We denote by $S^dV$ the $d$-th symmetric power of $V$.
The $d$-Veronese variety embedded in $\P S^dV$ is the 
image of the map
$$\begin{array}{ccc}
\P V&\to&\P S^dV\\
v&\mapsto &v^d\end{array}$$

We denote it by $v_d(\P V)$, it consists of all homogeneous polynomials of degree $d$ in $n+1$ variables which
are the $d$-th power of a linear form.
Historically, this construction gave the main motivation to study algebraic geometry in higher dimensional spaces. Its importance is due to the fact
that hypersurfaces of degree $d$ in $\P V$ are cut out by hyperplanes
in the Veronese embedding.

These elementary remarks are summarized in the following.

\begin{theorem}\label{veronesespan}
\
\begin{itemize}

\item{(i)} A linear function on $S^dV$ is uniquely determined
by its restriction to the Veronese variety.

\item{(ii)} Linear functions over $S^dV$ correspond to homogeneous polynomials of degree $d$ over $V$.
\end{itemize}
\end{theorem}
\begin{proof} Let $H$ be a linear function which vanishes
on the $d$-th Veronese variety.
It induces a homogeneous polynomial of degree $d$ which vanishes for all the
values of the variables. Hence the polynomial is zero, proving (i). (ii) is proved in the same way, since both spaces have the same dimension.
\end{proof}

In equivalent way,  Theorem \ref{veronesespan} says that $\P S^dV$ is spanned by elements lying on the Veronese variety.
To make effective the previous Theorem, compare a general polynomial
$$f=\sum_{i_0+\ldots+i_n=d}\frac{d!}{i_0!\ldots i_n!}a_{i_0,\ldots, i_n}x_0^{i_0}\ldots x_n^{i_n}$$
with the $d$-th power
$$(b_0x_0+\ldots +b_nx_n)^d=\sum_{i_0+\ldots+i_n=d}\frac{d!}{i_0!\ldots i_n!}b_0^{i_0}\ldots b_n^{i_n}x_0^{i_0}\ldots x_n^{i_n}$$
getting the correspondence
\begin{equation}\label{explicitveronesespan}b_0^{i_0}\ldots b_n^{i_n}\mapsto a_{i_0,\ldots, i_n}
\end{equation}

It will be the basic tool
for the symbolic representation of an invariant that we will see in \S \ref{symbolicrep}.

The conormal space (which is the annihilator of the tangent space)
at a point $x\in v_d(\P V)$ can be identified with the space of hyperplanes 
in $\P S^dV$ containing the tangent space at $x$. These hyperplanes correspond to the hypersurfaces of degree $d$ which are singular at $x$, which give
the vector space
$H^0(I_{x^2}(d))$. 
This can be summarized with
\begin{proposition}[Lasker Lemma]\label{lasker}
The conormal space of $v_d(\P V)$ at $[x]$ is isomorphic to
$H^0(I_{x^2}(d))$.
\end{proposition}

\begin{theorem}
The Veronese variety is defined as scheme by the quadrics
 which  are the $2\times 2$-minors of the contraction
\begin{equation}\label{cf}V^{\vee}\rig{C_{f}} S^{d-1}V\end{equation}
for $f\in S^dV$.
\end{theorem}
\begin{proof}
We have to prove that $f$ is a power of a linear form if and only if $\mathrm{rk\ }C_f=1$. The elements in $V^{\vee}$ can be seen as differential operators  of first order.
If $f=l^d$ then $\frac{\partial f}{\partial x_i}=d\frac{\partial f}{\partial x_i}l^{d-1}$ so that $\mathrm{Im}C_f$ is spanned by $l^{d-1}$. It follows
$\mathrm{rk\ }C_f=1$. Conversely, assume that $\mathrm{rk\ }C_f=1$
and let $l\in V$ be a generator of the one dimensional annihilator of $\ker C_f$.
We may assume $l=x_0$. Then $\frac{\partial f}{\partial x_i}=0$ for $i>0$ implies
that $f$ is a multiple of $x_0^d$. This proves the result set-theoretically.
The proof can be concluded by a infinitesimal computation.
If $f=x^d$, then it is easy to check that $\ker C_{x}=H^0(I_x(1))$,
and $(\mathrm{Im}C_f)^{\perp}=H^0(I_x(d-1))$.

The conormal space of  the determinantal locus at $[x]$ 
is given by the image of the  natural map
$$H^0(I_x(1))\otimes H^0(I_x(d-1))\to H^0(I_{x^2}(d)).$$
It is easy to check that this map is surjective, and from Proposition \ref{lasker}
the result follows.
\end{proof}
The quadratic equations which define the Veronese variety can be considered as a $SL(V)$-module inside
$S^2(S^d V)$.

We have the decomposition 
$$S^2(S^d V)=\oplus_{i=0}^{\lfloor\frac d2\rfloor} S^{2d-2i,2i}$$
so that the quadratic part of the ideal $I$ of the Veronese variety is
$$I_2=\oplus_{i=1}^{\lfloor\frac d2\rfloor} S^{2d-2i,2i},$$
$I_2$ corresponds indeed to the $2\times 2$ minors of $C_f$ in (\ref{cf}).
A stronger and nontrivial result is true.
\begin{theorem}\label{kostantveronese}
The ideal $I$ of the Veronese variety is generated by its quadratic part $I_2$.
\end{theorem}

This Theorem is a special case of result, due to Kostant, holding for any rational homogeneous variety $G/P$.
For a proof in the setting of representation theory see \cite{Land} Theorem 16.2.2.6,
or \cite{Pr} chap. 10 \S 6.6. For a somewhat different approach, generalized to flag varieties,
see \cite{Wey} Prop. 3.1.8.

The composition of two symmetric powers like
$S^k(S^d V)$ is quite hard to be computed.
The formula solving this problem in the case $\dim V=2$
is due to Cayley and Sylvester and it is one of the most beautiful achievements
of XIX century invariant theory. We will review it in \S \ref{cayleysylvestercomputation}.

For $d=3$ we have
$$\begin{array}{ccccc}S^2(S^3 V))&=&S^6 V&\oplus &S^{4,2}V\\
S^2\yng(3)&=&\yng(6)&\oplus&\begin{matrix}\yng(4,2)\end{matrix}\end{array}$$

and the equations consist of the irreducible module
$S^{4,2}V$.

This is given by the $2\times 2$-minors of the matrix
$$\left[\begin{array}{ccc}
x_0&x_1&x_2\\
x_1&x_2&x_3
\end{array}\right]$$

In classical notation, there are ``dual'' variables $y_i$
and the single Hessian covariant
\begin{equation}\label{Hexample}H=\left|\begin{array}{ccc}
x_0&x_1&x_2\\
x_1&x_2&x_3\\
y_0^2&y_0y_1&y_1^2
\end{array}\right|=0\end{equation}

For $d=4$ we have
$$\begin{array}{ccccccc}S^2(S^4 V))&=&S^8 V&\oplus& S^{6,2}V&\oplus& S^{4,4}V\\
\\
S^2\yng(6)&=&\yng(8)&\oplus&\begin{matrix}\yng(6,2)\end{matrix}&\oplus &\begin{matrix}\yng(4,4)\end{matrix}\end{array}$$
and the equations consist of the sum of the two last modules
$S^{6,2}V\oplus S^{4,4}V$.
The module structure is not evident at a first glance from the minors
and it is commonly not considered when these equations are encountered.
There is a single invariant quadric, which
is called the {\it equianharmonic} quadric and it is classically denoted by $I$.
It corresponds to $4$-ples which are apolar to themselves
(we will describe apolarity in \S \ref{apolaritysection}).

We have, given a binary quartic $4$
$$f=\sum_{i=0}^4{4\choose i}f_ix^{4-i}y^i$$ the expression
\begin{equation}\label{firstI}I=f_0f_4-4f_1f_3+3f_2^2\end{equation}

This expression gives the most convenient way to check if a given $f$ is anharmonic.
Note that from this expression the invariance is not at all evident.

We will prove in Theorem \ref{IJgenerate} that the ring of invariants for a binary quartic is generated by $I$ and another invariant $J$,
which has a simpler geometric construction. It is the equation of the $2$nd secant variety
$\sigma_2(v_4(\P^1))$, that is, it is the equation of the locus spanned by the secant lines at the rational normal quartic curve $v_4(\P^1)$.

We get

\begin{equation}\label{firstJ}J=
\det\left[\begin{array}{ccc}
a_0&a_1&a_2\\
a_1&a_2&a_3\\
a_2&a_3&a_4\end{array}\right]\end{equation}

Indeed the above matrix has rank $1$ on $v_4(\P^1)$, hence it has rank $\le 2$ on any secant line.

The weight of a monomial $f_0^{\nu_0}\ldots f_4^{\nu_4}$
is by definition $\sum_{i=0}^4i\nu_i$.
A (homogeneous) polynomial is called isobaric
if all its monomials have the same weight.
The monomials $f_0f_4$, $f_1f_3$, $f_2^2$ are all the  
monomials of degree $4$ and weight $2$, and we will check (see the Proposition \ref{torus}) 
that they are the only ones that can appear in (\ref{firstI}).

But why the coefficients in the expression (\ref{firstI}) have to be proportional at
 $(1,-4,3)$ ? During these lectures, we will answer 
three times to this question, respectively in \S \ref{apolaritysection}, \S \ref{liealgebraaction}, in \S \ref{reynoldsection}.

These answers follow different approaches that are useful ways to look at invariants.

\begin{remark}\label{dualsecvero}
The dual variety to the Veronese variety is the discriminant
hypersurface of degree $(n+1)(d-1)^n$ in $\P S^dV$, parametrizing all singular
hypersurfaces of degree $d$ in $\P V$.
More generally, the dual variety to the $k$-secant variety (see \cite{Land} 5.1) to the $d$-Veronese variety
(which is denoted by $\sigma_k(v_2(\P^n))$) consists of all hypersurfaces of degree $d$ in $\P V$ with at least $k$ double points.
This follows by Terracini Lemma (\cite{Land} 5.3).

For example the dual to $\sigma_k(v_2(\P^n))$ (symmetric matrices of rank $\le k$)
is given by $\sigma_{n+1-k}(v_2(\P^n))$ (symmetric matrices of rank $\le n+1-k$).

The dual to $\sigma_2(v_3(\P^2))$ is given by plane cubic curves with
two double points, that is by reducible cubics.
The dual to $\sigma_3(v_3(\P^2))$ is given by cubics with
three double points, that is by triangles. This is called a split variety,
and we will consider it in next subsection.
\end{remark}

\subsection{The split variety and its equations}

The split variety in $\P S^dV$ consists of all polynomials which split as a product of linear factors.
This subsection is inspired by \S 8.6 in \cite{Land}, where the split variety is called Chow variety.
The first nontrivial example is the variety of ``triangles'' in $\P S^3\CC^3$,
which is a $6$-fold of degree $15$.

We consider the natural map

$$S^kS^dV\to S^dS^kV$$ constructed by dividing
$V^{dk}$ represented by a $d\times k$ rectangle, first by rows (in the source)
and then first by columns (in the target).

The above map takes $(x_1^d\ldots x_k^d)$ to $(x_1\ldots x_k)^d$,
so it is nonzero on the coordinate ring of the split variety.

\begin{theorem}[Brion\cite{Brfaulkes}]
The kernel of the above map gives the degree $k$ part of
the ideal of the split variety of $d$-forms on $\P V$.
\end{theorem}

Note that in case $\dim V=2$ we have that all $d$-forms split and indeed
the previous map is an isomorphism (Hermite reciprocity).

\begin{example}

$S^2S^3\CC^3=S^{6}\CC^3\oplus S^{4,2}\CC^3$,

$S^3S^2\CC^3=S^{6}\CC^3\oplus S^{4,2}\CC^3\oplus S^{2,2,2}\CC^3$.

Indeed conics which split in two lines have a single invariant in degree $3$,
which is the determinant of the symmetric matrix defining them,
while cubics which split in triangles have no equations in degree two.
\end{example}

Even the case $k=d$ is interesting,
the natural map

$$S^dS^dV\to S^dS^dV$$
obtained by reshuffling between rows and columns, turns out to be a isomorphism
for $d\le 4$, but is is degenerate for $d=5$. When
$d=5$ get ${9\choose 5}=126$ and ${{126+4}\choose 5}=286,243,776$,
so the question corresponds to the rank computation for a square matrix of this size
and it is already a computational challenge.
It has been performed by M\"uller and Neunh\"offer in \cite{MN}.
It should be interesting to understand theoretically this phenomenon.

So there are equations of degree $5$ for the split variety
of ``pentahedra" in $\P ^4$.

It is interesting the split variety of triangles in the plane,
which has quartic equations. These equations correspond
to the proportionality between a cubic form $f\in S^3\CC^3$ and its Hessian $H(f)$.

\begin{remark}\label{plueckerquadrics} The next interesting variety for invariant theory is certainly the Grassmannian.
We give for granted its description and the fact that its ideal is generated by the Pl\"ucker quadrics. 
For a proof, like in the case of Veronese variety, we may refer again to  \cite{Land} Theorem 16.2.2.6.
Let just remind the shape of the Pl\"ucker quadrics.
The coordinates in the Pl\"ucker embedding of the Grassmannian of $k+1$
linear subspaces of $V$ are indexed by sequences
$[i_0\ldots i_k]$ where $0\le i_0<i_1<\ldots <i_k\le n$.
Fix a subset of $k+2$ elements
$i_0\ldots i_{k+1}$ and a set of $k$ elements $j_0\ldots j_{k-1}$.
Then the Pl\"ucker relations are
$$\sum_{s=0}^{k+1}(-1)^j[i_0\ldots \hat{i_s}\ldots i_{k+1}]
[i_sj_0 \ldots j_{k-1}]=0$$
which hold for  any  subsets of respectively
$k+2$ and $k$ elements.

\end{remark}

\section{Facts from representation theory}
\subsection{Basics about representations}\label{basicsrepresentations}
In this section we recall basic facts about representation theory, that can be found
for example in \cite{FH}. From a logical point of view, the facts in this section are the foundations
of the following sections. Anyway, the reader may find useful  reading the section \ref{sectinvariant}
for a better understanding of the use of representations, and going back when needed.

We will need to study representations of two basic groups, namely the finite symmetric group $\Sigma_d$ of
permutations on $d$ elements, and the group $SL(n+1)$ of $(n+1)\times(n+1)$ matrices having $\det=1$.
Both are reductive groups.

A representation of a group $G$ is a group morphism $G\rig{\rho} GL(W)$, where $W$ is a complex vector space.
We say that $W$ is a $G$-module, indeed we may write $g\cdot w=\rho(g)(w)$ for any $g\in G$, $w\in W$.
This notation underlines that $G$ acts over $W$.
This action satisfies the following properties, which follow immediately from the definitions

$g\cdot (w_1+w_2)=g\cdot w_1+g\cdot w_2$, $\forall g\in G$, $w_1, w_2\in W$,

$g\cdot \lambda w=\lambda g\cdot w$, $\forall g\in G$, $w\in W$, $\lambda\in\CC$,

$(g_1g_2)\cdot w=g_1\cdot(g_2\cdot w)$.

These properties resume the fact that a representation is a {\it linear} action.

Given two $G$-modules $V$, $W$, then $V\oplus W$ and $V\otimes W$ are $G$-modules in a natural way,
namely

$g\cdot (v+w):=(g\cdot v)+(g\cdot w)$,

$g\cdot (v\otimes w):=(g\cdot v)\otimes (g\cdot w)$.

The $m$-th symmetric power $S^mW$ is a $G$-module, satisfying
$g\cdot (v^m):=(g\cdot v)^m$.

A morphism between two $G$-modules $V$, $W$ is a linear map $f\colon V\to W$
which is $G$-equivariant, namely it satisfies $f(g\cdot v)=g\cdot f(v)$
$\forall g\in G$, $v\in V$.

Every  $G$-module $V$ corresponding to $\rho\colon G\to GL(V)$ has a character $\chi_{V}\colon G\to\CC$
defined as $\chi_{V}(g)=\textrm{trace\ }\rho(g)$.

It satisfies the property
$\chi_{V}(h^{-1}gh)=\chi_{V}(g)$
hence the characters are defined on conjugacy classes in $G$.

Moreover $\chi_{V\oplus W}=\chi_V+\chi_W$, $\chi_{V\otimes W}=\chi_V\chi_W$.

Characters are the main tool to work with representations and to identify them.

Every $G$-module has a $G$-invariant submodule
$$V^G=\{v\in V| g\cdot v=v\quad\forall g\in G\}.$$
In other words $G$ acts trivially over $V^G$. The character of the trivial representation
of dimension $r$ is constant, equal to $r$ on every conjugacy classes.

For any $G$-module $V$, $G$ acts on the graded ring $\CC[V]=\oplus_{m=0}^{\infty}S^mV$.

Since the sum and the product of two invariant elements are again invariant, we have
the invariant subring $\CC[V]^G=\oplus_{m=0}^{\infty}\left(S^mV\right)^G$.

\subsection{Young diagrams and  symmetrizers}

A Young diagram denoted by $\lambda=(\lambda_1,\ldots \lambda_k)$,
where $\lambda_1\ge\lambda_2\ge\ldots $
consists of a collection of boxes ordered in consecutive rows,
where the $i$-th row has exactly $\lambda_i$ boxes. The number of boxes
in $\lambda$ is denoted by $|\lambda|$.

All Young diagrams with $d$ boxes correspond to the partitions of $d$,
namely $\lambda=(\lambda_1,\ldots \lambda_k)$ corresponds to $|\lambda|=\lambda_1+\ldots +\lambda_k$.

The following are the Young diagram corresponding to $(2,1,1)$ and $(4,4)$:

$\yng(2,1,1)$
\vskip 0.5cm

$\yng(4,4)$.

Any filling of $\lambda$ with numbers is called a {\it tableau}.

Just to fix a convention, for a given Young diagram, number consecutively the boxes like

$$\young(1234,5678).$$

Here we used all numbers from $1$ to $d$ to fill $d$ boxes.
More generally, a tableau can allow repetitions of numbers, as we will see in the sequel.

Let $\Sigma_d$ be the symmetric group of permutations over $d$
elements. Due to the filling, we can consider the elements of $\Sigma_d$
as permuting the boxes.
Let $R_{\lambda}\subseteq \Sigma_d$ be the subgroup of permutations
preserving each row.

Let $C_{\lambda}\subseteq \Sigma_d$ be the subgroup of permutations
preserving each column.

\begin{definition}\label{defyoungsym}
The element \begin{equation}\label{youngsymm}c_{\lambda}=\sum_{\sigma\in R_{\lambda}}\sum_{\tau\in C_{\lambda}}
\epsilon(\tau)\sigma\tau\in\Sigma_d\end{equation}
is called the Young symmetrizer corresponding to $\lambda$.
\end{definition}
Note immediately that it depends on $\lambda$ but also on the filling of $\lambda$, see Remark \ref{diffilling}.

\subsection{Representations of finite groups and of $\Sigma_d$}

Let $G$ be a finite group.
\begin{proposition}
There are exactly $n(G)$ irreducible representations of $G$, where $n(G)$ is the number of conjugacy classes of $G$.
\end{proposition}
\begin{proof}\cite{FH} Prop. 2.30.\end{proof}

If $V_i$ $i=1,\ldots , n(G)$ are the irreducible representations of $G$ and
$g_j$, $j=1,\ldots, n(G)$ are representatives in the conjugacy classes of $G$ then the square matrix
$\chi_{V_i}(g_j)$ is called the {\it character table } of $G$.

\begin{example}\label{z2} The group $G=\Sigma_2$ has two elements $1$, $-1$, each one is a conjugacy class.
Besides the trivial representation $V_{2}$, we have another representation $V_{1,1}$
defined on its basis element $w$ by $1\cdot w=w$, $(-1)\cdot w=-w$.

The character table is
$$\begin{array}{c|cc}
&1&-1\\
\hline
V_{2}&1&1\\
V_{1,1}&1&-1\\
\end{array}$$

Note that $V_{2}^G=V_{2}$, $V_{1,1}^G=0$.

We will see in section \ref{reps6} the character table of $\Sigma_6$.
\end{example}

The proof of the following proposition is straightforward

\begin{proposition}
There is a $G$- surjective morphism $R\colon V\to V^{G}$
which is defined as $R(v)=\frac{1}{|G|}\sum_{g\in G}g\cdot v$ which satisfies
$R(v)=v$ $\forall v\in V^{G}$.
\end{proposition}

The main theorem on finite groups is the following

\begin{theorem}
Let $G$ be a finite group.

\begin{itemize}
\item{(i)} Given any $G$-module $V$, there are uniquely determined nonnegative $a_i$ for $i=1,\ldots n(G)$
such that $V=\oplus V_i^{\oplus a_i}$.

\item{(ii)} The $a_i$ are the unique solution to the square linear system
$\chi_V(g_j)=\sum_ia_i\chi_{V_i}(g_j), \quad j=1,\ldots, n(G)$.
In particular, from the character $\chi_V$, it can be computed the dimension of the invariant part $V^G$.

\end{itemize}
\end{theorem}

\begin{proof}
(i) is \cite{FH}, Prop. 1.8. 
(ii) follows from \cite{FH} Coroll. 2.14, Prop. 2.30.
\end{proof}

In the case of the symmetric group  $G=\Sigma_d$ 
its irreducible representations  are in bijective correspondence with the conjugacy classes of $\Sigma_d$,
which correspond to the cycle structures of permutations and so they are  described by partitions $\lambda$ of $d$.

Recall that for any finite group $G$, the group algebra 
$\CC G$ is defined in the following way. The underlying vector space of dimension $|G|$ has a basis $e_g$
corresponding to the elements $g\in G$ and the algebra structure is defined by the rule
$e_g\cdot e_h=e_{gh}$ .

$\Sigma_d$ acts on the vector space $\CC \Sigma_d$ by
$\sigma\cdot e_{\alpha}=e_{\alpha\sigma}$,
which extends by linearity to $\sigma\colon\CC \Sigma_d\to\CC \Sigma_d$.

\begin{definition}\label{defspecht}
The Young symmetrizer $c_{\lambda}$ defined in (\ref{youngsymm}) 
acts by right multiplication on $\CC\Sigma_d$, its image is a $\Sigma_d$-module that we denote by $V_{\lambda}$.
\end{definition}

We remark that the notations we have chosen in Example \ref{z2} are coherent with this definition.

Let $T$ be a tableau corresponding to a filling of the Young diagram $\lambda$ with $d$ boxes with the numbers
 $\{1,\ldots, d\}$. A tableau is called {\it standard} if the filling is strictly increasing on rows and columns.
A tableau corresponds to a permutation $\sigma_T\in\Sigma_d\subset\CC\Sigma_d$.
By definition $V_{\lambda}$ is spanned by $c_{\lambda}\sigma_T$ for any tableau $T$.
\begin{remark}\label{diffilling} $\Sigma_d$ acts by conjugation over $\CC\Sigma_d$.
Any conjugate $\sigma^{-1}c_{\lambda}\sigma$ acts by right multiplication on $\CC\Sigma_d$,
its image is isomorphic to $V_{\lambda}$, although they may be embedded in different ways.
These different copies can be obtained, in equivalent way, starting by a different tableau in Def. \ref{defyoungsym},
see \cite{Pr} chap. 9, remark 2.2.5.
\end{remark}

\begin{theorem}[\cite{Pr} chap. 9, \S 2.4 and \S 9.2]
\
\begin{itemize}

\item{(i)} $V_{\lambda}$ is an irreducible representation of $\Sigma_d$.

\item{(ii)} All irreducible representations of $\Sigma_d$ are isomorphic to $V_{\lambda}$ for some Young diagram $\lambda$.

\item{(iii)} If $\lambda$ and $\mu$ are different partitions, then $c_{\lambda}c_{\mu}=0$
and $c_{\lambda}^2$ is a scalar multiple of $c_{\lambda}$.

\item{(iv)} A basis of $V_{\lambda}$ is given by $c_{\lambda}\sigma_T$ where $\sigma_T\in\Sigma_d\subset\CC\Sigma_d$
corresponds to standard tableau  $T$.
In particular $\dim V_{\lambda}$ is equal to the number of standard tableaux on $\lambda$.
\end{itemize}
\end{theorem}

\begin{theorem}[\cite{FH} Prop. 3.29]
We have the isomorphism of algebras
\begin{equation}\label{groupsim}\CC\Sigma_d=\oplus_{\{\lambda| |\lambda|=d\}}\mathrm{End}(V_{\lambda}).
\end{equation}

According to the isomorphism (\ref{groupsim}), any $c_{\mu}$ is a diagonalizable endomorphism of rank one (and nonzero trace) in $\mathrm{End}(V_{\mu})$
and it is zero in  $\mathrm{End}(V_{\lambda})$ for $\lambda\neq\mu$.

\end{theorem}

Note the two extreme cases,
when $\lambda=d$ we get the trivial one dimensional representation of $\Sigma_d$,
while if $\lambda=(\underbrace{1,\ldots ,1}_d)=1^d$ (for short)
we get the  one dimensional representation given by sign, that is
the action on a generator $e$ is given by $\sigma\cdot e=\epsilon(\sigma)e$.

\subsection{Representations of $GL(n+1)$ and $SL(n+1)$, Schur functors}

\begin{theorem}
Let $f\colon GL(n+1)\to\CC$ be a polynomial function invariant by conjugation, that is
$f(G^{-1}AG)=f(A)$ for every $G, A\in GL(n+1)$.

Then $f$ is a polynomial symmetric function of the eigenvalues of $A$.
\end{theorem}
\begin{proof}
Let $D(d_1,\ldots, d_{n+1})$ be the diagonal matrix having $d_i$ on the diagonal. If $\tau\in\Sigma_{n+1}$ and $M_{\tau}$ is the
corresponding permutation matrix, then $$M_{\tau}^{-1}D(d_1,\ldots, d_{n+1})M_{\tau}=D(d_{\tau(1)},\ldots, d_{\tau(n+1)}).$$
It follows that $f(D(d_1,\ldots, d_{n+1}))$ is a polynomial symmetric function of $d_1,\ldots, d_{n+1}$.
By the Main Theorem on symmetric polynomials (see \cite{Stu} Theor. 1.1.1) there is a polynomial $g\in \CC[x_1,\ldots x_{n+1}]$
such that $$f(D(d_1,\ldots, d_{n+1}))=g(\sigma_1(d_1,\ldots, d_{n+1}),\ldots, \sigma_{n+1}(d_1,\ldots, d_{n+1})) ,$$
where $\sigma_i$ is the $i$-th elementary symmetric polynomials. 
For any matrix $A$,  denote $\det(A-tI)=\sum_{i=0}^{n+1}t^i(-1)^ic_{n+1-i}(A)$. Note that $\sigma_i(d_1,\ldots, d_{n+1})=c_i(D(d_1,\ldots, d_{n+1}))$,
and that $c_i(A)$ is the $i$-th elementary symmetric function of the eigenvalues of $A$.
We have proved that $f(A)=g(c_1(A),\ldots, c_{n+1}(A))$
for every diagonal matrix $A$. Both sides are invariant by conjugation,
then the equality is satisfied by any diagonalizable matrix. Since diagonalizable matrix form a dense subset,
it follows that $f(A)=g(c_1(A),\ldots, c_{n+1}(A))$ for any matrix $A$,
\end{proof}

The same argument works for polynomial functions  $f\colon SL(n+1)\to\CC$ which are invariant by conjugation.
See \cite{Vac} for an extension to invariants of several matrices.

\begin{corollary}
Characters of $GL(n+1)$ (and of $SL(n+1)$) are symmetric functions of the eigenvalues.
\end{corollary}

Let $V=\CC^{n+1}$, the group $\Sigma_d$ acts on $\otimes^dV$ by permuting the factors.

We define the Schur projection
$c_{\lambda}\colon \otimes^d V\to\otimes^d V$
from the Young symmetrizer $c_{\lambda}$ defined in (\ref{youngsymm}).

We fill the Young diagram with numbers from $1$ to $n+1$, allowing repetitions.
 After a basis of $V$ has been fixed,
any such tableau $T$ gives a basis vector $v_T\in\otimes^d V$.

A tableau is called {\it semistandard} if it has nondecreasing rows and strictly increasing columns.

\begin{theorem}\label{fillglv}\

\begin{itemize}
\item{} The image of $c_{\lambda}$ is a irreducible $GL(n+1)$-module,
which is nonzero if and only if the number of rows is $\le n+1$, denoted by $S^{\lambda}V$.

\item{} All irreducible $GL(n+1)$-modules are isomorphic to $S^{\lambda}V$
for some Young diagram $\lambda$ with number of rows $\le n+1$.

\item{} All $GL(n+1)$-modules is a sum of irreducible ones.

\item{} The images $c_{\lambda}(v_T)$ where $T$ is a semistandard tableau give
a basis of $S^{\lambda}V$. 

\end{itemize}
\end{theorem}
\begin{proof}\cite{FH} Prop. 15.47. 
\end{proof}

 A particular case, very useful in the sequel, is when
$\lambda$ consists of $g$ columns of length $n+1$. This happens if and only if
$S^{\lambda}V$ is one dimensional. 

The theory of $SL(n+1)$-representations is very similar.
The basic fact is that if $S^{\lambda_1,\lambda_2\ldots, \lambda_{n+1}}V$ and
$S^{\lambda_1-\lambda_{n+1},\lambda_2-\lambda_{n+1}\ldots, \lambda_n-\lambda_{n+1},0}V$
are isomorphic as $SL(n+1)$-modules, indeed
$$S^{\lambda_1,\lambda_2\ldots, \lambda_{n+1}}V\simeq \wedge^{n+1}V\otimes S^{\lambda_1-\lambda_{n+1},\lambda_2-\lambda_{n+1}\ldots, \lambda_n-\lambda_{n+1},0}V. $$ 

In other words, all columns of length $n+1$ correspond to the one dimensional determinantal representation,
which is trivial as $SL(n+1)$-module. Removing these columns, we get another Young diagram with 
the number of rows $\le n$.

\begin{theorem}\
\begin{itemize}
\item{} All irreducible $SL(n+1)$-modules are isomorphic to $S^{\lambda}V$
for some Young diagram $\lambda$ with number of rows $\le n$.

\item{} All $SL(n+1)$-modules is a sum of irreducible ones.

\end{itemize}
\end{theorem}

\begin{proof} \cite{FH}Prop. 15.15, Theor. 14.18 , \cite{Pr} chap. 9 \S 8.1
\end{proof}

\begin{remark} The construction $W\mapsto S^{\lambda}W$ is indeed functorial,
in the sense that a linear map $W_1\to W_2$ induces a linear map $S^{\lambda}W_1\to S^{\lambda}W_2$
with functorial properties, see \cite{Pr} chap. 9 \S 7.1. See also \cite{Wey} chap. 2.
\end{remark}

\begin{definition}\label{defschurpoly} The character of $S^{\lambda}V$ is a symmetric polynomial
$s_{\lambda}$ which is called a Schur polynomial.\end{definition}

There is  a way to write $s_{\lambda}$ as the quotient of two Vandermonde determinants
(\cite{Man0} Prop. 1.2.1), and more efficient ways to write down explicitly $s_{\lambda}$
which put in evidence the role of tableau (\cite{Man0} Theor. 1.4.1).

If $\lambda=d$, then $s_{\lambda}(t_1,\ldots t_n)$ is the sum of all possible monomials
of degree $d$ in $t_1,\ldots, t_n$. If $\lambda={\underbrace{1,\ldots, 1}_i}$, then $s_{\lambda}(t_1,\ldots t_n)$ is the $i$-th elementary symmetric polynomial  in $t_1,\ldots, t_n$. 

$SL(n+1)$ contains the torus $\left(\CC^*\right)^{n}$ of diagonal matrices
$$T=\{D\left(t_1,\ldots, t_{n},\frac{1}{(t_1\cdots t_{n})}\right)|t_i\neq 0\}.$$

Given $(a_1,\ldots , a_n)\in\ZZ^n$ we have the one dimensional (algebraic) representation
$\rho\colon T\to\CC^*$ defined by
$$\rho\left(D\left(t_1,\ldots, t_{n},\frac{1}{(t_1\cdots t_{n})}\right)\right):=t_1^{a_1}\cdots t_n^{a_n}.$$

We denote it by $V_{a_1,\ldots , a_n}$.

\begin{proposition}\label{torusrep}\

\begin{itemize}
\item{} Every irreducible (algebraic) representation of $T$ is isomorphic to $V_{a_1,\ldots , a_n}$ for some $(a_1,\ldots , a_n)\in\ZZ^n$.
\item{} Every (algebraic) $T$-module is isomorphic to the direct sum of irreducible representations.
\end{itemize}\end{proposition}

\begin{proof}
\cite{Pr} chap. 7, \S 3.3 
\end{proof}

\begin{theorem}[Cauchy identity]\label{cauchyidentity}
$$S^p(V\otimes W)=\oplus_{\lambda}S^{\lambda}V\otimes S^\lambda W$$
where the sum is extended to all partitions $\lambda$ of $p$.
\end{theorem}
\begin{proof}By using characters, the proof reduces to a nontrivial identity on symmetric functions, see \cite{Pr} (6.3.2) or \cite{Man0} 1.4.2.
\end{proof}

\subsection{The Lie algebra  $\mathfrak{sl}(n+1)$ and the weight structure of its representations}
\label{slsection}
We denote by $\mathfrak{sl}(n+1)$ the Lie algebra of $SL(n+1)$. It corresponds
to the traceless matrices of size $(n+1)$, where the bracket is $[A,B]=AB-BA$ $\forall A, B\in \mathfrak{sl}(n+1)$.
The tangent space at the identity of $SL(n+1)$ is naturally isomorphic to  $\mathfrak{sl}(n+1)$.

A representation of the Lie algebra $\mathfrak{sl}(n+1)$ is a Lie algebra morphism
$\mathfrak{sl}(n+1)\to \mathfrak{sl}(W)$.
The derivative (computed at the identity) of a group representation
$SL(n+1)\to SL(W)$
is a representation of the Lie algebra $\mathfrak{sl}(n+1)$.

Since $SL(n+1)$ is simply connected, there is a natural bijective correspondence between
$SL(n+1)$-modules and $\mathfrak{sl}(n+1)$-modules, in the sense
that every Lie algebra morphism
$\mathfrak{sl}(n+1)\to \mathfrak{sl}(W)$ is the derivative of a unique
group morphism
$SL(n+1)\to SL(W)$.
In particular all  $\mathfrak{sl}(n+1)$-modules are direct sum of irreducible ones.

This definition behaves in a different way when applied on  direct sums and tensor products.
Let $\mathfrak{g}$ be a Lie algebra. If  $V$ and $W$ are two $\mathfrak{g}$-modules
then $V\oplus W$ and $V\otimes W$ are $G$-modules in a natural way,
namely

$g\cdot (v+w):=(g\cdot v)+(g\cdot w)$,

$g\cdot (v\otimes w):=(g\cdot v)\otimes (w)+(v)\otimes (g\cdot w)$.

The $m$-th symmetric power $S^mW$ is a $G$-module, satisfying
\begin{equation}\label{symlie}g\cdot (v^m):=m(g\cdot v)v^{m-1}.\end{equation}

A morphism between two $\mathfrak{g}$-modules $V$, $W$ is a linear map $f\colon V\to W$
which is $\mathfrak{g}$-equivariant, namely it satisfies $f(g\cdot v)=g\cdot f(v)$
$\forall g\in \mathfrak{g}$, $v\in V$.

Every $\mathfrak{g}$-module has a $\mathfrak{g}$-invariant submodule
$$V^{\mathfrak{g}}=\{v\in V| g\cdot v=0\quad \forall g\in \mathfrak{g}\}.$$

The adjoint representation of $SL(n+1)$ is the group morphism

$$\begin{array}{ccc}SL(n+1)&\to &GL(\mathfrak{sl}(n+1))\\
g&\mapsto &(h\mapsto g^{-1}hg).\end{array}$$
Its derivative is the Lie algebra morphism
$$\begin{array}{ccc}\mathfrak{sl}(n+1)&\to &End(\mathfrak{sl}(n+1))\\
g&\mapsto &(h\mapsto [h,g]).\end{array}$$

The diagonal matrices $H\subset \mathfrak{sl}(n+1)=\mathfrak{g}$ make a Lie subalgebra
which can be identified with $\textrm{Lie\ }(T)$. It is abelian, that is $[H,H]=0$, and it is called a Cartan subalgebra.
Write generators as
$$H=\{D\left(t_1,\ldots, t_{n+1}\right)|\sum_it_i= 0\}.$$
A basis of $H$ is given by $\{D\left(1,0,\ldots,0, -1\right), 
D\left(0,1,0,\ldots, 0,-1\right), \ldots\}$.
In dual coordinates, we have a basis $h_i=(0,\ldots,0,\underset{i}{\underset{\uparrow}{1}},0,\ldots, 0,-1)\in H^{\vee}$ for $i=1,\ldots, n$.
We can consider $h_i$ as Lie algebra morphisms
$h_i\colon H\to\CC$.

Representations of $\mathfrak{sl}(n+1)$, when restricted to $H$,
 satisfy the analogous to Proposition \ref{torusrep}.

\begin{proposition}\label{lietorusrep}
 Let $W$ be a Lie algebra representation of ${sl}(n+1)$.
When restricting the representation to $H$,
it splits in the direct sum of irreducible representations,
each one isomorphic to an integral combination $\sum_{i=1}^na_ih_i$ with $a_i\in\ZZ$.
These representations are the derivative of the representations
$V_{a_1,\ldots , a_n}$ of Prop. \ref{torusrep}.
\end{proposition}

The strictly upper triangular matrices $$N^+:=\{g\in \mathfrak{sl}(n+1)| g_{ij}=0\textrm{\ for\ }i<j\}$$
make a nilpotent subalgebra, in the sense that the descending chain
$$\mathfrak{g}^+\supset[\mathfrak{g}^+,\mathfrak{g}^+]\supset[[\mathfrak{g}^+,\mathfrak{g}^+],\mathfrak{g}^+]\supset\ldots$$
terminates to zero.
It holds $[N^+,H]\subset H$, which means that
$N^+$ is an invariant subspace for the adjoint representation restricted to $H$, indeed it splits as the sum of
one dimensional representations, which are spanned by the elementary matrices $e_{ij}$ which are zero unless one upper triangular entry
which is $1$.

More precisely, there are certain $\alpha=\sum_{i=1}^na_ih_i\in H^*$ as in Prop. \ref{lietorusrep}
and  $n_{\alpha}\in N^+$ such that
 \begin{equation}\label{root}[h, n_\alpha]=\alpha(h)n_\alpha\qquad\forall h\in H.
\end{equation} Such $\alpha$ are called  (positive) roots,
and $N^+$ has a basis of eigenvectors $n_{\alpha}$.

In the same way we can define the subalgebra of strictly lower triangular matrices $N^-$,
which has a similar basis of eigenvectors. The corresponding roots are called negative.

\begin{theorem}[Weight Decomposition]\label{weightdecomposition}
Let $W$ be a $\mathfrak{sl}(n+1)$-module. $W$ is also a $H$-module, and it splits as the sum of $H$-representations $W_{(\lambda_i)}$ where $\lambda_i\in H^{\vee}$ is called a weight and satisfies the following property.
 $$W_{(\lambda_i)}=\{w\in W| h\cdot w=\lambda_i(h)w, \forall h\in H\}.$$ 
The elements in the weight space $W_{(\lambda_i)}$ are called $H$-eigenvectors with weight $\lambda_i$.

\end{theorem}

\begin{theorem}
Let $W$ be a irreducible $\mathfrak{sl}(n+1)$-module. Then there is a unique (up to scalar multiples) vector $w\in W$ satisfying
$N^-\cdot w=0$. $w$ is called a maximal vector and it is an  $H$-eigenvector.

The representation $W$ is spanned by repeated applications of elements $g\in N^+$ to $w$.
More precisely, $e_{i_1j_1}\ldots e_{i_sj_s}w$ span $W$ for convenient $e_{i_kj_k}\in N^+$. 

\end{theorem}

The basic principle is that the generators of $N^+$ (as well as the generators of $N^-$) act on the weight decomposition of $W$.

\begin{proposition}\label{rootaction}
Let $n_{\alpha}\in N^+$ be an eigenvector with eigenvalue the root $\alpha$
as in (\ref{root}).
Let $\lambda\in H^{\vee}$ be a weight of a representation $W$.
Then
$$n_{\alpha}( W_{(\lambda)} )\subseteq  W_{(\alpha+\lambda)}.$$
\end{proposition}
\begin{proof}
Let $w\in W_{\lambda}$. For any $h\in H$ we have
$$h(n_\alpha(w))=[h,n_\alpha](w)+n_{\alpha}(h(w))=\alpha(h)n_{\alpha}(w)+n_{\alpha}\lambda(h)(w)=
\left(\alpha(h)+\lambda(h)\right)(n_{\alpha}(w)).$$
This proves that $n_{\alpha}(w)\in W_{(\alpha+\lambda)}$ as we wanted. \end{proof}

It is instructive to draw pictures of eigenspaces decomposition,
denoting any weight space as a vertex,
identifying the action of elements $n_\alpha$ as in Proposition \ref{rootaction}
with arrows from $W_{(\lambda)}$ to $W_{(\alpha+\lambda)}$.

We begin with  $\mathfrak{sl}(2)$, which has dimension three, spanned by
$h=\begin{pmatrix}1&0\\
0&-1\end{pmatrix}$, $x=\begin{pmatrix}0&1\\
0&0\end{pmatrix}$, $y=\begin{pmatrix}0&0\\
1&0\end{pmatrix}$,

satisfying
$$[h,x]=2x,\quad [h,y]=-2y,\quad[x,y]=h.$$
Each irreducible $\mathfrak{sl}(2)$-module is isomorphic to $S^m\CC^2$ for an integer $m\in\ZZ_{\ge 0}$.
If $v$ is the maximal vector,
the space $S^m\CC^2$ is isomorphic to $\oplus_{i=0}^m<x^i\cdot v>$ for $i=0,\ldots m$ as in the following picture
$$x^m\cdot v\lef{x}x^{m-1}\cdot v\lef{x}\ldots x\cdot v\lef{x}v.$$

If we want to emphasize the dimensions of the weight spaces, we just write

$$\xymatrix{1&1\ar[l]&1\ar[l]&\ldots\ar[l]&1\ar[l]&{\bf 1}\ar[l]}$$

The exterior skeleton of $S^{p}\CC^2$ has the following pattern (the maximal vector is marked)
$$\xymatrix{\bullet&&&&&&\odot\ar[llllll]_p}$$

Note that $x^{m+1}v=0$. We have $h\cdot v=m v$, and in general
$h\cdot (x^i\cdot v)=(2m-2i)(x^i\cdot v)$. The natural construction in the setting of invariant theory is the following.
Each $H$-eigenspace is generated by the monomial $s^{m-i}t^i$. 
The monomial $s^m$ corresponds to the maximal vector.
According to (\ref{symlie}) we compute that 
 $(x^i\cdot s^m)$ is a scalar multiple of the monomial $s^{m-i}t^i$.
\vskip 0.5cm

Each irreducible $\mathfrak{sl}(3)$-module is isomorphic to $S^{a,b}\CC^3$ for a pair of integers $a, b\in\ZZ_{\ge 0}$
with $a\ge b$.

We have the following pictures,
which describe the general pattern that ``weights increase by one on hexagons and remain constant on triangles'',
see \cite{FH} chapter 6. The maximal vectors are marked in bold.

In  these pictures 
\begin{equation}\label{dira1}\textrm{the arrow\ }\xymatrix@R-1.5em{\bullet\\
&\bullet\ar[ul]\\}\textrm{\ corresponds to 
the action of\ }
A_1=\begin{bmatrix}0&1&0\\
0&0&0\\
0&0&0\end{bmatrix}\end{equation}

\begin{equation}\label{dira2}\textrm{the arrow\ }\xymatrix@R-1.5em{\bullet\ar[dd]\\
\\
\bullet\\}\textrm{\ corresponds to 
the action of\ }
A_2=\begin{bmatrix}0&0&0\\
0&0&1\\
0&0&0\end{bmatrix}\end{equation}

$$\textrm{and the "intermediate'' arrow\ }\xymatrix@R-1.5em{&\bullet\ar[dl]\\
\bullet\\}\textrm{\ corresponds to 
the action of\ }
[A_1,A_2]=\begin{bmatrix}0&0&1\\
0&0&0\\
0&0&0\end{bmatrix}$$
\vskip 0.5cm

$$S^3\CC^3=S^{3,0}\CC^3\qquad\xymatrix@R-1.5em{1\ar[dd]\\
&{ 1}\ar[dd]\ar[dl]\ar[ul]\\
1\ar[dd]&&{ 1}\ar[dl]\ar[ul]\ar[dd]\\
&1\ar[dl]\ar[ul]\ar[dd]&&{\bf 1}\ar[dl]\ar[ul]\\
1\ar[dd]&&{ 1}\ar[dl]\ar[ul]\\
&{ 1}\ar[dl]\ar[ul]\\
1\\}$$

$$\left(S^3\CC^3\right)^{\vee}=S^{3,3}\CC^3\qquad\xymatrix@R-1.5em{&&&{\bf 1}\ar[dd]\ar[dl]\\
 &&1\ar[dd]\ar[dl]\\
&1\ar[dd]\ar[dl]&&1\ar[dd]\ar[dl]\ar[ul]\\
1&&1\ar[dd]\ar[dl]\ar[ul]\\
&1\ar[ul]&&1\ar[dd]\ar[dl]\ar[ul]\\
 &&1\ar[ul]\\
&&&1\ar[ul]\\
}$$

$$S^{2,1}\CC^3\qquad\xymatrix@R-1.5em{&1\ar[dl]\ar[dd]\\
1\ar[dd]&&{\bf 1}\ar[ul]\ar[dl]\ar[dd]\\
&2\ar[ul]\ar[dl]\ar[dd]\\
1&&1\ar[ul]\ar[dl]\\
&1\ar[ul]}$$

$$S^{5,1}\CC^3\qquad\xymatrix@R-1.5em{&1\ar[dl]\ar[dd]\\
1\ar[dd]&&1\ar[ul]\ar[dl]\ar[dd]\\
&2\ar[dd]\ar[ul]\ar[dl]&& 1\ar[ul]\ar[dl]\ar[dd]\\
1\ar[dd]&&2\ar[dd]\ar[ul]\ar[dl]&& 1\ar[ul]\ar[dl]\ar[dd]\\
&2\ar[dd]\ar[ul]\ar[dl]&&2\ar[ul]\ar[dl]\ar[dd]&& {\bf 1}\ar[ul]\ar[dl]\ar[dd]\\
1\ar[dd]&&2\ar[ul]\ar[dd]\ar[dl]&& 2\ar[ul]\ar[dl]\ar[dd]\\
&2\ar[dd]\ar[ul]\ar[dl]&&2\ar[ul]\ar[dl]\ar[dd]&& 1\ar[ul]\ar[dl]\\
1\ar[dd]&&2\ar[ul]\ar[dd]\ar[dl]&& 1\ar[ul]\ar[dl]\\
&2\ar[ul]\ar[dl]\ar[dd]&&1\ar[ul]\ar[dl]\\
1&&1\ar[ul]\ar[dl]\\
&1\ar[ul]}$$

$$S^{4,2}\CC^3\qquad\xymatrix@R-1.5em{&&1\ar[dl]\ar[dd]\\
&1\ar[dd]\ar[dl]&& 1\ar[ul]\ar[dl]\ar[dd]\\
1\ar[dd]&&2\ar[dd]\ar[ul]\ar[dl]&& {\bf 1}\ar[ul]\ar[dl]\ar[dd]\\
&2\ar[dd]\ar[ul]\ar[dl]&&2\ar[ul]\ar[dl]\ar[dd]\\
1\ar[dd]&&3\ar[ul]\ar[dd]\ar[dl]&& 1\ar[ul]\ar[dl]\ar[dd]\\
&2\ar[dd]\ar[ul]\ar[dl]&&2\ar[ul]\ar[dl]\ar[dd]\\
1&&2\ar[ul]\ar[dd]\ar[dl]&& 1\ar[ul]\ar[dl]\\
&1\ar[ul]&&1\ar[ul]\ar[dl]\\
&&1\ar[ul]}$$

The exterior skeleton of $S^{p,q}\CC^3$ has the following pattern (the maximal vector is marked differently)
$$S^{p,q}\CC^3\qquad
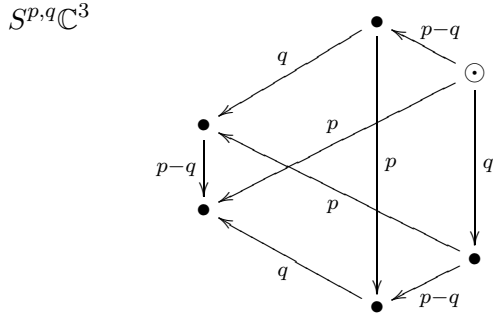
\xymatrix@R-1.5em{&&\bullet\ar[ddll]_{q}\ar[dddddd]^p\\
&&& {\odot}\ar[ul]_{p-q}\ar[dddlll]_p\ar[dddd]^q\\
\bullet\ar[dd]_{p-q}&&\\
&&&\\
\bullet&&\\
&&&\bullet\ar[dl]^{p-q}\ar[uuulll]^p\\
&&\bullet\ar[uull]^q}$$

\subsection{Schur-Weyl duality}

\begin{theorem}[Schur-Weyl duality]\label{schurweyltheorem}
There is a $\Sigma_d\times SL(V)$-decomposition
$$V^{\otimes k}=\oplus_{\lambda}V_{\lambda}\otimes S^{\lambda}V$$
where the sum is extended on all Young diagrams with $k$ boxes,
$V_{\lambda}$ has been defined in Def. \ref{defspecht} and $S^{\lambda}V$ has been defined in
Theorem \ref{fillglv}.
\end{theorem}
\begin{proof}\cite{Pr} chap. 9 (3.1.4).\end{proof}

Now fill the Young tableau with numbers 
from $1$ to $n+1$ in such a way that they are nondecreasing on rows
and strictly increasing on columns. For example we have the following

$\young(11,2)$, $\young(12,2)$, $\young(13,2)$, $\young(11,3)$,  
$\young(12,3)$, $\young(13,3)$, $\young(22,3)$, $\young(23,3)$.

Each filling describes a vector in $V^{\otimes d}$. The image of
these  vectors through $c_{\lambda}$ give a basis of $S^{\lambda}V$.
In other words, the elements $c_{\lambda}(v_T)$ where $T$ is any tableau span $S^{\lambda}V$.

Composing with a permutation of all the boxes,
we may find different isomorphic copies of the same
representation $V_{\lambda}$ inside $V^{\otimes d}$.

This construction is quite important and make visible that
$\mathrm{Im}c_{\lambda}$ defines just one copy of the representation
$S^{\lambda}V$ inside $V^{\otimes d}$, which is not intrinsic but it depends on the convention
we did in the definition of $c_{\lambda}$. For example by swapping in the definition \ref{youngsymm} of $c_{\lambda}$
the order of $R_{\lambda}$ and $C_{\lambda}$ we get another copy of $S^{\lambda}V$, in general skew with respect to the previous one.

But the main reason is that the order we chose in the $d$ copies of $V$ is arbitrary,
so acting with $\Sigma_d$ we can find other embeddings of $S^{\lambda }V$, all together spanning
$V_{\lambda}\otimes S^{\lambda}V$ which is canonically embedded in $V^{\otimes d}$,
and it is $\Sigma_d\times SL(V)$-equivariant.

For example in
in ${\CC^2}^{\otimes 4}$ there is, corresponding to $\lambda=(2,2)$
the invariant subspace $\CC^2_{2,2}\otimes S^{2,2}\CC^2$.

Note that $\dim \CC^2_{2,2}=2$, while $\dim S^{2,2}\CC^2=1$.

This $2$-dimensional space of invariants is spanned by the three functions that take

$$x_1\otimes x_2\otimes x_3\otimes x_4\in \left({\CC^2}^{\vee}\right)^{\otimes 4}$$
respectively in $(x_1\wedge x_2)(x_3\wedge x_4)$,
$(x_1\wedge x_3)(x_2\wedge x_4)$,
$(x_1\wedge x_4)(x_2\wedge x_3)$.

Note that we have the well known relation
 $$(x_1\wedge x_2)(x_3\wedge x_4)-(x_1\wedge x_3)(x_2\wedge x_4)+(x_1\wedge x_4)(x_2\wedge x_3)=0.$$
The image of the Schur symmetrizer $c_{2,2}$ is a scalar multiple of 
$f=(x_1\wedge x_4)(x_2\wedge x_3)+(x_1\wedge x_3)(x_2\wedge x_4)$ .
Note that applying the permutation $(12)\in\Sigma_4$, we get $(12)\cdot f=f$,
while applying the permutation $(13)\in\Sigma_4$ we get an independent $SL(2)$-invariant 
element $(13)\cdot f$,
which together with $f$ spans the $\Sigma_4\times SL(2)$-invariant subspace  $\CC^2_{2,2}\otimes S^{2,2}\CC^2$.

This example generalizes to the following

\begin{theorem}\label{schurweylapply}
If $\sigma\in\Sigma_d$, 
the elements $\sigma c_{\lambda}(v_T)$ where $T$ is any tableau,
span a $SL(n+1)$-module isomorphic to $S^{\lambda}V$, 
lying in the subspace $V_{\lambda}\otimes S^{\lambda}V$.

The whole subspace $V_{\lambda}\otimes S^{\lambda}V$ is spanned
by these copies of $S^{\lambda}V$, for any  $\sigma\in\Sigma_d$.

\end{theorem}

On the other hand, if $T$ is a fixed tableau such that $c_{\lambda}(v_T)\neq 0$,
the elements $\sigma c_{\lambda}(v_T)$ for  $\sigma\in\Sigma_d$ span
a $\Sigma_d$-module  isomorphic to $V_{\lambda}$ lying in the subspace $V_{\lambda}\otimes S^{\lambda}V$.

\section{Invariants of forms and representation theory.}\label{sectinvariant}
Again we denote $V=\CC^{n+1}$. The subject of this section is the action of $SL(V)$ over $\oplus_mS^m(S^dV)$.
The invariant subspace $S^m(S^dV)^{SL(V)}$ is, by definition, the space of invariants of degree $m$
for forms of degree $d$.
We have from \S \ref{slsection} the decomposition $\mathfrak{sl}(V)=H\oplus N^+\oplus N^-$ 
and we study separately the actions of $H$ and $N^+$.

\subsection{Invariance for the torus}
\label{invtorusection}

Denote the coefficients of a form $f\in S^dV$ of degree $d$ in $n+1$
variables in the following way

$$f=\sum_{i_0+\ldots +i_n=d}\frac{d!}{i_0!\ldots i_n!}f_{i_0\ldots i_n}
x_0^{i_0}\ldots x_n^{i_n}.$$

The space $S^mS^dV$ is spanned by monomials $f_{i_{0,1}\ldots i_{n,1}}\ldots f_{i_{0,m}\ldots i_{n,m}}$.

The {\it weight} of the monomial $f_{i_{0,1}\ldots i_{n,1}}\ldots f_{i_{0,m}\ldots i_{n,m}}$ is the vector 
$$\left(\sum_{j=1}^mi_{0,j},\sum_{j=1}^mi_{1,j},\ldots,\sum_{j=1}^mi_{n,j}\right)\in\ZZ_{\ge 0}^{n+1}.$$
As we will see in the proof of Proposition \ref{torus}, the subspaces of monomial of a given weight,
are exactly the $H$-eigenspaces for the action of $\mathfrak{sl}(n+1)$ on $S^mS^d\CC^{n+1}$
seen in Theorem \ref{weightdecomposition}.

A monomial is called {\it isobaric} if its weight has all equal entries (``democratic''). 
Consider the double sum
$\sum_{k=0}^n\sum_{j=1}^mi_{k,j}=\sum_{j=1}^md=md$. So the weight of an isobaric monomial in $S^mS^dV$
is $(\underbrace{\frac{md}{n+1},\ldots,\frac{md}{n+1}}_{n+1})$, in particular isobaric monomials exist if and only if $n+1$ divides $md$.

An polynomial $I\in S^mS^dV$ has degree $m$
in the coefficients $f_{i_0\ldots i_n}$.
\begin{proposition}\label{torus} $I\in S^mS^dV$ is invariant for the action of the torus of diagonal matrices
$(\CC^*)^{n}\subset SL(V)$ if an only if it is sum of isobaric monomials.
\end{proposition}

Note that it is enough to check if a monomial of given degree is isobaric 
for $n$ places in the $(n+1)$-dimensional weight vector. In particular for binary forms
it is enough to check the condition just for one place.

\begin{proof} Consider the diagonal matrix with entries
$(\frac{1}{t_1\ldots t_n},t_1,\ldots, t_n)$
which acts on $f_{i_0\ldots i_n}$ by multiplying for
$(t_1\ldots t_n)^{-i_0}t_1^{i_1}\ldots t_n^{i_n}.$

Acting on the monomial
$$f_{i_{0,1}\ldots i_{n,1}}\ldots f_{i_{0,m}\ldots i_{n,m}},$$
we multiply it for
$(t_1\ldots t_n)^{-\sum_{j}i_{0,j}}t_1^{\sum_ji_{1,j}}\ldots t_n^{\sum_ji_{n,j}}$
and this is equal to $1$ if and only if $\sum_ji_{k,j}$ does not depend on $k$.\end{proof}

\begin{corollary}\label{n+1divides}
A necessary condition
for the existence
of a nonzero   $I\in S^mS^dV$ 
which is invariant for $SL(V)$  is that
$n+1$ divides $md$. 
\end{corollary}
\begin{proof} If $I$ is $SL(V)$- invariant then it is also invariant for the torus of diagonal matrices.
\end{proof}

An equivalent way to express the fact that the polynomial $I$ has to be isobaric
is through the action of the Lie algebra $H$ of diagonal matrices.
This translates to the fact that $I$ satisfies the system of differential equations
(see \cite{Stu} Theor. 4.5.2)

\begin{equation}\label{lietorus}
\sum_{i_0\ldots i_n} i_jf_{i_0\ldots i_n}
\frac{\partial I}{\partial f_{i_0\ldots i_n}}=\frac{md}{n+1}I\quad\forall j=0,\ldots, n.
\end{equation}

There is a second set of differential equations for the action of the triangular part $N^+\subset \mathfrak{sl}(n+1)$,
see Propositions \ref{invariantlie}, \ref{invariantlie3}.

\begin{example} In $S^3(S^4\CC^3)$ there are $23$ isobaric monomials
among $680$ monomials. In $S^6(S^3\CC^3)$ there are $103$ isobaric monomials
among $5005$. In Prop. \ref{precayley0} we will give a generating function that allows to compute
these numbers.
\end{example}

\subsection{Counting monomials of given weight.}

We compute now the number of monomials with a given weight.

Let $H_{g,d,p_0,\ldots, p_n}$
be the space of  monomials in $S^g(S^d\CC^{n+1})$ of weight
$(p_0,\ldots, p_n)$.
Since $p_0+\ldots +p_n=dg$, it is enough to record
the last $n$ entries of weight vector. 

We have
\begin{proposition}\label{precayley0}
$$\sum_{g=0}^{\infty}\sum_{p_0+\ldots+p_n=dg}H_{g,d,p_0,\ldots, p_n}x_1^{p_1}\ldots x_n^{p_n}y^g=
\prod_{i_1+\ldots +i_n\le d}\frac{1}{1-x_1^{i_1}\ldots x_n^{i_n}y}.$$
\end{proposition}
\begin{proof}
The variable $a_{i_0,\ldots, i_n}$ has degree $1$ and weight $(i_0,\ldots, i_n)$.
The left hand side is the Hilbert series for the multigraded ring
$K[a_{d,0,\ldots,0},\ldots, a_{0,\ldots, 0,d}]$.
The graded ring in just the variable $a_{i_0,\ldots, i_n}$ has Hilbert series
$\frac{1}{1-x_1^{i_1}\ldots x_n^{i_n}y}=1+x_1^{i_1}\ldots x_n^{i_n}y+x_1^{2i_1}\ldots x_n^{2i_n}y^2+\ldots$. Taking into account all the variables, we have the product
of the corresponding Hilbert series.
\end{proof}

Recall that for invariants of weight $(p_0,\ldots, p_n)$ we have
$\sum_i{p_i}=dg$.
\vskip 0.5cm

In case of binary forms, let $H_{g,p,d}$ be the space of homogeneous polynomials
of degree $g$ and weight $(p,dg-p)$ in $a_0,\ldots a_d$.

\begin{theorem}[Cayley-Sylvester]\label{cayleysylvester}

\begin{equation}\label{csformula}
\sum_p\dim H_{g,p,d}x^p=\frac{(1-x^{d+1})\ldots (1-x^{d+g})}{(1-x)\ldots (1-x^g)}.
\end{equation}
\end{theorem}

\begin{proof}
Write (after Proposition \ref{precayley0})
$\phi_d(x,y):=\sum_{p,g}dim H_{g,p,d}x^py^g=\prod_{i=0}^d\frac{1}{1-x^iy}=\sum_{g=0}^{\infty}C_{g,d}(x)y^g$,
so that the expression to be computed is $C_{g,d}(x)$.

We get
$$(1-y)\phi(x,y)=\prod_{i=1}^d\frac{1}{1-x^iy}=(1-x^{d+1}y)\prod_{i=1}^{d+1}\frac{1}{1-x^iy}=$$
$$(1-x^{d+1}y)\prod_{i=0}^{d}\frac{1}{1-x^{i+1}y}=(1-x^{d+1}y)\phi(x,xy).$$

Hence
$$(1-y)\sum_{j=0}^{\infty}C_{j,d}(x)y^j=(1-x^{d+1}y)\sum_{j=0}^{\infty}C_{j,d}(x)x^jy^j.$$

Comparing the coefficients of $y^j$ we get
$$C_{j,d}-C_{j-1,d}=C_{j,d}x^j-C_{j-1,d}x^{d+j},$$

hence
$$C_{j,d}(x)=\frac{1-x^{d+j}}{1-x^j}C_{j-1,d}(x).$$

Since $C_{0,d}(x)=1$, by induction on $j$ we get the thesis.
\end{proof}

Let's state the result in the case of $SL(3)$, for future reference.
We denote by $H_{g,d,p_0,p_2, p_2}$
be the space of monomials in $S^g(S^d\CC^{3})$ of weight
$(p_1,p_2,p_2)$, and we get
\begin{equation}\label{gensl3}\sum_{p_0+\ldots+p_n=d,g}H_{g,d,p_0,p_1, p_2}x_1^{p_1}x_2^{p_2}y^g=
\prod_{i_1=0}^d\prod_{i_2=0}^{d-i_1}\frac{1}{1-x_1^{i_1}x_2^{i_2}y}.\end{equation}

\subsection{Lie algebra action on forms}\label{liealgebraaction}

We give a $SL(2)$ example, which illustrates the general situation.
Recall that the generator $x=\left(\begin{array}{cc}0&1\\
0&0\end{array}\right)\in\mathfrak{sl}(2)$ integrates in the Lie group
to the one parameter subgroup  $e^{xt}=\left(\begin{array}{cc}1&t\\
0&1\end{array}\right)$.

\begin{proposition}\label{invariantlie} \

(i) $I\in S^mS^d\CC^2$ is invariant with respect to the subgroup
$\left(\begin{array}{cc}1&t\\
0&1\end{array}\right)$ if and only if  it is invariant with respect to the subalgebra $N^+$ if and only if $DI=0$ where
$D=\sum_{i=0}^{d-1}(i+1)a_i\frac{\partial}{\partial a_{i+1}}=a_0\frac{\partial}{\partial a_{1}}+2a_1\frac{\partial}{\partial a_{2}}+\ldots$

(ii) $I\in S^mS^d\CC^2$  is $SL(2)$-invariant if and only if it is isobaric and $DI=0$.
\end{proposition}

\begin{proof}
$I(x,y)=I(x'+ty',y')=a_0(x'+ty')^n+a_1n(x'+ty')^{n-1}y'+a_2{n\choose 2}(x'+ty')^{n-2}y'^2\ldots=
a_0x'^n+n(a_0t+a_1)x'^{n-1}y'+{n\choose 2}(a_0t^2+2a_1t+a_2)x'^{n-2}y'^2+\ldots$

So $a_k'=\sum_{i=0}^k{k\choose i}a_it^{k-i}=a_k+ka_{k-1}t+\ldots$

Now the condition is $0=\frac{d}{dt}I_{|t=0}=\sum\frac{\partial I}{\partial a_i'}\frac{\partial a_i'}{\partial t}_{t=0}$
and the proof of (i) is concluded by the fact that 
$\frac{\partial a_i'}{\partial t}_{t=0}=ia_{i-1}$.
(ii) follows because $\mathfrak{sl}(2)$ is generated by $h$ and $x$, see \S \ref{slsection}.\end{proof}

The main application of the previous proposition is that it allows to
compute explicitly invariants.
Remind the  equianharmonic quadric
$I=f_0f_4-4f_1f_3+3f_2^2\in S^2S^4\CC^2$ and our question posed in \S \ref{veronesesubsection}, namely
why the coefficients $(1,-4,3)$ ?

Now the coefficients can be computed by Prop. \ref{invariantlie}.

Call $\alpha$, $\beta$, $\gamma$ unknown coefficients and apply $D(\alpha f_0f_4+\beta f_1f_3+\gamma f_2^2)=
f_0f_3(4\alpha+\beta)+f_1f_2(3\beta+4\gamma)=0$.
We get that $(\alpha, \beta, \gamma)$ is proportional to $(1,-4,3)$.

In the same way we can prove the following dual version.

\begin{proposition} \

(i) $I\in S^mS^d\CC^2$ is invariant with respect to the subgroup
$\left(\begin{array}{cc}1&0\\
t&1\end{array}\right)$ if and only if  it is invariant with respect to the subalgebra $N^-$ if and only if $\Delta I=0$ where
$\Delta=\sum_{i=0}^{n-1}(n-i)a_{i+1}\frac{\partial}{\partial a_i}=na_1\frac{\partial}{\partial a_{0}}+(n-1)a_2\frac{\partial}{\partial a_{1}}+\ldots$

(ii) $I\in S^mS^d\CC^2$  is $SL(2)$-invariant if and only if it is isobaric and $\Delta I=0$.
\end{proposition}

The following Proposition is a computation contained in \cite{Hilb},
it is interesting because gave a motivation to study Lie algebras.

\begin{proposition}\

(i) $D\Delta-\Delta D=\sum_{i=0}^{d}(d-2i)a_{i}\frac{\partial}{\partial a_i}.$

(ii) $(D\Delta-\Delta D)(a_0^{\nu_0}\ldots a_d^{\nu_d})=\sum_{i=0}^d(d-2i)\nu_i(a_0^{\nu_0}\ldots a_d^{\nu_d})=(dg-2p)(a_0^{\nu_0}\ldots a_d^{\nu_d}).$
\end{proposition}

Hilbert proved from this proposition that an isobaric polynomial $F\in S^g(S^d\CC^2)$ (all its monomials have weight $p$ where  $dg-2p=0$)
satisfying $D F=0$ must satisfy also $\Delta F=0$, which is nowadays clear from the structure
of $\mathfrak{sl}(2)$-modules (see \S \ref{slsection} and also \S \ref{reynoldsection}).
Indeed their weight are segments centered around $(\frac{dg}{2}, \frac{dg}{2})$. 

Note that we have in the case of the quadrics generating the twisted cubic (see (\ref{hessiantwisted}))

$$\xymatrix{a_1a_3-a_2^2\ar@/^/[r]^{D}&a_0a_3-a_1a_2\ar@/^/[r]^{D}\ar@/^/[l]^{\Delta}&
a_0a_2-a_1^2\ar@/^/[l]^{\Delta}.}$$

For ternary forms we introduce the differential operators
$$D_1=
\sum_{i_0+\ldots +i_2=d} i_1f_{i_0+1,i_1-1, i_2}
\frac{\partial }{\partial f_{i_0\ldots i_2}}=0 ,
$$
$$D_2=
\sum_{i_0+\ldots +i_2=d} i_2f_{i_0,i_1+1, i_2-1}
\frac{\partial }{\partial f_{i_0\ldots i_2}}=0 .
$$

Note that $D_1$ adds $(1,-1,0)$ to the weight,
while $D_2$ adds $(0,1,-1)$ to the weight.

They correspond to $A_1$ and $A_2$ in \S \ref{slsection} and give the action on the directions depicted
in (\ref{dira1}) and (\ref{dira2}).
The following result is the natural extension to $\mathfrak{sl}(3)$ of Prop. \ref{invariantlie}, the proof is the same.

\begin{proposition}\label{invariantlie3} \

(i) $I\in S^mS^d\CC^3$ is invariant with respect to the subalgebra $N^+$
 if and only if  $D_1I=D_2I=0$. 

(ii) $I\in S^mS^d\CC^3$  is $SL(3)$-invariant if and only if it is isobaric and $D_1I=D_2I=0$.
\end{proposition}

\subsection{Cayley-Sylvester formula for the number of invariants of binary forms}
\label{cayleysylvestercomputation}
We recall now Cayley-Sylvester computation of the dimension of
invariants and covariants for binary forms.

We have already seen the Hessian $f_{xx}f_{yy}-f_{xy}^2$ of a binary form $f\in S^d\CC^2$,
which can be considered as a module $S^{2d-4}\CC^2\subset S^2(S^d\CC^2)$.
Indeed it is a polynomial of degree $2d-4$ in $x, y$ with coefficients of degree $2$ in $f_i$.
This is an example of a {\it covariant} of $f$. 
In general a covariant of degree $g$ of $f$ is any module $S^{e}\CC^2\subset S^g(S^d\CC^2)$,
so it corresponds to a ($SL(2)$-invariant) polynomial of degree $e$ in $x, y$ with coefficients of degree $g$ in $f_i$.

One of the most advanced achievement of classical period was the following computation of the number of covariants,
which is a nontrivial example of $SL(2)$-plethysm.

We recall from \S \ref{invtorusection} that $H_{g,p,d}$ is the space of monomials in $S^{g}S^d\CC^2$ of weight $(p,dg-p)$.

Let $D$ be the differential operator defined in Proposition \ref{invariantlie}.

Let $I_{g,p.d}$ be the kernel of the map $H_{g,p,d}\rig{D}H_{g,p-1,d}$. 
of degree $g$ and weight $p$ in $a_0,\ldots a_d$.

\begin{theorem}[Cayley-Sylvester]\label{cscoro}

(i) Let $p\le\frac{dg}{2}$. Then $\dim I_{g,p,d}$ is the degree $p$ coefficient in
$$\frac{(1-x^{d+1})\ldots (1-x^{d+g})}{(1-x^2)\ldots (1-x^g)}.$$

(ii) Let $2p=dg$. Then $\dim I_{g,\frac{dg}{2},d}$ is the dimension of the space of invariants so it is the coefficient of degree
$\frac{dg}{2}$ in 
$$\frac{(1-x^{d+1})\ldots (1-x^{d+g})}{(1-x^2)\ldots (1-x^g)}.$$

(iii) More generally we have the  $SL(2)$-decomposition $$S^g(S^d\CC^2)=\oplus_e S^e\CC^2\otimes I_{g,\frac{dg-e}{2},d}$$
where
$\dim I_{g,\frac{dg-e}{2},d}$  is equal to 
the coefficient of degree $\frac{dg-e}{2}$ in 
$$\frac{(1-x^{d+1})\ldots (1-x^{d+g})}{(1-x^2)\ldots (1-x^g)}.$$
\end{theorem}
\begin{proof}
(i) Any $\mathfrak{sl}(2)$ representations splits as a sum of irreducible representations,
each one centered around the weight $(\frac{dg}{2}, \frac{dg}{2})$ (this is meaningful even if
$\frac{dg}{2}$ is not an integer).
Hence for $p\le\frac{dg}{2}$ the differential $D$ is surjective and the result follows from Theorem \ref{cayleysylvester}.

(ii) follows from (i) and \ref{invariantlie}.

(iii) By the description in (i), the number of irreducible representations
isomorphic to $S^e$ can be computed looking at the space of monomials of weight $(\frac{dg-e}{2},\frac{dg+e}{2})$
which are killed by $D$.
\end{proof}

\begin{remark} For $p\le\frac{dg}{2}$, $\dim I_{g,p,d}$ may be computed by Proposition \ref{precayley0} also as the coefficient of $x^py^g$ in
$$(1-x)\prod_{i=0}^d\frac{1}{1-x^iy}.$$
\end{remark}
The following proof is borrowed from \cite{Hilb}.
We include it to show that concrete applications of Cayley-Sylvester formula can be painful.
We will give an alternative proof in Corollary \ref{corobinarycubic2}.

\begin{corollary}\label{corobinarycubic}
Let $d=3$. Let $I_{g,\frac{3g}{2},3}$ be the dimension of the space of invariants of degree $g$
of the binary cubic.  Then
$$\sum_{g=0}^{\infty}I_{g,\frac{3g}{2},3}x^g=\frac{1}{1-x^4}.$$
The ring of invariants is freely generated by the discriminant $D$, that is
$\oplus_m S^m(S^3(\CC^2))^{SL(2)}=\CC[D].$
\end{corollary}
\begin{proof}
We have from Theorem \ref{cscoro} that

$I_{g,\frac{3g}{2},3}$ is the degree $\frac{3g}{2}$ coefficient of
$$\frac{(1-x^{3+1})\ldots (1-x^{3+g})}{(1-x^2)\ldots (1-x^g)}=
\frac{(1-x^{g+1})(1-x^{g+2})(1-x^{g+3})}{(1-x^2)(1-x^3)}$$
and we write it as
$$\left\{\frac{(1-x^{g+1})(1-x^{g+2})(1-x^{g+3})}{(1-x^2)(1-x^3)}\right\}_{\frac{3g}{2}}.$$

We can remove the terms which do not change the coefficient of $x^{\frac{3g}{2}}$, so getting

$$\left\{\frac{(1-x^{g+1}-x^{g+2}-x^{g+3})}{(1-x^2)(1-x^3)}\right\}_{\frac{3g}{2}}=\left\{\frac{1}{(1-x^2)(1-x^3)}\right\}_{\frac{3g}{2}}-\left\{\frac{x(1+x+x^2)}{(1-x^2)(1-x^3)}\right\}_{\frac g2}=$$
(by using that $(1-x^3)=(1+x+x^2)(1-x)$ )
{\small
$$=\left\{\frac{1}{(1-x^2)(1-x^3)}\right\}_{\frac{3g}{2}}-\left\{\frac{x}{(1-x)(1-x^2)}\right\}_{\frac{g}{2}}=
\left\{\frac{1}{(1-x^2)(1-x^3)}-\frac{x^3}{(1-x^3)(1-x^6)}\right\}_{\frac{3g}{2}}=
$$}

$$=\left\{\frac{1-x^6-x^3+x^5}{(1-x^2)(1-x^3)(1-x^6)}\right\}_{\frac{3g}{2}}=
\left\{\frac{(1-x^2)(1-x^3)+x^2-x^6}{(1-x^2)(1-x^3)(1-x^6)}\right\}_{\frac{3g}{2}}=$$

$$=\left\{\frac{1}{(1-x^6)}\right\}_{\frac{3g}{2}}+
\left\{\frac{x^2(1+x^2)}{(1-x^3)(1-x^6)}\right\}_{\frac{3g}{2}}=$$

$$=\left\{\frac{1}{(1-x^4)}\right\}_{g}+
\left\{\frac{x^4(1+x^4)}{(1-x^6)(1-x^{12})}\right\}_{3g}.$$

The second summand contains only terms of the form $x^{4m+3n}$ with $m=1$ or $2$,
hence none of the form $x^{3g}$, and so does not contribute anything.

\end{proof}

\subsection{Counting partitions and symmetric functions}
The following combinatorial interpretation is interesting.

\begin{proposition} The number $H_{g,p,d}$ is the number of partitions
of $p$ as a sum of at most $g$ summands from $\{1,\ldots, d\}$,
which is equal to  the number of Young diagrams
with $p$ boxes, at most $d$ rows and at most $g$ columns.
\end{proposition}
\begin{proof}
$H_{g,p,d}$ counts monomials
$a_0^{i_0}\ldots a_d^{i_d}$ with $i_0+\ldots +i_d=g$ and
$i_1+2i_2+\ldots +di_d=p$ which correspond to the partition
$$\underbrace{1+\ldots+ 1}_{i_1}+ \underbrace{2+\ldots + 2}_{i_2}+\ldots
+ \underbrace{d+\ldots + d}_{i_d}=p.$$
\end{proof}

\begin{corollary}[Hermite reciprocity]\label{hermiterep}

$$H_{g,p,d}=H_{d,p,g},\qquad S^g(S^d\CC^2)=S^d(S^g\CC^2).$$

In equivalent way the number $H_{g,p,d}$  counts also 
the number of partitions
of $p$ as a sum of at most $d$ summands from $\{1,\ldots, g\}$.

\end{corollary}
\begin{proof}Taking transposition of Young diagrams.
\end{proof}

For $d\ge p$, $H_{g,p,d}$ stabilizes to $H_{g,p}$ which is the number of partitions
of $p$ as a sum of at most $g$ summands, or also  the number of Young diagrams
with $p$ boxes and at most $g$ rows. The coefficient $H_{g,p}$ is equal also to the number of partitions
of $p$ with summands from $\{1,\ldots, g\}$, indeed it is the number of Young diagrams
with $p$ boxes and at most $g$ columns, which are the transpose of the previous ones.

For $d\ge p$, the numerator of (\ref{csformula}) has no role in the computation
and we get

\begin{equation}\label{basicformula}\sum_{p=0}^{\infty}H_{g,p}x^p=\prod_{i=1}^g\frac{1}{(1-x^i)}.\end{equation}

This formula has a more elementary interpretation, indeed we have
{\footnotesize
$$\frac{1}{(1-x)(1-x^2)\ldots(1-x^g)}=\left(\sum_{i_1\ge 0}x^{i_1}\right)\left(\sum_{i_2\ge 0}x^{2i_2}\right)\cdots\left(\sum_{i_g\ge 0}x^{gi_g}\right)=
\sum_{i_1,\ldots, i_g}x^{i_1+2i_2+\ldots +gi_g}$$}
and the coefficient of $x^p$ counts all indexes $i_j$ which make  $i_1+2i_2+\ldots +gi_g=p$.
Cayley-Sylvester argument can be interpreted as a refinement of this elementary computation.
The formula (\ref{basicformula}) can be considered as the Hilbert series
of the ring of symmetric functions polynomials in $g$ variables,
having Schur polynomials $s_{\lambda}$ as a basis.

Note from the Proposition \ref{precayley0} also the identity

\begin{equation}\label{infsym}\sum_{p,g}dim H_{g,p}x^py^g=\prod_{i=0}^{\infty}\frac{1}{1-x^iy}\end{equation}
which is meaningful because for each power of $x$ only finitely many factors are taken into account
in the right-hand side.

For any fixed $p$, for $g\ge p$ $H_{g,p}$ stabilizes to $H_p$,
the number of partitions of $p$, and from
(\ref{basicformula}) we can write the basic formula

$$\sum_{p=0}^{\infty}H_px^p=\prod_{i=1}^{\infty}\frac{1}{(1-x^i)}$$
which is again meaningful because for each power of $x$ only finitely many factors are taken into account
in the right-hand side. The formula (\ref{infsym})  can be considered as the Hilbert series
of the ring of symmetric functions polynomials in infinitely many variables.

Hardy and Ramanujam proved the remarkable asympotic formula   
$$H_p\sim \frac{1}{4p\sqrt{3}}\exp{\pi\sqrt{
\frac{2p}{3}}}\textrm{\ for\ }p\to\infty$$
which is one of the many wonderful and ``not expected'' appearances of $\pi$
in combinatorics.
For an interesting proof in the probabilistic setting, see
\cite{BD}.

\subsection{Generating formula for the number of invariants of ternary forms}
We saw in Theorem \ref{cscoro} that the dimension of the space
of invariants in $S^gS^d\CC^2$ can be obtained as $\dim H_{g,p,d}-\dim H_{g,p-1,d}$.
The following Theorem perform an analogous computation for $SL(3)$-invariants,
and gives a clue how to perform such computations in general.

\begin{theorem}[Sturmfels\cite{Stu} Algorithm 4.7.5, Bedratyuk \cite{Bed1, Bed2}], \label{stubed}
Let $dg=p_0+p_1+p_2$. Let $h_{g,d,p_0,p_1,p_2}=\dim H_{g,d,p_0,p_1,p_2}$.
Let $I_{g,d,p,p,p}$ be the space of invariants in $S^gS^d\CC^3$.
{ \begin{equation}\label{sixterms}\begin{split}\dim I_{g,d,p,p,p}=
h_{g,d,p,p,p}-h_{g,d,p+1,p-1,p}-h_{g,d,p-1,p,p+1}+\\h_{g,d,p+1,p-2,p+1}
+h_{g,d,p-1,p-1,p+2}-h_{g,d,p,p-2,p+2}.\end{split}\end{equation}}
So by (\ref{gensl3}) we get that $\dim I_{g,d,p,p,p}$ is the coefficient of $x_1^{p+2}x_2^{p}y^g$ in 
$$(x_2-x_1)(x_2-1)(x_1-1)\prod_{i_1=0}^d\prod_{i_2=0}^{d-i_1}\frac{1}{1-x_1^{i_1}x_2^{i_2}y}.$$
\end{theorem}
 \begin{proof} $S^dS^g\CC^2$ splits as the sum of irreducible representations. 
Representing their weights, they span hexagons or triangles like in \S \ref{slsection}, all
centered around $(\frac{dg}{3},\frac{dg}{3},\frac{dg}{3})$. The right hand side counts $1$ on
every trivial summand in $S^gS^d\CC^3$ and it counts $0$ on
every nontrivial summand in $S^gS^d\CC^3$. The details are in \cite{Bed1, Bed2}.
\end{proof}

The six summands in (\ref{sixterms}) correspond to the action of the Weyl group $\Sigma_3$ for $SL(3)$,
and are identified as vertices of a hexagon (see next table (\ref{bedformula})). In \cite{Bed2} the formula is extended to any $SL(n+1)$.

The following M2 script shows how to get the whole decomposition of
$h_{g,d,*}$ in cases of cubic invariant of plane quartics,
other cases can be obtained by  changing the values of $d$ and $m$.
For larger values, the size of the computation can be reduced
by contracting separately with $y^mx1^{i_1}x2^{i_2}$ and bounding correspondingly the upper limit of the sum.

{\footnotesize
\begin{verbatim}
R=QQ[y,x1,x2]
d=4, m=3
sub(contract(y^m,product(d+1,i->product(d+1-i,j->sum(m+1,k->y^k*x1^(k*i)*x2^(k*j))))),{y=>0})
\end{verbatim}
}

The output of previous script can be recorded in the following table
\begin{equation}\label{bedformula}\xymatrix@R-2.0em@C-1.3em{1\\
&1\\
1&&2\\
&2&&3\\
2&&4&&4\\
&4&&6&&4\\
3&&8&&8&&5\\
&6&&11&&9&&4\\
4&&11&&*++[o][F]{15}&&9&&4\\
&8&&*++[o][F]{16}&&*++[o][F]{15}&&8&&3\\
4&&15&&19&&15&&6&&2\\
&9&&*++[o][F]{19}&&*++[o][F]{19}&&11&&4&&1\\
5&&15&&*++[o][F=]{23}&&16&&8&&2&&1\\
&9&&19&&19&&11&&4&&1\\
4&&15&&19&&15&&6&&2\\
&8&&16&&15&&8&&3\\
4&&11&&15&&9&&4\\
&6&&11&&9&&4\\
3&&8&&8&&5\\
&4&&6&&4\\
2&&4&&4\\
&2&&3\\
1&&2\\
&1\\
1\\}
\end{equation}

The sum of all the numbers in the triangle is $680=\dim S^3(S^4\CC^3))$.
The rounded entries correspond to the six summands in Bedratyuk formula (\ref{sixterms}).
According to the table (\ref{bedformula})
we compute from Theorem \ref{stubed} $\dim I_{3,4,4,4,4}=23-19-19+16+15-15=1$.

Note that in (\ref{bedformula}) it is enough to record the following ``one third'' part,
and the others can be filled by symmetry.

$$\xymatrix@R-2.0em@C-1.3em{
4\\
&4\\
8&&5\\
&9&&4\\
15&&9&&4\\
&15&&8&&3\\
19&&15&&6&&2\\
&19&&11&&4&&1\\
*+[o][F]{23}&&16&&8&&2&&1\\
&19&&11&&4&&1\\
&&15&&6&&2\\
&&&8&&3\\
&&&&4\\}$$

 In Sturmfels'  book\cite{Stu} there is the computation of Hilbert series for plane quartics up to degree $21$.
Shioda got in \cite{Shi} the complete series, from where one can guess possible generators.
There is an unpublished paper by Ohno, claiming a system of  generators for the invariant subring of plane quartics
(see \cite{BLRS}).

\begin{remark} L. Bedratyuk gives in \cite{Bed1, Bed2} other formulas which extend Theorem \ref{cscoro} to $n$-ary forms.
On his web page is available some software packages implementing these formulas.
\end{remark}

\subsection{The Reynolds operator and how to compute it. Hilbert finiteness theorem.}\label{reynoldsection}

When $G=SL(n+1)$ acts on a vector space $V$, we recall from \S \ref{basicsrepresentations} the notation
$$V^G=\{v\in V| g\cdot v=v\forall g\in G\}.$$

Since $SL(n+1)$ is reductive we have the direct sum
$V=V^G\oplus \left(V^G\right)^{\perp}$, where $\left(V^G\right)^{\perp}$ is the sum of all irreducible invariant subspaces of $V$
which are nontrivial (see Theorem \ref{schurweyltheorem}).

The Reynolds operator is the projection
$$R\colon V\to V^G.$$

If $G$ is a finite group, such a projection can be obtained just averaging over the group,
that is $R(v)=\frac{1}{|G|}\sum_{g\in G}g\cdot v$. In the case of $G=SL(n+1)$, Weyl replaced the sum
with the integration over a maximal compact real subgroup, which is $SU/n+1)$, this is the ``Weyl unitary trick''\cite{FH}.
It is difficult to compute the integral directly from the definition.
We want to show, in an example, how the Reynolds operator can be explicitly computed by the Lie algebra techniques.

\begin{proposition}
Let $f=a_0x^4+4a_1x^3y+6a_2x^2y^2+4a_3xy^3+a_4y^4\in S^4\CC^2$.
We consider $R\colon S^2(S^4\CC^2)\to  S^2(S^4\CC^2)^G$.
 Then $$R(a_0a_4)=\frac{2}{5}I,\qquad R(a_1a_3)=-\frac{1}{10}I,\qquad R(a_2^2)=\frac{1}{15}I,$$
where  $I=a_0a_4-4a_1a_3+3a_2^2\in S^2(S^4\CC^2)^G$,
while $R$ vanishes on all other monomials in $S^2(S^4\CC^2)$.
\end{proposition}

\begin{proof}

The key is to consider the differential operator
$D=\sum_{i=0}^4(i+1)a_i\frac{\partial}{\partial a_{i+1}}$ acting on $S^2(S^4\CC^2)$.
The torus action has the eigenspaces $S^2(S^4\CC^2)=\oplus_{}H_i$
where the weight of $a_ia_j$ is $2(i+j)-8$.
We write dimensions in superscripts

$H_{-8}^1\lef{D}H_{-6}^1\lef{D}H_{-4}^2\lef{D}H_{-2}^2\lef{D}H_{0}^3\lef{D}H_{2}^2\lef{D}H_{4}^2\lef{D}H_{6}^1\lef{D}H_{8}^1$

which come from the three representations
$$S^2(S^4\CC^2)=S^8\CC^2\oplus S^4\CC^2\oplus S^0\CC^2$$
namely

$\begin{array}
{ccccccccccccccccc}
\circ&\lef{}&\circ&\lef{}&\circ&\lef{}&\circ&\lef{}&\circ&\lef{}&\circ&\lef{}&\circ&\lef{}&\circ&\lef{}&\circ\\
&&&&\circ&\lef{}&\circ&\lef{}&\circ&\lef{}&\circ&\lef{}&\circ\\
&&&&&&&&\circ\\
\end{array}$

We  refer to the three rows above. We have

$
a_0^2\lef{}a_0a_1\lef{}\begin{array}{c}a_0a_2\\a_1^2\end{array}\lef{}\begin{array}{c}a_0a_3\\a_1a_2\end{array}\lef{}\begin{array}{c}a_0a_4\\a_1a_3\\a_2^2\end{array}\lef{}\begin{array}{c}a_1a_4\\a_2a_3\end{array}\lef{}\begin{array}{c}a_2a_4\\a_3^2\end{array}\lef{}a_3a_4\lef{}a_4^2$

and the goal is to split these monomial spaces into the three previous rows.
We compute the splitting as follows.

Since $D^5(a_3^2)=720a_0a_1$, $D^5(a_2a_4)=720a_0a_1$,
then $D^5(a_2a_4-a_3^2)=0$ and  $a_2a_4-a_3^2$ belongs to the second row.

Hence $D^2(a_2a_4-a_3^2)=2(a_0a_4+2a_1a_3-3a_2^2)$ belongs again to the second row. 

Moreover $D^4(a_4^2)=48(a_0a_4+16a_1a_3+18a_2^2)$  belongs to the first row.
Hence we have the decomposition $$H_{0}^3=<I>\oplus <a_0a_4+2a_1a_3-3a_2^2, a_0a_4+16a_1a_3+18a_2^2>$$
and $R$ is the projection over $<I>$.
We get the scalars stated by inverting the matrix

${\left[\begin{array}{rrr}1&
      {-4}&
      3\\
      1&
      2&
      {-3}\\
      1&
      16&
      18\\
      \end{array}\right]}^{-1}={\left[\begin{array}{rrr}\frac 25&
      *&
      *\\
      -\frac{1}{10}&
      *&
      {*}\\
      \frac{1}{15}&
      *&
      *\\
      \end{array}\right]}.$

\end{proof}

From the above proof it should be clear that the same technique can be applied in 
order to compute the Reynolds operator in other cases. For an alternative approach, by using Casimir operator, see \cite{DK} 4.5.2.
The computations performed by Derksen and Kemper are essentially equivalent to ours.

\begin{proposition}\label{reyn}
Let $R=\oplus_i R_i$ be a graded ring where $G=SL(n)$ acts.
If $f\in R_i^G$, $g\in R_j$, then
$$R(fg)=fR(g).$$
\end{proposition}
\begin{proof}
  Decompose $g=g_1+g_2$ where $g_1\in R_j^G$ and $g_2$ belongs to its complement $(R_j^G)^{\perp}$.

Then $fR_j^G\subset R_{i+j}^G$ and $f(R_j^G)^{\perp}\subset (R_{i+j}^G)^{\perp}$,
indeed if $f$ is nonzero, $fR_j^G$ is a $G$-module isomorphic to $R_j^G$
and $f(R_j^G)^{\perp}$  is a $G$-module isomorphic to $(R_j^G)^{\perp}$.
\end{proof}

We close this section by recalling the wonderful proof about the finite generation of the invariant ring,
proved first by Hilbert in 1890. 
Hilbert result, together with the 1FT (which we will see in \S \ref{twoffforms}) implies that the ring
of invariants is generated by finitely many products of tableau like in \S \ref{tableausubsection}.

\begin{theorem}[Hilbert]\label{hilbertfiniteness}
Let $G=SL(n+1)$ (although any reductive group works as well).
Let $W$ be a finite dimensional $G$-module.
Then the invariant subring $\CC[W]^G$ is finitely generated.
\end{theorem}
\begin{proof}
Let $I$ be the ideal in $\CC[W]$ generated by all the homogeneous invariants of positive degree.
Since $\CC[W]$ is Noetherian we get that $I$ is generated by homogeneous invariants $f_1,\ldots, f_r$.

We will prove that for any degree $d$
$\CC[W]^G_d=\CC[f_1,\ldots, f_r]_d$ by induction on $d$.
The case $d=0$ is obvious.
Let $f\in \CC[W]^G_d\subset I_d$.
Then there exist $a_i\in R$ such that $f\in\sum_{i=0}a_if_i$.

We get $f=R(f)=\sum_iR(a_if_i)=\sum_iR(a_i)f_i$,
the last equality by Proposition \ref{reyn}.
By the inductive assumption,  each $R(a_i)$ is a polynomial in $f_i$,
hence the same is true for $f$.

\end{proof}

\subsection{Tableau functions. Comparison among different applications of Young diagrams}
\label{tableausubsection}
Let $V=\CC^{n+1}$, more precisely for the following it is enough to consider a $(n+1)$-dimensional vector space
with a fixed isomorphism $\wedge^{n+1}V\simeq\CC$. 
Then for every $v_1,\ldots v_{n+1}\in V$ the determinant
$v_1\wedge\ldots\wedge v_{n+1}\in\CC$ is well defined and it is $SL(n+1)$-invariant for the natural
action of $SL(n+1)$ on $V$. Every rectangular tableau $T$ over a Young diagram of size
$(n+1)\times g$ gives a tableau function which is constructed
by taking the product of the determinant arising from each column.
 So
$\young(1,2,3)$
represents $x_1\wedge x_2\wedge x_3$, and $\young(11,23,34)$ represents $\left(x_1\wedge x_2\wedge x_3\right)\left(x_1\wedge x_3\wedge x_4\right)$.

To define formally this notion, we set $[a]=\{1,\ldots, a\}$ for any natural number $a$
and we notice that a tableau $T$ is encoded in a function
$t\colon [n+1]\times [g]\to [m]$ , where $t(i,j)$ corresponds to the entry at the place $(i,j)$ of the tableau.

\begin{definition}[From tableau to multilinear invariants]\label{tableaufunction} For any tableau $T$ over a Young diagram of size
$(n+1)\times g$ filled with numbers from $1$ appearing $h_1$ times until $m$ appearing $h_m$ times,
so that $h_1+\ldots +h_m=g(n+1)$, we denote by $G_T$ the multilinear function
$S^{h_1}V^{\vee}\times \ldots \times S^{h_m}V^{\vee}\to \CC$ defined by
$$G_T(x_1^{h_1},\ldots, x_m^{h_m})=\prod_{j=1}^m\left(x_{t(1,j)}\wedge\ldots\wedge x_{t(n+1,j)}\right).$$
$G_T$ is well defined by Theorem \ref{veronesespan}.
\end{definition}

\begin{proposition}\label{gtinvariant} Every tableau function $G_T$ is $SL(V)$-invariant.

\end{proposition}
\begin{proof} Immediate by the properties of the determinant.
\end{proof}

Every Young diagram $\lambda$ defines (at least) four interesting objects, which give four similar theories, that is

\begin{itemize}
\item{}a representation $S^{\lambda}V$ of $SL(V)$ (see Theor. \ref{fillglv})

\item{}a representation $V_{\lambda}$ of the symmetric group (see Def. \ref{defspecht})

\item{}a symmetric polynomial $s_{\lambda}$ (see Def. \ref{defschurpoly})

\item{}a Schubert cell $X_{\lambda}$ in the Grassmannian (it will be defined in a while)
\end{itemize}

This comparison is carefully studied in \cite{Man0}.
The Schur functions $s_{\lambda}$ give an additive basis of the ring of symmetric polynomials.
There is a algorithm to perform,
given a symmetric function $f$, the decomposition $f=\sum_{\lambda}c_{\lambda}s_{\lambda}$,
described in the Algorithm 4.1.16 in \cite{Stu}.
Essentially, let $ct_1^{\nu_1\ldots t_n^{\nu_n}}$ be the leading term in $f$,
let $\lambda$ be the unique partition of $d$ such that the corresponding highest weight 
of $s_{\lambda}$ is $t_1^{\nu_1\ldots t_n^{\nu_n}}$. Then we output the summand $cs_{\lambda}$
and we continue with $f-cs_{\lambda}$.

From the computational point of view, among the four similar theories listed above, the symmetric functions are the ones that can be better understood. 
For example, to compute $S^2(S^3\CC^3)$
consider first the $10$ monomials $t_1^3, t_1^2t_2,\ldots, t_3^3$.
Then consider the sum of all product of two of these monomials, which is
$t_1^6+t_1^5t_2+\ldots +t_3^6$. This last polynomial can be decomposed as
$s_6+s_{4,2}$. It follows that $S^2(S^3\CC^3)=S^6\CC^3+S^{4,2}\CC^3$.

Often it is convenient, in practical computations, to consider Newton sum of powers,
which behave better according to plethysm.

The Schubert cell $X_{\lambda}$ is defined as

$X_{\lambda}=\{m\in Gr(\P^k,\P^n)|\dim\ m\cap V_{n-k+i-\lambda_i}\ge i, \textrm{\ for\ }i=1,\ldots ,k+1\}$

where $e_0,\ldots , e_n$ is a basis of $V$ and
$V_i=<e_0,\ldots , e_i>$.

We have the following formulas which show the strong similarities among different theories 

$$S^{\lambda}V\otimes S^{\mu}V=\sum_\nu c_{\lambda\mu\nu}S^{\nu}V$$
for some integer coefficients $c_{\lambda\mu\nu}$,
which repeat in the following

$$s_{\lambda}\cdot s_{\mu}=\sum_\nu c_{\lambda\mu\nu}s_{\nu},$$
$$X_{\lambda}\cap X_{\mu}=\sum_\nu c_{\lambda\mu\nu}X_{\nu}.$$

The tensor product of $\Sigma_d$-modules $V_{\lambda}$ behaves in a different way.

\subsection{The symbolic representation of invariants.}\label{symbolicrep}

The symbolic representation is an economic way to encode and write down invariants.
It works both for invariants of forms, that we consider here, and for invariants of points, that we consider in \S \ref{sectionpoints}.

It was called by Weyl ``the great war-horse of nineteenth century invariant theory''. 
The reader should not lose the historical article \cite{Ro}.
The ``symbolic calculus'' is essential to understand the classical sources.
In the words of Enriques and Chisini (\cite{EC} pag. 37, chap. 1):

{\it ``Ma a supplire calcoli laboriosi determinandone a priori il risultato, si può anche far uso del procedimento di notazione simbolica di Cayley-Aronhold, che risponde a questa esigenza economica porgendo un modo sistematico di formazione. L'idea fondamentale contenuta nella rappresentazione simbolica costituisce un fecondo principio di conservazione formale rispetto alle degene\-razioni.''}\footnote{In order to avoid messy computations, by determining in advance their result,
it may be used Cayley-Aronhold symbolic representation. It gives a systematic way to construct the invariants by answering the need of simplicity. The fundamental idea of symbolic representation gives a fruitful invariance principle with respect to degeneration.}

Since every invariant corresponds to a one dimensional representation in $S^mS^dV$, they are spanned by tableau as in Theorem \ref{youngsymm}.

Every $F\in S^mS^dV$ corresponds to a multilinear function $$F\colon\underbrace{S^dV\times\ldots\times S^dV}_m\to\CC$$
and this last is determined by Theor. \ref{veronesespan} by
$F(x_1^d,\ldots, x_m^d)$ for any linear forms $x_i$, $i=1,\ldots, m$,
which is symmetric in the $m$ entries corresponding to our label numbers.

Let $md=g(n+1)$. In the symbolic representation we start from a tableau $T$ filling the Young diagram $g^{n+1}$
with the numbers $1$ repeated $d$ times, $2$ repeated $d$ times and so on until $d$ repeated $d$ times.
The main construction of the symbolic representation is to define an invariant $F_T\in S^m(S^d\CC^{n+1})^{SL(n+1)}$, by using the above idea.

\begin{definition}\label{symboliconstruction}[From tableau to polynomial invariants] Let $md=g(n+1)$. Let $T$ be a tableau filling the Young diagram of rectangular size $(n+1)\times g$
with the numbers $1$ repeated $d$ times, $2$ repeated $d$ times and so on until $m$ repeated $d$ times.
Let $G_T\colon \underbrace{S^dV^{\vee}\times\ldots\times S^dV^{\vee}}_m\to\CC$ be the
function introduced in Definition \ref{tableaufunction}.
We denote by $F_T\in S^m\left(S^dV\right)$  the polynomial obtained by symmetrizing $G_T$,
that is $F_T(h)=G_T(\underbrace{h,\ldots , h}_m)$ for any $h\in S^dV^{\vee}$. $F_T$ is called a symmetrized tableau function.
\end{definition}
\begin{theorem}
Any $F_T$ as in Definition \ref{symboliconstruction} is $SL(n+1)-invariant$.
\end{theorem}
\begin{proof} Let $g\in SL(n+1)$ and  $h\in S^dV^{\vee}$. We have $F_T(g\cdot h)=G_T(\underbrace{g\cdot h,\ldots, g\cdot h}_m)=
G_T(\underbrace{h,\ldots, h}_m)=F_T(h)$,
the second equality by Proposition \ref{gtinvariant}.
\end{proof}

By the description of representations  of $GL(V)$,
we get an element $F_T$ in the representation of dimension one
$S^{\lambda}V\subset S^m(S^dV)$, possibly zero.

We emphasize that $F_T(x^d)=G_T(\underbrace{x^d,\ldots , x^d}_m)=0$ for every $x\in V^{\vee}$, because we get a determinant with equal rows.
Nevertheless the symmetrization of Def. \ref{symboliconstruction} is meaningful
for general $h\in S^dV^{\vee}$. It is strongly recommended to practice with some examples in order to understand how the construction 
of Def. \ref{symboliconstruction} is powerful.
Let's start with the example of the invariant $J$ for binary quartics (compare also with the different example 4.5.7 in \cite{Stu}).

Fill the Young diagram $\lambda=(6,6)$ with respectively four copies of each $1$, $2$, $3$, obtaining the following

\begin{equation}\label{invj}T=\begin{matrix}\young(111122,223333).\end{matrix}\end{equation}

It is defined the function
 $F_T(x^4,y^4,z^4)=(x\wedge y)^2(x\wedge z)^2(y\wedge z)^2$.
In the classical literature, this representation was denoted sometimes as 
\begin{equation}\label{symbclas}[12]^2[13]^2[23]^2.\end{equation}

By developing $(x_0y_1-x_1y_0)^2(x_0z_1-x_1z_0)^2(y_0z_1-y_1z_0)^2$,
 we get $19$ monomials.
The first monomial is $x_0^4y_0^2y_1^2z_1^4$ and , according to the correspondence
seen in (\ref{explicitveronesespan}), we get
$x_0^4\mapsto a_0$, $y_0^2y_1^2\mapsto a_2$, $z_1^4\mapsto a_4$,
so that the first monomial corresponds to
$a_0a_2a_4$.

The following M2 script does automatically this job and it can be adapted to other symbolic expressions.
\begin{verbatim}
S=QQ[x_0,x_1,y_0,y_1,z_0,z_1,a_0..a_4]
f=(x_0*y_1-x_1*y_0)^2*(x_0*z_1-x_1*z_0)^2*(y_0*z_1-y_1*z_0)^2
symb=(x,h)->(contract(x,h)*transpose matrix{{a_0..a_4}})_(0,0)
fx=symb(symmetricPower(4,matrix{{x_0,x_1}}),f)
fxy=symb(symmetricPower(4,matrix{{y_0,y_1}}),fx)
fxyz=symb(symmetricPower(4,matrix{{z_0,z_1}}),fxy)
\end{verbatim}

The result we get for $F_T$ is (up to scalar multiples) the well known expression of $J$ (compare with (\ref{firstJ}))
$$F_T=-{a}_{2}^{3}+2 {a}_{1} {a}_{2} {a}_{3}-{a}_{0} {a}_{3}^{2}-{a}_{1}^{2}
     {a}_{4}+{a}_{0} {a}_{2} {a}_{4}.$$

In formula (\ref{invj}), the bracket $[12]$ appears in two among the columns. If  $[12]$ appears in three columns, one is forced to repeat $3$ on the same column, getting zero. By similar elementary arguments, the reader can easily convince himself that (\ref{invj}) gives the only nonzero invariant in $S^3S^4\CC^2$.

For higher degree invariants, one meets quickly
very large expressions. A computational trick, to reduce  the size of the memory used,
 is to introduce the expression $f$ step by step and to manage the symbolic reduction
of any single variable correspondingly.
In the following example, the variable $x$ appears already in the first two square factors of $f$,
and the symbolic reduction of $x$ can be done just at this step. 

\begin{verbatim}
S=QQ[x_0,x_1,y_0,y_1,z_0,z_1,a_0..a_4]
ff=(x_0*y_1-x_1*y_0)^2*(x_0*z_1-x_1*z_0)^2
symb=(x,h)->(contract(x,h)*transpose matrix{{a_0..a_4}})_(0,0)
fx=symb(symmetricPower(4,matrix{{x_0,x_1}}),ff)
--now we introduce the third square factor
fxy=symb(symmetricPower(4,matrix{{y_0,y_1}}),fx*(y_0*z_1-y_1*z_0)^2)
fxyz=symb(symmetricPower(4,matrix{{z_0,z_1}}),fxy)
\end{verbatim}

Note that $\left(\CC^2\right)^6$ contains for $\lambda=(3,3)$
a $SL(2)$-invariant subspace of dimension equal to $\dim V_{3,3}=5$.
The dimension of $V_{m,m}$ is equal to $\frac{1}{m+1}{{2m}\choose m}$, the $m$-th Catalan number.

The main hidden difficulty in the application of the method of  ``symbolic representation'' is that it is hard to detect
in advance if a given symbolic expression gives the zero invariant.

\subsection{The two Fundamental Theorems for invariants of forms}
\label{twoffforms}

The First Fundamental Theorem for invariants of forms says that any invariant of forms is a linear combination
of the invariants $F_T$ in Definition \ref{symboliconstruction}.

\begin{theorem}[First Fundamental Theorem (1FT) for forms]\label{1ftforms}
\

Let $md=(n+1)g$. The space of invariants of degree $m$ for $S^dV$ is generated
by symmetrized tableau functions $F_T$, constructed by  tableau $T$ as in Definition \ref{symboliconstruction}. 

\end{theorem}

We postpone the proof until \S \ref{sectionpoints}, after 1FT for invariants of points (Theorem \ref{1ftpoints}) will be proved.

The theorem extends in a natural way to covariants.
For example the ``symbolic representation'' of the Hessian $H$ defined in \ref{Hexample}
is

$\young(1112,22yy)$
or $H=[12]^2[1y][2y]$.

The second fundamental theorems  describes the relations
between these invariants.

\begin{theorem}[Second Fundamental Theorem (2FT) for forms]\label{2ftforms}
\

The relations among the generating invariants $F_T$ of Theorem \ref{1ftforms} are 
 generated by  Pl\"ucker relations like in Remark \ref{plueckerquadrics}.

More precisely, fix a subset of $k+2$ elements
$i_0\ldots i_{k+1}$ and a set of $k$ elements $j_0\ldots j_{k-1}$.
Let $T_s=[i_0\ldots \hat{i_s}\ldots i_{k+1}]
[i_sj_0 \ldots j_{k-1}]$ a $2\times (k+1)$ tableau and assume  the numbers appearing are as in 
 Definition \ref{symboliconstruction} (this does not depend on $s$ because every $T_s$ permutes the same numbers).
Then the Pl\"ucker relations are
$$\sum_{s=0}^{k+1}(-1)^sF_{T_s}=0,$$
which hold for  any  subsets of respectively
$k+2$ and $k$ elements.
\end{theorem}

This description gives unfortunately cumbersome computations.

\section{ Hilbert series of invariant rings.  Some more examples of invariants.   }
\subsection{Hilbert series}\label{hilbertseries}

In all the examples where a complete system of invariants (or covariants) is known,
the following steps can be performed

\begin{itemize}
\item{(i)} compute the Hilbert series of the invariant ring.

\item{(ii)} guess generators of the corrected degree.

\item{(iii)} compute the syzygies among the generators of step (ii),
hence compute the subalgebra generated by these generators.

\item{(iv)} check if the subalgebra coincides at any degree with the algebra by comparing
the two Hilbert series.

\end{itemize}

The algebra of covariants has the bigraduation
$$Cov(S^d\CC^2)=\oplus_{n,e}Hom^{SL(2)}(S^n(S^d\CC^2),S^e\CC^2)$$ and its Hilbert series
depends correspondingly on two variables $z, w$
$$F_d(z,w)=\dim Cov(S^d\CC^2)_{n,e}z^nw^e$$

in such a way that
 the coefficient of $w^az^b$ denotes the dimension of
$Hom^{SL(2)}(S^b\CC^2,S^a(S^d\CC^2))$. The exponent of $z$ is the degree, the exponent of $w$ is called the order. The covariants of order zero coincide with the invariants.

In 1980 Springer has found an efficient algorithm for the computation of Hilbert series for binary forms,
by using residues.
Variations of this algorithm may be found in Procesi's book \cite{Pr} chap. 15, \S 3 or in Brion's notes \cite{Br}.
In the following we follow   \cite{Br}.

Denote $\Phi_j\colon\CC[z]\to\CC[z]$ the linear map given by

$$\Phi_j(z^n)=\left\{\begin{array}{cc}z^{n/j}&\textrm{if\ }n\equiv 0\textrm{\ mod\ }j\\
0&\textrm{otherwise}
\end{array}\right.$$

$\Phi_1(z^n)=z^n$

$\Phi_2(z^n)=\left\{\begin{array}{cc}z^{n/2}&\textrm{if\ }n\textrm{\ is even}\\
0&\textrm{otherwise}
\end{array}\right.$

$\Phi_3(z^n)=\left\{\begin{array}{cc}z^{n/3}&\textrm{if\ }n\equiv 0\textrm{\ mod\ }3\\
0&\textrm{otherwise}
\end{array}\right.$

We have the equality $\Phi_j(ab)=\Phi_j(a)\Phi_j(b)$, if $a\in\CC[z^j]$ (or $b\in\CC[z^j]$).
The map $\Phi_j$ extends to a unique linear map
$$\Phi_j\colon\CC(z)\to\CC(z)$$
which again satisfies the above equality.

\begin{theorem} Let $M=\oplus_nHom^{SL(2)}(S^n(S^d\CC^2),S^e\CC^2)$.
The Hilbert series is
$$F_M(z)=\sum_{0\le j<d/2}(-1)^j\Phi_{d-2j}\left((1-z^2)z^e\gamma_{d,j}(z)\right)$$
where $$\gamma_{d,j}(z)=\frac{z^{j(j+1)}}{\prod_{k=1}^j(1-z^{2k})\prod_{l=1}^{d-j}(1-z^{2l})}.$$
\end{theorem}
\begin{proof} \cite{Br}
\end{proof}

\begin{corollary}\label{finalhilbertcov}
$$F_d(z,w)=\sum_{0\le j<d/2}(-1)^j\Phi_{d-2j}\left(\frac{1-z^2}{1-zw}\gamma_{d,j}(z)\right).$$
\end{corollary}

The formula of Corollary \ref{finalhilbertcov} can be used to compute the Hilbert series
of covariant rings of binary forms of small degree.

We reproduce for completeness some of the results, although in the following we will reprove
some of them with tools from representation theory (like in \S \ref{molienelementary}), when available.

We have
{\small\begin{equation*}\begin{split} F_3(z,0)=\Phi_3\left((1-z^2)\gamma_{3,0}(z)\right)-\Phi_1\left((1-z^2)\gamma_{3,1}(z)\right)=\\
\Phi_3\left(\frac{1}{(1-z^4)(1-z^6)}\right)-\Phi_1\left(\frac{z^2}{(1-z^2)(1-z^4)}\right)\end{split}\end{equation*}}
and get
{\small \begin{equation*}\begin{split}F_3(z,0)=\frac{1}{1-z^2}\Phi_3\left(\frac{1}{1-z^4}\right)-\frac{z^2}{(1-z^2)(1-z^4)}=\\
\frac{1}{(1-z^2)(1-z^4)}-\frac{z^2}{(1-z^2)(1-z^4)}=\frac{1}{1-z^4}.\end{split}\end{equation*}}

Moreover

$$F_4(z,0)=\Phi_4\left((1-z^2)\gamma_{4,0}(z)\right)-\Phi_2\left((1-z^2)\gamma_{4,1}(z)\right)=$$
$$\Phi_4\left(\frac{1}{(1-z^4)(1-z^6)(1-z^8)}\right)-\Phi_2\left(\frac{z^2}{(1-z^2)(1-z^4)(1-z^6)}\right)$$
and get 

$$F_4(z,0)=\frac{1}{(1-z)(1-z^2)}\Phi_4\left(\frac{1}{1-z^6}\right)-\frac{z}{(1-z)(1-z^2)(1-z^3)}=$$
$$\frac{1}{(1-z)(1-z^2)(1-z^3)}-\frac{z}{(1-z)(1-z^2)(1-z^3)}=\frac{1}{(1-z^2)(1-z^3)}.$$

For the Hilbert series of covariants
we get

\begin{equation}\label{hilbcovcubic}F_3(z,w)=\frac{1+z^3w^3}{(1-z^4)(1-zw^3)(1-z^2w^2)},\end{equation}

$$F_4(z,w)=\frac{1+z^3w^6}{(1-z^2)(1-z^3)(1-zw^4)(1-z^2w^4)}.$$

\subsection{Covariant ring of binary cubics }
Let $f=a_0x^3+3a_1x^2y+3a_2xy^2+a_3y^3$.
From the series (\ref{hilbcovcubic}), computed also in \cite{Stu} (4.2.17), we can guess the table of covariants, namely
 \begin{table}[H]
\[ \qquad \qquad \text{order} \] 
\[ \text{degree} \; \; 
\begin{tabular}{|c||c|c|c|c|c|c|c|c|c|} \hline 
{} &  0 &  1 &  2 &  3  \\ \hline \hline 
1  & {} & {} & {} & {$1$}  \\ \hline 
2  & {} & {} &  $1$ & {}  \\ \hline 
3  & {} & {} & {} &{$1$} \\ \hline 
4  &  $1$ & {} & {} & {}  \\ \hline 

\end{tabular} \] 
\end{table}

Indeed we know some covariants exactly of the expected degrees, namely

 \begin{table}[H]
\[ \qquad \qquad \text{order} \] 
\[ \text{degree} \; \;  
\begin{tabular}{|c||c|c|c|c|c|c|c|c|c|} \hline 
{} &  0 &  1 &  2 &  3  \\ \hline \hline 
1  & {} & {} & {} & {$f$}  \\ \hline 
2  & {} & {} &  $H$ & {}  \\ \hline 
3  & {} & {} & {} &{$Q$} \\ \hline 
4  &  $\Delta$ & {} & {} & {}  \\ \hline 

\end{tabular} \] 
\end{table}

where $\Delta=[12]^2[13][24][34]^2$ is the discriminant,
$H=(f,f)_2=[12]^2[1x][2x]$ is the Hessian, $Q=(f,H)_1$.
The Hessian vanishes identically on the twisted cubic,
indeed its plain expression is (we divide by $36$ for convenience)

\begin{equation}\label{hessiantwisted}H=\frac{1}{36}\left(f_{00}f_{22}-f_{12}^2\right)=(a_1a_3-a_2^2)x^2+(a_0a_3-a_1a_2)xy+
(a_0a_2-a_1^2)y^2.\end{equation}

The condition $H=0$ represents the two points such that,
together with $f=0$, make a equianharmonic 4ple
(see \S \ref{binaryquarticsection}).

The condition $Q=0$ represents the three points such that,
together with $f=0$, make a harmonic 4ple
(see \S \ref{binaryquarticsection}). 

To fix the scalars, we pose

$\Delta=4(a_0a_2-a_1^2)(a_1a_3-a_2^2)-(a_0a_3-a_1a_2)^2.$

Then there is a unique syzygy which is
$36H^3+9\Delta f^2+Q^2=0$ ,
note that $Q^2$ is expressed rationally by the others.
Indeed the Hilbert series is

$$F_3(z,w)=\frac{1+z^3w^3}{(1-z^4)(1-zw^3)(1-z^2w^2)}.$$

The three factors at the denominator correspond respectively to
$\Delta$, $f$, $H$. The Hilbert series says that the subalgebra generated by these covariants
coincides with the covariant ring.
In particular the invariant ring is free and it is generated 
by $\Delta$. We have proved

\begin{theorem}\label{covcubic}
The covariants of a binary cubic $f$ are generated by
$\Delta$, $H$, $f$ and $Q$, satisfying
the single relation
$$36H^3+9\Delta f^2+Q^2=0.$$ 
\end{theorem}

We recommend reading  lecture XXI in \cite{Hilb} 1.8, where the solution of the cubic equation
is obtained from the relation of Theorem \ref{covcubic}.

\subsection{Apolarity and transvectants}
\label{apolaritysection}

Let $V$ be a vector space with basis $x_0,\ldots, x_n$
and let $$R=K[x_0,\ldots, x_n]=\oplus_{m=0}^{\infty}S^mV$$ be the polynomial ring. Let's recall that the dual ring $$R^{\vee}=
K[\partial_0,\ldots, \partial_n]=\oplus_{m=0}^{\infty}S^mV^{\vee}$$
can be identified with ring of polynomial differential operators, where $\partial_i=\frac{\partial}{\partial x_i}$.

The action of $R^{\vee}$ over $R$ was classically called as {\it apolarity}.
In particular for any integers $e\ge d\ge 0$ we have the apolar linear maps
$$S^dV^{\vee}\otimes S^eV\to S^{e-d}V.$$

When $\dim U=2$, that is for polynomials over a projective line,
the a\-po\-la\-rity is well defined for $f$, $g$ both in $S^dU$.
This is due to the canonical isomorphism $U\simeq U^{\vee}\otimes\wedge^2U$.

This allows to make explicit  the $SL(2)$-decompositions
$$S^d\CC^2\otimes S^e\CC^2=\oplus_{i=0}^{\min(d,e)}S^{d+e-2i}\CC^2,\qquad
S^2(S^d\CC^2)=\oplus_{i=0}^{\lfloor\frac d2\rfloor}S^{2d-4i}\CC^2.$$
Let $(x_0,x_1)$ be coordinates on $U$.
If $f=(a_0x_0+a_1x_1)^d$ and $g=(b_0x_0+b_1x_1)^d$,
then the contraction between $f$ and $g$ is seen to be
proportional to $(a_0b_1-a_1b_0)^d$.
This computation extends by linearity to any pair $f, g\in S^dU$,
because any polynomial can be expressed as a sum of $d$-th powers.
The resulting formula for $f=\sum_{i=0}^d{d\choose i}f_ix^{d-i}y^i$ and
$g=\sum_{i=0}^d{d\choose i}g_ix^{d-i}y^i$ is that $f$ is apolar to $g$ if and only if
$$\sum_{i=0}^d(-1)^i{d\choose i}f_ig_{d-i}=0.$$

In order to prove this formula, by linearity and by Theorem \ref{veronesespan} it is again sufficient to assume $=u^d$ and $g=v^d$.
In particular 
\begin{proposition}\label{divides}Let $p, l^d\in S^dU$.
$p$ is apolar to $l^d$ if and only if $l$ divides $p$.
\end{proposition}

\begin{proposition}
If $d$ is odd, any $f\in S^dU$ is apolar to itself. Apolarity defines a skew nondegenerate form in $\wedge^2(S^d\CC^2)$.

If $d$ is even, the condition that $f$ is apolar to itself defines
a smooth quadric in $\P S^dU$. Apolarity defines a symmetric nondegenerate form in $S^2(S^d\CC^2)$.
\end{proposition}

From the geometric point of view,
let $f=l_1\ldots l_d\in S^d\CC^2$ be the decomposition in linear factors
and denote by $[f]$ the corresponding point in $\P S^d\CC^2$.
 Let $P(f)$ be
the hyperplane spanned by $\{l_1^d,\ldots l_d^d\}$.
$f$ is apolar to itself if and only if $[f]\in P(f)$. 

The natural algebraic generalization of the apolarity
is given by transvection.

If $u^d\in S^dU$, $v^e\in S^eU$ we define the $n$-th transvectant\footnote{translated {\it ``scorrimento''} in \cite{EC}, from the original German {\it ``ueber-schiebung.''}} to be 
$$(u^d,v^e)_n=u^{d-n}v^{e-n}[uv]^n\in S^{d+e-2n}\CC^2.$$

If $f\in S^dU$, $g\in S^eU$ and $0\le n\le \min(d,e)$,
the general $SL(2)$-invariant formula is
$$(f,g)_n=\sum_{i=0}^n(-1)^i{n\choose i}\frac{\partial f}{\partial x^{n-i}\partial y^i}
\frac{\partial g}{\partial x^{i}\partial y^{n-i}}.$$

 Note that $(f,g)_1$
is the Jacobian, while $(f,f)_2$ is the Hessian.

For $f\in S^4\CC^2$ we can express the invariants $I$ and $J$ introduced in \S \ref{veronesesubsection}
in terms of transvectants. Indeed it is easy to check that
$I=(f,f)_2$, $J=(f,(f,f)_2)_4$. This gives a recipe to compute the expressions of $I$ and $J$,
that can be extended to other situations. A Theorem of Gordan states that all invariants
of binary forms can be expressed by the iterate application of transvectants. We will not pursue this approach here.
Transvectants are close to symbolic representation of \S \ref{symbolicrep}, see \cite{Olv} example 6.26.

\subsection{Invariant ring of binary quartics}
\label{binaryquarticsection}

A polynomial $f=f_0x^4+4f_1x^3y+6f_2x^2y^2+4f_3xy^3+f_4y^4\in S^4U$ is called {\it equianharmonic} if its apolar to itself.
So $f$ is equianharmonic if and only if
$$f_0f_4-4f_1f_3+3f_2^2=0,$$
which is the expression for the classical invariant $I$ of binary quartics, see (\ref{firstI}).

The symbolic expression is
$[12]^4$
or
\vskip 0.5cm

$\young(1111,2222)$.
\vskip 0.5cm

The invariant $J$ (see (\ref{firstJ})) has the symbolic expression 
$[12]^2[13]^2[23]^2$
or
\vskip 0.5cm

$\young(111122,223333)$.

It is equal to the determinant

$$J=\left[\begin{array}{ccc}
f_0&f_1&f_2\\
f_1&f_2&f_3\\
f_2&f_3&f_4
\end{array}\right].$$

A binary quartic with vanishing $J$ is called {\it harmonic}.
A binary quartic $f$ is harmonic if and only if
$f$ has an apolar quadratic form, if and only if $f$ is sum of two
$4$th powers, instead of the three which are needed for the general $f$.

\begin{theorem}\label{IJgenerate}
Let $d=4$, let $I_{g,2g,4}$ be the dimension of the space of invariants of degree $g$
of the binary quartic. Then
$$\sum_{g=0}^{\infty}I_{g,2g,4}x^g=\frac{1}{(1-x^2)(1-x^3)}.$$
The ring of invariants is freely generated by $I, J$, that is
$\oplus_m S^m(S^4(\CC^2))^{SL(2)}=\CC[I,J]$.
\end{theorem}
\begin{proof} The series has been shown (without proof) in \S \ref{hilbertseries}.
We sketch the proof by following again \cite{Hilb},
in a way similar to the proof of Corollary \ref{corobinarycubic}.
We will see a different proof in Theorem \ref{IJgenerate2}.

 We have from Corollary \ref{cscoro} that
$I_{g,2g,4}$ is the degree $2g$ coefficient of
$$\frac{(1-x^{4+1})\ldots (1-x^{4+g})}{(1-x^2)\ldots (1-x^g)}=
\frac{(1-x^{g+1})(1-x^{g+2})(1-x^{g+3})(1-x^{g+4})}{(1-x^2)(1-x^3)(1-x^4)}$$
and we write it as
$$\left\{\frac{(1-x^{g+1})(1-x^{g+2})(1-x^{g+3})(1-x^{g+4})}{(1-x^2)(1-x^3)(1-x^4)}\right\}_{2g}.$$

We can remove the terms which do not change the coefficient of $x^{2g}$, so getting

{\footnotesize
$$\left\{\frac{(1-x^{g+1}-x^{g+2}-x^{g+3}-x^{g+4})}{(1-x^2)(1-x^3)(1-x^4)}\right\}_{2g}=\left\{\frac{1}{(1-x^2)(1-x^3)(1-x^4)}\right\}_{2g}-\left\{\frac{x(1+x+x^{2}+x^{3})}{(1-x^2)(1-x^3)(1-x^4)}\right\}_{g}=
$$

(by using that $(1-x^4)=(1+x+x^2+x^3)(1-x)$ )

$$=\left\{\frac{1}{(1-x^2)(1-x^3)(1-x^4)}\right\}_{2g}-\left\{\frac{x}{(1-x)(1-x^2)(1-x^3)}\right\}_{g}=$$

$$=\left\{\frac{1}{(1-x^2)(1-x^3)(1-x^4)}-\frac{x^2}{(1-x^2)(1-x^4)(1-x^6)}\right\}_{2g}=
\left\{\frac{1+x^3-x^2}{(1-x^2)(1-x^4)(1-x^6)}\right\}_{2g}=
$$

Since in the denominator only even powers appear, we can remove $x^3$ from the numerator and we get

$$=\left\{\frac{1}{(1-x^4)(1-x^6)}\right\}_{2g}=
\left\{\frac{1}{(1-x^2)(1-x^3)}\right\}_{g},
$$}
as we wanted.

For the last assertion, consider  the subring 

$\CC[I,J]\subset\oplus_m S^m(S^4(\CC^2))^{SL(2)}.$

In order to prove the equality, it is enough to show that the two rings have the same Hilbert series.
We have proved that the Hilbert series of the invariant ring
$\oplus_m S^m(S^4(\CC^2))^{SL(2)}$ is $\frac{1}{(1-x^2)(1-x^3)}$,
which is the  Hilbert series of a ring with two algebraically independent generators of degree $2$ and $3$.
The invariants $I$ and $J$  have respectively degree $2$ and $3$.
So it is enough to prove that $I$, $J$ are algebraically independent.

Assume we have a relation in degree $g$ between $I$, $J$,
that is a relation $\sum_{2k+3l=g}c_{kl}I^{k}J^l=0$
which hold identically for any $a_0,\ldots, a_4$.
Since $I(a_0,0,0,a_3,a_4)=a_0a_4$,
$J(a_0,0,0,a_3,a_4)=-a_0a_3^2$ we get

$\sum_{2k+3l=g}c_{kl}(-1)^la_0^{k+l}a_3^{2l}a_4^k=0$.

All $k$, as well as all $l$,  are distinct, because every $k$ determines uniquely $l$ and conversely. It follows get $c_{kl}=0$ $\forall k, l$,
as we wanted.

\end{proof}

Note that in degree $6$ there are two independent invariants, $I^3$ and $J^2$
given respectively by
$$\young(111122223333,444455556666)$$
and

$$\young(111122444455,223333556666).$$

It is not trivial to show directly that all the semistandard  $2\times 12$ tableau functions give, under the Pl\"ucker relations, a linear combination of these two.

Indeed the Hilbert series we computed in \S \ref{hilbertseries} is 

$$F_4(z,w)=\frac{1+z^3w^6}{(1-z^2)(1-z^3)(1-zw^4)(1-z^2w^4)}$$

The factors at the denominator correspond respectively at
$I$, $J$, $f$, $H$.

The syzygy represents $Q^2$ as a combination of $Jf^3$, $If^2H$, $H^3$.

$I^3/J^2$ is an absolute invariant, $\Delta=I^3-27J^2$ is the discriminant.

The table of covariants is the following,
where in the first column we read $I$, $J$, in the column labeled with $4$ we find respectively $f$, $H$
and in the last column we find $Q=(f,H)_1$.

\[ \qquad \qquad \text{order} \] 
\[ \text{degree} \; \; 
\begin{tabular}{|c||c|c|c|c|c|c|c|c|c|} \hline 
{} &  0 &  1 &  2 &  3 & 4  &  5 &  6  \\ \hline \hline 
1  & {} & {} & {} & {} & {1} &  {} & {}  \\ \hline 
2  & {1} & {} &  {} & {} & {1} & {} &  {}  \\ \hline 
3  & {1} & {} & {} &  {} & {} & {}  & {1}  \\ \hline 

\end{tabular} \]

The vanishing of $Q$ as covariant (meaning that its seven cubic generators
all vanish, express the fact that the quartic is a square). So they give the
equations of a classically well known surface, which is the projection in $\P^4$
of the Veronese surface from a general point in $\P^5$.

For an extension to binary forms of any degree see \cite{AC}.

$Q=0$ represents the three pairs of double points for the three involutions
which leave $f$ invariant.

\subsection{$SL(2)$ as symplectic group. Symplectic construction of invariants for binary quartics.}\label{gherardelli}
The content of this subsection was suggested by Francesco Gherardelli
several years ago. 
I report here with the hope that somebody could be interested and take this idea further.

The starting point is that $SL(2)$ can be seen as the symplectic group,
preserving $J=
\left[\begin{array}{cc}0&1\\-1&0\\
\end{array}\right]$.
Many invariants for binary forms of even degree can be computed in a symplectic setting. In the example of binary quartics we have
$a_0x^4+4a_1x^3y+6a_2x^2y^2+4a_3xy^3+a_4y^4=
\left[\begin{array}{ccc}x^2&2xy&y^2\end{array}\right]
\left[\begin{array}{ccc}a_0&a_1&a_2\\
a_1&a_2&a_3\\
a_2&a_3&a_4\end{array}\right]
\left[\begin{array}{c}
x^2\\
2xy\\
y^2\end{array}\right]$

Define for $g=\left[\begin{array}{cc}\alpha&\beta\\
\gamma&\delta\end{array}\right]$
$\left[\begin{array}{c}
x\\
y
\end{array}\right]=g
\left[\begin{array}{c}
x'\\
y'
\end{array}\right]
$
and note that, setting

$f(g)=S^2g=
\left[\begin{array}{ccc}\alpha^2&\alpha\beta&\beta^2\\
2\alpha\gamma&\alpha\delta+\beta\gamma&2\beta\delta\\
\gamma^2&\gamma\delta&\delta^2\end{array}\right],$
 we get

$\left[\begin{array}{c}
x^2\\
2xy\\
y^2\end{array}\right]
=f(g)
\left[\begin{array}{c}
x'^2\\
2x'y'\\
y'^2\end{array}\right].$

Hence
$f(g)^t\left[\begin{array}{ccc}
0&0&1\\
0&-\frac{1}{2}&0\\
1&0&0\end{array}\right]f(g)=\left[\begin{array}{ccc}
0&0&1\\
0&-\frac{1}{2}&0\\
1&0&0\end{array}\right]$

and we get

$$\det\left[\begin{array}{ccc}a_0&a_1&a_2+t\\
a_1&a_2-\frac{t}{2}&a_3\\
a_2+t&a_3&a_4\end{array}\right]=\frac{t^3}{2}+t\cdot I(a_i)+J(a_i).$$
The beauty of this formula is that the two invariants $I$ and $J$,
coming respectively from (\ref{firstI}) and (\ref{firstJ}), are defined at once. What
happens for higher degree?

\begin{remark}
The reader can find something similar at the end of Procesi's book \cite{Pr}.
In \cite{Pr}, $S^d\CC^2$ is considered inside $S^{d-2}\CC^2\otimes S^{d-2}\CC^2
\simeq End(S^{d-2}\CC^2)$,
in the case $d=4k$. The coefficients of this characteristic polynomial are
conjecturally the generators of the invariant ring for $S^d\CC^2$.
\end{remark}

\subsection{The cubic invariant for plane quartics}

This is the easiest example of invariant of ternary forms defined by the symbolic representation of \S \ref{symbolicrep}. 
Let $(x_0,x_1,x_2)$ be coordinates on a $3$-dimensional complex space $V$ and
$(y_0,y_1,y_2)$ be coordinates on  $V^{\vee}$. 
Let $$f(x_0,x_1,x_2)=\sum_{i+j+k=4}\frac{4!}{i!j!k!}f_{ijk}x_0^ix_1^jx_2^k\in S^4V$$
be the equation of a plane quartic curve on $\P(V)$.
By Corollary \ref{n+1divides}, all invariants of $f$ have degree which is multiple of $3$.
The invariant of smallest degree has degree $3$ and it is defined by the tableau

$$T=\begin{matrix}\young(1111,2222,3333)\end{matrix}.$$

We denote $A=F_T$.
The trilinear form $A(f,g,h)$, for $f, g, h\in S^4V$ satisfies $A\left(a^4,b^4,c^4\right)=(a\wedge b\wedge c)^4$,
or, expanding the linear forms
{\scriptsize
$$A\left((a_0x_0+a_1x_1+a_2x_2)^4,(b_0x_0+b_1x_1+b_2x_2)^4,(c_0x_0+c_1x_1+c_2x_2)^4\right)=
\left|\begin{array}{ccc}
a_0&a_1&a_2\\
b_0&b_1&b_2\\
c_0&c_1&c_2\\
\end{array}\right|^4.$$
}

The explicit expression of the cubic invariant $A$ can be found
at art. 293 of Salmon's book\cite{Sal}, it can be checked with the M2 script of \ref{symbolicrep} and it is the sum of the following $23$ summands
(we denote $A(f)$ for $A(f,f,f)$)
{\scriptsize
\begin{gather}
\label{cubicinvariant}
A(f)=f_{400}f_{040}f_{004}+3(f_{220}^2f_{004}+f_{202}^2f_{040}+f_{400}f_{022}^2)
+12(f_{202}f_{121}^2+f_{220}f_{112}^2+f_{022}f_{211}^2)+6f_{220}f_{202}f_{022}+\nonumber\\
-4(f_{301}f_{103}f_{040}+f_{400}f_{031}f_{013}+f_{310}f_{130}f_{004})
+4(f_{310}f_{103}f_{031}+f_{301}f_{130}f_{013})+\nonumber\\
-12(f_{202}f_{130}f_{112}+f_{220}f_{121}f_{103}+f_{211}f_{202}f_{031}+f_{301}f_{121}f_{022}+
f_{310}f_{112}f_{022}+f_{220}f_{211}f_{013}+f_{211}f_{121}f_{112})+\nonumber\\
+12(f_{310}f_{121}f_{013}+f_{211}f_{130}f_{103}+f_{301}f_{112}f_{031}).
\end{gather}
}

This expression for the cubic invariant can be found
also by applying the differential operators $D_1$, $D_2$ defined in \S \ref{liealgebraaction} 
to the space of isobaric monomials of degree $3$ and total weight
$[4,4,4]$. Indeed there are $23$ such monomials
and the only linear combination of these monomials which is killed
by the differential operators 
is a scalar multiple of the cubic invariant.

The differential operators are analogous to (\ref{lietorus})
and are

$$
\sum_{i_0+\ldots +i_2=4} i_1f_{i_0+1,i_1-1, i_2}
\frac{\partial I}{\partial f_{i_0\ldots i_2}}=0, 
$$

$$
\sum_{i_0+\ldots +i_2=4} i_2f_{i_0,i_1+1, i_2-1}
\frac{\partial I}{\partial f_{i_0\ldots i_2}}=0 .
$$

It is enough to impose $19+19=38$ conditions to the $23$-dimensional space.
Compare with \cite{Stu} example 4.5.3 where similar computations were performed for the ternary cubic. 

Note that given $f, g\in S^4V$, the equation $A(f,g,*)=0$  defines an element in the dual space $S^4V^{\vee}$,
possibly vanishing.

\begin{proposition}\label{alternativedef}
(i) $A(f,g,l^4)=0$ if and only if the restrictions $f_{|l}$, $g_{|l}$
to the line $l=0$ are apolar.

(ii) Let $A(f,g,*)=H$. We have  $A(f,g,l^4)=0$ if and only if $H(l)=0$ .
\end{proposition}
\begin{proof} To prove (i) , consider 
$f=(\sum_{i=0}^2a_ix_i)^4$, $g=(\sum_{i=0}^2b_ix_i)^4$, $l=x_2$.
Then $$A\left(f,g,l^4\right)=
\left|\begin{array}{ccc}
a_0&a_1&a_2\\
b_0&b_1&b_2\\
0&0&1\\
\end{array}\right|^4=
\left|\begin{array}{cc}
a_0&a_1\\
b_0&b_1\\
\end{array}\right|^4=f_{|l}\cdot g_{|l}.$$
This formula extends by linearity to any $f$, $g$.

(ii) follows because $H(l)=H\cdot l^4$.
\end{proof}

\begin{remark}\label{l1l2} Let $l_1$, $l_2$ be two lines.
$A\left(l_1^4,l_2^4,f\right)=0$
gives the condition that $f$ passes through the intersection point
$l_1=l_2=0$.
\end{remark}

Note also from Prop. \ref{alternativedef} that $A(f,f,l^4)=0$ if and only if $f$ cuts $l$ in an
equianharmonic $4$-tuple.
The quartic curve $A(f,f,*)$ in the dual space  is called the {\it equianharmonic envelope} of $f$.
It is sometimes called a ``contravariant''.
The ``transfer principle of Clebsch'' says that from
the symbolic expression $(ab)(cd)\ldots$ for a invariant
it follows $(ab*)(cd*)\ldots$
for the (envelope) contravariant.

This gives the classical geometric interpretation of the cubic invariant for plane quartics.
The condition $A(f,f,f)=0$ means that $f$ is apolar with its own
equianharmonic envelope (see \cite{Cia}), note that it gives a solution to Exercise (1) in the last page
of \cite{Stu}.

\subsection{The Aronhold invariant as a pfaffian}

Another classical invariant of ternary forms is the Aronhold invariant for plane cubics, it is defined by the tableau

$$T=\begin{matrix}\young(1112,2233,3444).\end{matrix}$$

We denote by $Ar$ the corresponding multilinear form $G_T$ and also its symmetrization $F_T$. We get

$$Ar(x^3,y^3,z^3,w^3)=(x\wedge y\wedge z)(x\wedge y\wedge w)(x\wedge z\wedge w)(y\wedge z\wedge w).$$

The expression of the Aronhold invariant $Ar$ has $25$ monomials and it can be found in \cite{Stu} Prop. 4.4.7 or in 
\cite{DoK} (5.13.1). The Aronhold invariant is a ``lucky'' case, were the geometric interpretation follows easily from the symbolic notation.
It is not a surprise that was one of the first examples leading Aronhold to the symbolic notation.
If a plane cubic $f$ is sum of three cubes (namely, it is $SL(3)$-equivalent to the Fermat cubic $f=x_0^3+x_1^3+x_2^3$)
we have $Ar(f)=0$. Indeed
$Ar(f,f,f,f)$ splits as a sum of $Ar(x_{i_0}^3,x_{i_1}^3,x_{i_2}^3,x_{i_3}^3)$ where $i_k\in\{0,1,2\}$,
so that $\{i_0,i_1,i_2,i_3\}$ contains at least a repetition, in such a way that all summands contributing to
 $Ar(f,f,f,f)$ vanish.

Let $W$ be a vector space of dimension $3$. In particular 
$\Gamma^{2,1}W=\textrm{ad\ }W$ is self-dual and it has dimension $8$. 
We refer to \cite{Ott} Theor. 1.2
for the proof of the following result, see also \cite{LO} example 1.2.1. 
\begin{theorem}\label{aronhold}
For any $\phi\in S^3W$, let
$A_{\phi}\colon\Gamma^{2,1}W\to\Gamma^{2,1}W$ be the $SL(V)$-invariant contraction operator.
Then  $A_{\phi}$ is skew-symmetric and the pfaffian $\textrm{Pf\ } A_{\phi}$
 is the equation of $\sigma_3(\P {{}}(W),\O(3))$,
i.e. it is the Aronhold invariant $Ar$.
\end{theorem}

The corresponding picture is
$$\yng(2,1)\otimes\young(***)\quad\to\quad\young(\ \ *,\ *,*)\simeq\ \yng(2,1)$$

 The $SL(W)$-module $\textrm{End}_0~W$
 consists of the subspace of endomorphisms of $W$ with zero trace.
  The contraction 
  $$A_{\phi}\colon\textrm{End}_0~W\to\textrm{End}_0~W$$  in the case $\phi=v^3$ satisfies

$$A_{v^3}(M)(w)=\left(M(v)\wedge v\wedge w\right)v$$

 where $M\in \textrm{End}~W$, $w\in W$.

\begin{remark}\label{scorza}  We recall from \cite{DoK} the definition of the Scorza map.
For any plane quartic $F$ and any point $x\in{\bf P}(W)$
we consider the polar cubic  $P_x(F)$. Then $Ar(P_x(F))$ is a quartic in the variable $x$
which we denote by $S(F)$. The rational map $S\colon{\bf P}(S^4W)\dashrightarrow {\bf P}(S^4W)$
 is called the Scorza map. Our description of the Aronhold invariant shows
that $S(F)$ is defined as the degeneracy locus of a skew-symmetric morphism  on ${\bf P}(W)$
$$\O(-2)^8\rig{f} \O(-1)^8.$$
It is easy to check  (see \cite{Be}) that $\textrm{Coker\ }f=E$ is a rank two vector bundle
over $S(F)$ such that $c_1(E)=K_{S(F)}$. 

I owe to A. Buckley the claim that from $E$ it is possible to recover the even theta characteristic 
$\theta$ on $S(F)$ defined in \cite[(7.7)]{DoK} (see also next section), by following a construction by C. Pauly
\cite{Pau} . 
There are exactly eight maximal line subbundles $P_i$
of $E$
of maximal degree equal to $1$ such that $h^0(E\otimes P_i^{\vee})>0$.
These eight line bundles are related by the equality
(Lemma 4.2 in \cite{Pau})
$$\otimes_{i=1}^8P_i=K_{S(F)}^2.$$
The construction in \cite{Pau} \S 4.2  gives a net of quadrics in the following way.
For the general stable $E$ as ours, there exists a unique
stable bundle $E'$ with $c_1(E')=\O$ such that
$h^0(E\otimes E')$ has the maximal value $4$.
The extensions 
$$0\rig{}P_i\rig{} E\rig{}K_{S(F)}\otimes P_i^{-1}\rig{}0$$
$$0\rig{}P_i^{-1}\rig{} E'\rig{} P_i\rig{}0$$
define eight sections in $Hom(E',E)\simeq E\otimes E'$ as the compositions

$$E'\rig{} P_i\rig{} E$$
which give
eight points in $\P H^0(E\otimes E')$.
These eight points are the base locus for a net of quadrics,
which gives a symmetric representation of the original quartic 
curve $S(F)$ and then defines a even theta characteristic.

\end{remark}

\subsection{Clebsch and L\"uroth quartics. Theta characteristics.}
\label{lurothsection}
A plane quartic $f\in S^4V$ is called {\it Clebsch} if it has an apolar conic,
that is if there exists a nonzero $q\in S^2V^{\vee}$ such that $q\cdot f=0$.

One defines, for any $f\in S^4V$, the catalecticant map
$C_f\colon S^2V^{\vee}\to S^2V$
which is the contraction by $f$.
The equation of the Clebsch invariant
is easily seen as the determinant of $C_f$, that is we have(\cite{DoK}, example (2.7))

\begin{theorem}[Clebsch]
A plane quartic $f$ is Clebsch if and only if $\det C_f=0$.
The conics which are apolar to $f$ are the elements of $\ker~C_f$.
\end{theorem}

It follows (\cite{D}, Lemma 6.3.22) that the general Clebsch quartic  can be expressed
as a sum of five $4$-th powers, that is
\begin{equation}\label{clebschexpression}
f=\sum_{i=0}^4l_i^4.
\end{equation}
A general Clebsch quartic $f$ can be expressed 
as a sum of five $4$-th powers in $\infty^1$
many ways. Precisely the $5$ lines $l_i$ belong to a unique
smooth conic $Q$ in the dual plane, which is apolar to $f$ and it is found  as the generator
of  $\ker~C_f$. Equivalently, the $5$ lines $l_i$ are tangent to a unique conic, which is the dual conic of $Q$.

We recall that a {\it theta characteristic} on a general plane quartic $f$ is
a line bundle $\theta$ on $f$ such that $\theta^2$ is the canonical bundle.
Hence $\deg\theta=2$. There are $64$ theta characteristic on $f$.
If the curve is general, every bitangent is tangent in two distinct points $P_1$ and $P_2$, and the divisor $P_1+P_2$
defines a theta characteristic $\theta$ such that $h^0(\theta)=1$,
these are called odd theta characteristic and there are $28$ of them.
The remaining $36$ theta characteristic $\theta$ are called even
and they satisfy $h^0(\theta)=0$ (for any curve, even theta characteristics have even $h^0(\theta)$) .

The {\it Scorza map} is the rational map from $\P^{14}=\P(S^4V)$ to itself which associates
to  a quartic $f$ the quartic $S(f)=\{x\in\P(V)|Ar(P_x(f))=0\}$,
where $P_x(f)$ is the cubic polar to $f$ at $x$ and $Ar$ is the Aronhold invariant \cite{Sc, DoK, D}.
It is well known that it is a $36:1$ map.
Indeed the curve $S(f)$ is equipped with an even theta characteristic.
For a general quartic curve, its $36$ inverse images through the Scorza map
give all the $36$ even theta characteristic.

A {\it L\"uroth quartic} is a 
 plane quartic containing the ten vertices of a complete 
{\it pentalateral}, or the limit of such curves.

If $l_i$ for $i=0,\ldots, 4$ are the lines of the pentalateral,
we may consider as divisor (of degree $4$) over the curve.
Then $l_0+\ldots +l_4$ consists of $10$ double points,
the meeting points of the $5$ lines.
Let $P_1+\ldots +P_{10}$ the corresponding reduced divisor of degree $10$.
Then $P_1+\ldots +P_{10}=2H+\theta$
where $H$ is the hyperplane divisor and $\theta$ is a even theta characteristic,
which is called the pentalateral theta.
The pentalateral theta was called pentagonal theta in \cite{DoK},
and it coincides with \cite{D}, Definition 6.3.30 (see the comments thereafter).

The following result is classical\cite{Sc},  for a modern proof see \cite{DoK, D}.

\begin{proposition}\label{clebschluroth}
Let $f$ be a Clebsch quartic with apolar conic $Q$,
 then $S(f)$ is a L\"uroth quartic equipped with the pentalateral theta corresponding to $Q$.
\end{proposition}

\begin{proof} Let $f=\sum_{i=1}^5l_i^4$. Let $l_p$ and $l_q$ be two lines in the pentalateral
and let $x_{pq}=l_p\cap l_q$. Then $$P_{x_{pq}}(f)=\sum_{i=1}^5P_{x_{pq}}(l_i^4)=\sum_{i=1}^54l_i^3P_{x_{pq}}(l_i).$$
In the above sum at most three summands survive, because the ones with $i=p, q$ are killed.
Then $P_{x_{pq}}(f)$ is a Fermat cubic and $Ar(P_{x_{pq}}(f))=0$, hence ${x_{pq}}\in S(f)$.
It follows that $S(f)$ is inscribed in the pentalateral and it is L\"uroth. \end{proof}

The number of pentalateral theta on a general
L\"uroth quartic, called $\delta$, is equal to the degree
of the Scorza map when restricted to the hypersurface of Clebsch quartics.

Explicitly, if $f$ is Clebsch with equation
$$l_0^4+\ldots+l_4^4,$$
then $S(f)$ has equation
$$\sum_{i=0}^4k_i\prod_{j\neq i}l_j,$$
where $k_i=\prod_{p<q<r,i\notin\{p,q,r\}}|l_pl_ql_r|$
(see \cite{D}, Lemma 6.3.26)
so that $l_0,\ldots , l_4$ is a pentalateral inscribed in $S(f)$.
Note that the conic where the five lines which are the summands of $f$ are tangent,
is the same conic where the pentalateral inscribed in $S(f)$ is tangent.

\begin{remark}
The degree of L\"uroth invariant is $54$. This has been proved by Morley in 1919 \cite{Mor},
see \cite{OS1} for a review of his nice proof. To have the flavour of the complexity,
think that the space of monomials of degree $54$ in the $15$ variables $a_{ijk}$ with $i+j+k=4$
has dimension ${68\choose 14}\simeq 10^{14}$, while the space of isobaric ones has dimension $62, 422, 531, 333\simeq 10^{11}$.

The dimension of the space of invariants was computed first by Shioda \cite{Shi}, it is $1165$.
Ohno computes this space as dimension $1380$ with 215 relations.
This is reviewed by Basson, Lercier, Ritzenthaler, Sijsling in \cite{BLRS}, where the L\"uroth invariant has been described as linear combination of these monomials.
This computation allows to detect if a given plane quartic is L\"uroth. Moreover it disproves
a guess by Morley at the end of \cite{Mor} about the explicit form of the L\"uroth invariant. Still the existence
of a determinantal formula or other simple descriptions for the L\"uroth invariant is sought.

Recently, a determinantal description for the undulation invariant of degree $60$ has been found\cite{PS}.
It vanishes on quartic curves that have an undulation point, that is a line meeting the quartic
in a single point with multiplicity $4$.
A beautiful classical source about plane quartics is Ciani monograph \cite{Cia}.

\end{remark}

\begin{proposition}
Let $Y_{10}=\sigma_9(v_6(\P^2))$ be the determinantal hypersurface in $\P S^6\CC^3$
of sextics having a apolar cubic. The dual variety $Y_{10}^{\vee}$ is the variety of double cubics.
\end{proposition}

\begin{proof} By Remark \ref{dualsecvero}, the dual variety corresponds to the sextics
which are singular in $9$ points, hence they are double cubics.\end{proof}

\begin{proposition} The $3$-secant variety $\sigma_3(Y_{10}^{\vee})$
is the theta divisor, that is the locus of sextic curves which admit
an effective even theta-characteristic. Its degree is $83200$

\end{proposition}
\begin{proof} The sextics in the variety of $3$-secant to $Y_{10}$ can be written
as $A^2+B^2+C^2$ where $A$, $B$, $C$ are cubics. Since all plane conics are projectively equivalent,
they can be written as $AC-B^2=0$, that is as a $2\times 2$ symmetric determinant with
cubic entries. 
Write
$$\O(-3)^2\rig{M}\O^2$$
with $M$ symmetric. The cokernel is a effective theta-characteristic and conversely every
effective theta-characteristic arises in this way (see \cite{Be} and  remark 4 in \cite{BHORS}).
The computation of the degree is a nontrivial result proved in \cite{BHORS}.
\end{proof}

{\bf Question} What is the degree of the theta divisor for plane curves of degree $d$, that is
the locus of plane curves of degree $d$ which admit
an effective even theta-characteristic ?

\section{Invariants of points. Cremona equations for the cubic surface and invariants of six points.}
\label{sectionpoints}
\subsection{The two Fundamental Theorems for invariants of points.}\label{sectionFTpoints}
Given $p_1,\ldots p_d\in\P V=\P^{n}$, we can write their coordinates in a $(n+1)\times d$ matrix, writing the coordinates of $p_i$ in the $i$-th column.
The ring of polynomials over these coordinates $\CC[V\otimes\CC^d]=\oplus_mS^m(V\otimes\CC^d)$ has a natural multigraduation 
$\displaystyle \oplus_{m_1,\ldots, m_d} \left(S^{m_1}V\otimes\ldots\otimes S^{m_d}V\right)$ ,
where the coordinates of $p_i$ appear with total degree $m_i$.
The group $SL(n+1)$ acts on $\P V$, then it acts on the multigraded ring.
Classically, these rings have been studied in the ``democratic'' case when all $n_i$ are equal.
So the invariant ring to be studied was $\displaystyle \oplus_{m} {\underbrace{S^mV\otimes\ldots\otimes S^mV}_d}^{SL(n+1)}$. 
In these cases, there is the additional action of $\Sigma_d$ on the points and then on the invariant ring.
The $SL(n+1)\times \Sigma_d$-invariants were called ``rational'', while the ones invariant just for the smaller subgroup
$SL(n+1)\times Alt(d)$ were called ``irrational''.

After GIT has been developed, it has been understood that it is convenient to fix a weight (polarization)  $h=(h_1,\ldots, h_d)$,
so getting $$\CC[V\otimes\CC^d]_{(h)}=\oplus_p S^{ph_1}V\otimes\ldots \otimes S^{ph_d}V.$$

The invariant subring $\CC[V\otimes\CC^d]_{(h)}^{SL(n+1)}$ is called  
the invariant ring of $d$ ordered points on $\P V$ with respect to the weight $h$. When the weight $h$ is not specified,
it is understood that it is $h=1^d$.

The invariant subring $\CC[V\otimes\CC^d]^{SL(V)\times \Sigma_d}$ is called  
the invariant ring of $d$ unordered points on $\P V$.

In the case $n=1$, the invariant ring of $d$ unordered points on the line coincides with the invariant ring of binary forms of degree $d$.
This is clear associating to any binary form its scheme of roots.

For the convenience of the reader, we repeat with slight changes the construction
of Definition \ref{tableaufunction}.

\begin{definition}[From tableau to invariants of points]\label{tableaufunction2} \

Let $h_1+\ldots + h_d=(n+1)g$. For any tableau $T$ over a Young diagram of size
$(n+1)\times g$, filled with numbers from $1$ appearing $h_1$ times, until $d$ appearing $h_d$ times,
we denote by $G_T$ the multilinear function

$G_T\colon S^{h_1}V^{\vee}\times \ldots \times S^{h_d}V^{\vee}\to \CC$ defined by (compare with Definition \ref{tableaufunction})
$$G_T(x_1^{h_1},\ldots, x_d^{h_d})=\prod_{j=1}^d\left(x_{t(1,j)}\wedge\ldots\wedge x_{t(n+1,j)}\right).$$
$G_T$ is well defined by Theorem \ref{veronesespan} and it is $SL(V)$-invariant by Proposition \ref{gtinvariant}.
\end{definition}

The geometric meaning of the vanishing of $G_T$, where 
$$T=\begin{matrix}\young(1,2,3,4)\end{matrix},$$ is  that the corresponding points $x_1,\ldots x_4\in V^{\vee}$
 are dependent.

\begin{theorem}[1FT for ordered points]\label{1ftpoints}

The invariant ring $\CC[V\otimes\CC^d]_{(h)}^{SL(n+1)}$ of $d$ ordered points on $\P V$ 
with respect to $h$ is generated by tableau functions $G_T$
like in Definition \ref{tableaufunction2}, for tableau $T$ having weight multiple of $h$.
\end{theorem}

\begin{proof}

The decomposition (see Theorem \ref{cauchyidentity})
$S^p(V\otimes\CC^d)=\oplus_{\lambda}S^\lambda V\otimes S^\lambda \CC^m$,
where the sum is extended to all Young diagrams $\lambda$ with $|\lambda|=p$,
shows that $S^p(V\otimes\CC^d)^{SL(n+1)}= 0$ for $p$ which is not multiple of $(n+1)$
and $S^p(V\otimes\CC^d)^{SL(n+1)}=S^\mu\CC^d$
if $p=(n+1)g$ and $\mu$ is the Young diagram with $(n+1)$ rows and $g$ columns.

By Theorem \ref{fillglv}, $S^{\mu}\CC^d$ has a basis consisting of semistandard Young tableau $T$,
where the numbers $1,\ldots , d$ appear.

This basis has a natural multigraduation, depending on partitions $p=m_1+\ldots +m_d$, where in $T$ the number $1$ appears $m_1$ times,
$2$ appears $m_2$ times, until $d$ appears $m_d$ times.
Moreover, this basis fits with the other decomposition
$$S^p(V\otimes\CC^d)=\oplus \left(S^{m_1}V\otimes\ldots\otimes S^{m_d}V\right),$$
where the sum is extended to all the partitions with $d$ summands $p=m_1+\ldots m_d$ which induce

$S^p(V\otimes\CC^d)^{SL(n+1)}=\oplus \left(S^{m_1}V\otimes\ldots\otimes S^{m_d}V\right)^{SL(n+1)}.$

In other words, the summand $ \left(S^{m_1}V\otimes\ldots\otimes S^{m_d}V\right)^{SL(n+1)}$
has a natural basis consisting of semistandard Young tableau, consisting in the Young diagram
$(n+1)\times \frac{p}{n+1}$ filled with $1$ appearing $m_1$ times,
$2$ appearing $m_2$ times, until $d$ appearing $m_d$ times.

These semistandard tableau $T$ correspond to the multilinear function $G_T$ of Definition \ref{tableaufunction2}.
\end{proof}

\begin{corollary}[1FT for unordered points]
All invariants of $d$ unordered points are polynomials in
the tableau functions $G_T$, where in $T$ the numbers from $1$ to $d$ appear equally, which moreover are symmetric under permutation of points.

The invariant ring of $d$ unordered points on $\P V$ is isomorphic to
$$\oplus_m\left[S^d(S^mV)\right]^{SL(n+1)}.$$

\end{corollary}

\begin{remark} It is interesting to compare the invariant ring of $d$ unordered points with the invariant ring for forms in $S^dV$,
which is $\oplus_mS^m(S^dV)^{SL(n+1)}$. Note the swapping between $m$ and $d$.
Note also that on $P^1$ the swapping makes no difference, by Hermite reciprocity
(Corollary \ref{hermiterep}).
\end{remark}

\vskip 0.8cm

{\it Proof of Theorem \ref{1ftforms}, 1FT for forms }
Consider, in the proof of Theorem \ref{1ftpoints}, the case $m_1=\ldots =m_d=m$, so that $p=md$.

We get that $ (\underbrace{S^{m}V\otimes\ldots\otimes S^{m}V}_d)^{SL(n+1)}$
has a natural basis consisting of $G_T$, where $T$ is a semistandard Young tableau, filling the diagram
$(n+1)\times \frac{md}{n+1}$ with $1$ appearing $m$ times,
$2$ appearing $m$ times, until $d$ appearing $m$ times.

Considering the subspace of $\Sigma_d$-invariants, we get the space of symmetric multilinear functions
$F_T$, like in Definition \ref{symboliconstruction},
which indeed is $S^d(S^mV)^{SL(n+1)}$.
By swapping $m$ with $d$, we get exactly the construction performed in 
\S \ref{symbolicrep}.\qed

Note that by 1FT, $S^{mh_1}V\otimes\ldots \otimes S^{mh_d}V\neq 0$
if and only if $m(h_1+\ldots +h_d)$ is a multiple of $(n+1)$.
The invariants of minimal degree are those with
$$m=\frac{lcm\left(h_1+\ldots +h_d,n+1\right)}{h_1+\ldots+h_d}.$$

If $g(n+1)=md$ we have the weight $(\underbrace{1,\ldots, 1}_d)$ and
the graded ring $\oplus_m S^{m}V\otimes\ldots \otimes S^{m}V$.

\begin{theorem}[2FT for points]
In the invariant ring of $d$ (ordered or unordered) points, all the relations between
the tableau functions are generated by the Pl\"ucker relations $\sum_{s=0}^{k+1}(-1)^sG_{T_s}=0,$ exactly like in Theorem \ref{2ftforms},
with the tableau $T_s$ having the correct weight.
\end{theorem}

A strong improvement of 2FT for ordered points on $\P^1$ is relatively recent, and it will be treated in Theorem \ref{2ftpointsp1}.

The theory is better explained by examples.

\begin{proposition}\label{d26}
\

(i) The conic through $P_0,\ldots, P_4$ has equation
$$[014][234][02x][13x]-[024][134][01x][23x]=0.$$

(ii) $6$ points in the plane $\P^2$ lie on a conic if and only if
\begin{equation}\label{d2conic}d_2:=[014][234][025][135]-[024][134][015][235]=0.
\end{equation}
(it is an irrational invariant,
indeed it is $Alt(6)$-invariant but not $\Sigma_6$-invariant).
\end{proposition}

\begin{proof} The singular conics between $P_0,\ldots, P_3$ are
$[01x][23x]$, $[02x][13x]$, $[03x][12x]$
we ask that there exist $A$ and $B$ such that
$A[01x][23x]+B[02x][13x]$ vanish for $x=P_4$.
Hence we have
$A[014][234]+B[024][134]$
which is satisfied for $A=-[024][134]$ and $B=[014][234]$.
\end{proof}

\begin{remark} The expression in (ii) of Proposition \ref{d26} is the symbolic expression for the $6\times 6$ determinant
having in the $i$-th row the coefficients
$x_0^2,\ldots x_2^2$ computed at $P_i$.

Note again that the skew-symmetry is not at all evident from the symbolic expression.
It can be showed by using the Pl\"ucker relations,
in this case there are $35$ quadratic relations.
\end{remark}

\subsection{The graphical algebra for the invariants of $d$ points on the line. Kempe's Lemma.}

In 1894 A. Kempe\cite{Ke} introduced a graphical representation of invariants.

Let fix a weight $(h_1,\ldots , h_d)$ and consider the algebra
$$\CC[\CC^2\otimes\CC^d]_{(h)}=\oplus_m S^{mh_1}\CC^2\otimes\ldots \otimes S^{mh_d}\CC^2.$$
Any $d$ points on a line $\P^1$ can be represented as $d$ vertex of a regular polygon,
numbered clockwise from $1$ to $d$.

The bracket function $(ij)$ between $i$ and $j$ is represented as an arrow from $i$ to $j$.
A tableau function of weight $mh$ is represented as a graph with valence $mh_i$ at the vertex $i$.

For example
$(12)(34)(56)$ represents an invariant with respect to the weight $1^6$ and it corresponds to the graph

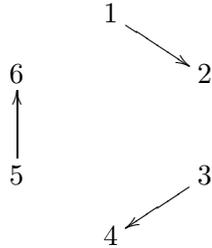
\begin{equation*}\begin{matrix}{\xymatrix@R-1.3em{&1\ar[dr]\\
6&&2\\
\\
5\ar[uu]&&3\ar[dl]\\
&4
} }\end{matrix}
\end{equation*}

Inverting one arrow corresponds to a sign change.
The relation
$(ij)(kl)-(ik)(jl)+(il)(jk)=0$ is represented graphically by

\begin{equation}\label{gr13}\begin{matrix}\xymatrix{i\ar[d]&k\ar[d]\\
j&l\\}\end{matrix}\quad - \quad  \begin{matrix}\xymatrix{i\ar[r]&k\\
j\ar[r]&l\\}\end{matrix}\quad + \quad\begin{matrix} \xymatrix{i\ar[dr]&k\\
j\ar[ur]&l\\}\end{matrix}=0\end{equation}

The graphical algebra is the same algebra generated by tableau functions. Every linear combination of tableau
functions transfers to a linear combination of corresponding graphs and conversely.
The product of two graphs corresponds to the union of the corresponding arrows, like in   
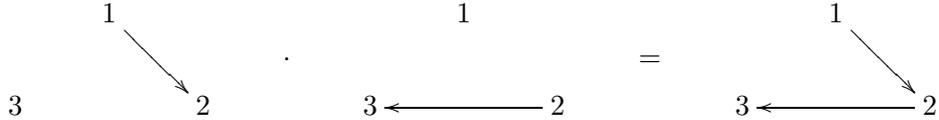
\begin{equation*}\begin{matrix}{\xymatrix@R2em{&1\ar[dr]\\
3&&2\\
} }\end{matrix}\qquad\cdot\qquad
\begin{matrix}{\xymatrix@R2em{&1\\
3&&2\ar[ll]\\
} }\end{matrix}\qquad=\qquad
\begin{matrix}{\xymatrix@R2em{&1\ar[dr]\\
3&&2\ar[ll]\\
} }\end{matrix}
\end{equation*}

This graphical algebra has been carefully studied in a series of recent papers by 
Howard, Millson, Snowden and Vakil. In order to review their approach and their main results,
we recall the following basic result from graph theory.
\begin{theorem}[Hall Marriage Theorem]
Consider a graph $G$  with $2m$ vertices, $m$ being positive and $m$ being negative.
A perfect matching is a collection of $m$ edges, each one joining one positive vertex with one negative vertex,
in such a way that every vertex belongs to one edge.
The necessary and sufficient condition that $G$ contains a perfect matching, is that for every subset $Y$
of positive vertices, the cardinality of the set of negative vertices which are connected to at least one member of $Y$
is bigger or equal than the cardinality of $Y$.
\end{theorem}

\begin{theorem}[Kempe's Lemma, 1FT for ordered points on $\P^1$ ]\label{kempe0}
All invariants of ordered points on $\P^1$  are generated by tableau functions
of minimal degree $$\frac{lcm(h_1+\ldots + h_d,(n+1))}{h_1+\ldots+h_d}.$$
With respect to the weight $(1,1,\ldots, 1)$, the generating invariants have degree $1$ when $d$ is even
and  degree $2$ when $d$ is odd.
\end{theorem}

\begin{proof} We follow \cite{HMSV1}, who gave a shorter proof than Kempe original one.
We assume for simplicity that $d$ is even and $h=(1,\ldots, 1)$, and refer to  \cite{HMSV1} for the general case.
Consider the graphical description of an invariant of weight $(m,\ldots, m)$ .
Divide the vertices into two subsets of equal cardinality, called positive and negative.

So the edges have three possible types: positive (both vertices positive),
 negative (both vertices negative) and  neutral (two opposite vertices).

Since every monomial is homogeneous, every vertex has valence $m$. It follows that the number
of positive edges is equal to the number of negative edges.
Applying the relation (\ref{gr13}) to a pair given by a positive and a negative edge, we get all neutral edges.
Continuing in this way, we get a combination of graphs, each one with all neutral edges.
Then the assumption of Hall Marriage Theorem is satisfied, because from every subset $Y$
 of $p$ positive vertices start $pm$ edges. Since the valence of each vertex is $m$, $Y$ must connect to at least $p$ negative vertices. 
So a perfect matching exists. Note that a perfect matching corresponds to a generator of minimal degree $1$.
By factoring this generator we can conclude by induction on $m$.
\end{proof}

From Theorem \ref{fillglv} and from the proof of Theorem \ref{1ftpoints}, a basis of tableau functions $G_T$ of minimal degree is given by semistandard tableau.
In the case of points of $\P^1$, an alternative description is possible.
A graph is said to be noncrossing if no two edges cross in an interior point.

\begin{theorem}[Kempe]\label{kempe}
A basis of the tableau functions of minimal degree is given by noncrossing complete matching
of minimal degree.
\end{theorem}
\begin{proof}
We apply the relation (\ref{gr13}) to a pair of noncrossing edges.
We get a combination of graphs, in each of them the total euclidean length of the edges is strictly smaller.
This process must terminate, because there is a finite number of complete matching,
hence a finite number of total euclidean length. When the process terminates, we have a combination of noncrossing
complete matching, otherwise the process could be repeated.
This shows that the noncrossing complete matching span. For the proof of independence of noncrossing graphs with the same weight $h$, assume we have a nonzero relation
involving a minimal number $n$ of vertices. Not all the graphs appearing contain the edge $(n-1)n$, otherwise we could remove
it, obtaining a smaller relation. Now identify the vertices $(n-1)$ and $n$, so that the graphs containing the edge $(n-1)n$ go to zero.
This gives a bijection between the graphs on $n$ vertices with weight $h$, not containing the edge $(n-1)n$, and the graphs
on $(n-1)$ vertices with weight $(h_1,\ldots, h_{n-2}, h_{n-1}+h_n)$. We get a nonzero relation on $(n-1)$ vertices, contradicting
the minimality.
\end{proof} 

\begin{remark} The above proof gives a graphical version of the straightening algorithm,
as pointed out in \cite{HMSV1} Prop. 2.5.
\end{remark}

For example, the space of invariants of $6$ points on $\P^1$ with weight $1^6$ is generated by  the noncrossing complete matching as follows.
We draw all the arrows from even to odd.

\begin{equation}\label{t05}t_1=\begin{matrix}{\xymatrix@R-1.3em{&1\\
6\ar[dd]&&2\ar[ul]\\
\\
5&&3\\
&4\ar[ur]
} }\end{matrix}
\qquad\qquad t_2=\begin{matrix}\xymatrix@R-1.3em{&1\\
6\ar[ur]&&2\ar[dd]\\
\\
5&&3\\
&4\ar[ul]
} \end{matrix}
\end{equation}

\vskip 1cm
\begin{equation*}t_3=\begin{matrix}\xymatrix@R-1.3em{&1\\
6\ar[dd]&&2\ar[dd]\\
\\
5&&3\\
&4\ar[uuuu]
} \end{matrix}
\qquad\qquad t_4=\begin{matrix}\xymatrix@R-1.3em{&1\\
6\ar[ru]&&2\ar[ddll]\\
\\
5&&3\\
&4\ar[ur]
} \end{matrix}
\qquad\qquad t_5=\begin{matrix}\xymatrix@R-1.3em{&1\\
6\ar[ddrr]&&2\ar[ul]\\
\\
5&&3\\
&4\ar[ul]
} \end{matrix}
\end{equation*}

\vskip 1cm

\begin{theorem}[2FT for ordered points on $\P^1$, Howard, Millson, Snowden and Vakil
\cite{HMSV2}]\label{2ftpointsp1}
In the invariant ring for $d$ ordered points on $\P^1$, with any weight $w\neq 1^6$,
the relation among the generators of minimal degree
are generated by quadric relations.

In the case $w=1^6$, we will see  in Theorem \ref{sixordered} that there is a unique cubic relation
among the generators of minimal degree of the invariant ring of six ordered points on $\P^1$ (this relation gives the Segre cubic primal, see Remark \ref{segreremark}).
\end{theorem}

\begin{remark} Theorem \ref{coblesix} will show that, in the case of $d$ unordered points, the relations are more complicated, certainly
not generated by quadric relations.
\end{remark}

\subsection{Molien formula and elementary examples}\label{molienelementary}
The following Theorem by Molien shows that the Hilbert series of the invariant subring is the average of the inverse of the characteristic polynomial.
\begin{proposition}
Let $G$ be a finite group acting on $U$.
The induced action on the symmetric algebra $S^*U$ has Hilbert series
\begin{equation}\label{molienformula}\sum_{i=0}^{+\infty}\dim \left(S^iU\right)^Gq^i=\frac{1}{|G|}\sum_{g\in G}\frac{1}{\det(1-qg)}\end{equation}
where $g$ acts on $U$.
\end{proposition}
\begin{proof} \cite{Stu} Theor. 2.2.1 \end{proof}

We analyze now some invariant rings for $d$ points on $\P^1$, for small $d$.
The simplest cases is the following.

\begin{theorem}The Hilbert series of the invariant ring of three ordered points
on $\P^1$
is $$\frac{1}{1-t^2}.$$
The ring is generated by the only noncrossing matching of valence $2$ which is 
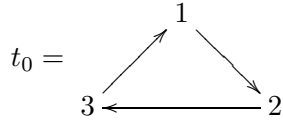
\begin{equation*}t_0=\begin{matrix}{\xymatrix@R2em{&1\ar[dr]\\
3\ar[ur]&&2\ar[ll]\\
} }\end{matrix}
\end{equation*}
\end{theorem}

\begin{proof} Immediate from Theorem \ref{kempe}.
\end{proof}

We get the second promised proof of the Corollary \ref{corobinarycubic}.
\begin{corollary}\label{corobinarycubic2}The Hilbert series of the invariant ring of three unordered points
on $\P^1$
is $$\frac{1}{1-t^4}.$$
The ring is generated by the discriminant $\Delta=t_0^2$ .
\end{corollary}

\begin{proof}  The function $t_0$ is $Alt(3)$-invariant (for even permutations),
while  $t_0^2$ is $\Sigma_3$-invariant.
\end{proof}

\begin{theorem}\label{orderedfour}The Hilbert series of the invariant ring of four ordered points
on $\P^1$
is $$\frac{1}{(1-t)^2}.$$
The ring is generated by the two noncrossing matchings which are
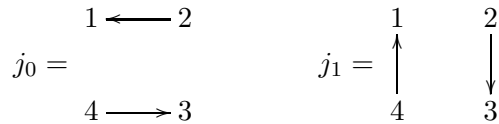
\begin{equation*}\label{t05}j_0=\begin{matrix}{\xymatrix@R2em{1&2\ar[l]\\
4\ar[r]&3\\
} }\end{matrix}\qquad\qquad
j_1=\begin{matrix}{\xymatrix@R2em{1&2\ar[d]\\
4\ar[u]&3\\
} }\end{matrix}
\end{equation*}
\end{theorem}

\begin{proof}  $j_0$ and $j_1$ generate the invariant ring by Theorem \ref{kempe}. Moreover $j_0$ and $j_1$ are algebraically independent. This can be shown directly
or by observing that the geometric quotient is one dimensional.
\end{proof}

In the point (i) of the following Theorem we get a second proof of Theorem \ref{IJgenerate}.
\begin{theorem}\label{IJgenerate2}
\

(i) The Hilbert series of the invariant ring of four unordered points
on $\P^1$ (binary quartics)
is $$\frac{1}{(1-t^2)(1-t^3)}.$$
 The ring is generated by $I$, $J$ defined in (\ref{firstI}), (\ref{firstJ}).

(ii) The Hilbert series of the $Alt(4)$-invariant ring of four unordered points
on $\P^1$ (binary quartics)
is $$\frac{1+t^3}{(1-t^2)(1-t^3)}=\frac{1-t^6}{(1-t^2)(1-t^3)^2}.$$
 The ring is generated by $I$, $J$, $\sqrt{D}$ where $\sqrt{D}$ has degree $3$ 
and it is the product of differences of the roots of the quartic.
The only relation
is $(\sqrt{D})^2=I^3-27J^2$.
\end{theorem}

\begin{proof} The result is elementary, but we give the details of the representation theoretic approach as warming up
for the more interesting case of six points (Theorem \ref{coblesix}). The two generators $j_0$, $j_1$ of Theorem \ref{orderedfour} span the unique irreducible representation $W$
of dimension $2$ of $\Sigma_4$. So we have to compute the Hilbert series of $\oplus_p S^p(W)^{\Sigma_4}$.
The proof is a straightforward computation by using Molien's formula (\ref{molienformula}), by summing
over the five conjugacy classes in $\Sigma_4$.

The result 
is

\begin{align*}
\frac{1}{24}\left[\frac{1}{(1-t)^2}+\frac{6}{1-t^2}+\frac{8}{1+t+t^2}+\frac{6}{1-t^2}+\frac{3}{(1-t)^2}\right]=\nonumber\\
=\frac{1}{(1-t^2)(1-t^3)}.\end{align*}
Let $I, J$ be the second and third elementary symmetric function of $j_0,-j_1, j_1-j_0$.
$I$ and $J$ correspond, up to scalar factor, to (\ref{firstI}), (\ref{firstJ}) and they are the generators of the invariant ring corresponding
to the two factors in the denominator of the Hilbert series.
This proves (i). (ii) can be proved by restricting the sum to the even conjugacy classes,
which are the ones with numerator $1$, $8$, and $3$.
\end{proof}

\begin{remark} The cross ratio $\frac{(x_3-x_1)(x_4-x_2)}{(x_3-x_2)(x_4-x_1)}$ of four points with affine coordinates $x_i$ for $i=1,\ldots , 4$
has the expression $\frac{j_1-j_0}{j_1}$. It parametrizes the moduli space of $4$ ordered points on $\P^1$, which is isomorphic to $\P^1$ itself.
\end{remark}

\subsection{Digression about the symmetric group $\Sigma_6$ and its representations}\label{reps6}

We list a representative for each of the $11$ conjugacy classes of $\Sigma_6$.

$$\begin{array}{l|l|l|r}
&&\textrm{even}&\textrm{number of elements}\\
\hline
C1& (1)&*&1\\
\hline
C2& (12)&&15\\
\hline
C3& (12)(34)&*&45\\
\hline
C4& (12)(34)(56)&&15\\
\hline
C5& (123)&*&40\\
\hline
C6&(123)(45)&&120\\
\hline
C7&(123)(456)&*&40\\
\hline
C8&(1234)&&90\\
\hline
C9&(1234)(56)&*&90\\
\hline
C10&(12345)&*&144\\
\hline
C11&(123456)&&120\\
\end{array}
$$

The natural action of $\Sigma_d$ on ${\bf C}^d$
splits into the trivial representation $U$ (dimension one)
and the standard representation $V$ (dimension $d-1$).
The subspace $V\subset {\bf C}^d$ is given by $\sum e_i=0$.

The exterior representation $U'\colon \Sigma_d\to {\bf C}^*$ which sends $p\in \Sigma_d$ to its sign $\epsilon(p)$
has again dimension one.
For any representation $W$, it is customary to denote
$W'=W\otimes U'$. $W'$ corresponds to the transpose
Young diagram of $W$.
Each representation of $\Sigma_n$ is self-dual.

We list the irreducible representations in the case $d=6$ and we put in evidence the transpose diagrams.
$$\begin{array}{llc|llc}
\textrm{name}&\textrm{shape}&\textrm{dimension}&
\textrm{name}&\textrm{shape}&\textrm{dimension}\\
\hline\\
X1=U&\yng(6)&1&X7=\wedge^3V&\yng(3,1,1,1)&10\\
X2=V&\yng(5,1)&5&X8=X5'&\yng(2,2,2)&5\\
X3&\yng(4,2)&9&X9=X3'&\yng(2,2,1,1)&9\\
X4=\wedge^2V&\yng(4,1,1)&10&X10=\wedge^4V=V'&\yng(2,1,1,1,1)&5\\
X5&\yng(3,3)&5&X11=U'&\yng(1,1,1,1,1,1)&1\\
X6=X6'&\yng(3,2,1)&16\\
\end{array}$$

One peculiarity of the above list is that there are
{\it four} irreducible representations of dimension $5=d-1$.
This happens only in the case $d=6$. $S_6$ is
the only symmetric group which admits an automorphism
which is not inner, which can be defined indeed
by means of $X5$ (or dually by $X8$).

We list the character table of $\Sigma_6$, from the appendix of \cite{JK}.
Note that the first column gives the dimension.

$$\begin{array}{r|rrrrrrrrrrr}
&C1&C2&C3&C4&C5&C6&C7&C8&C9&C10&C11\\
\hline\\
X1&1&1&1&1&1&1&1&1&1&1&1\\
X2&5&3&1&-1&2&0&-1&1&-1&0&-1\\
X3&9&3&1&3&0&0&0&-1&1&-1&0\\
X4&10&2&-2&-2&1&-1&1&0&0&0&1\\
X5&5&1&1&-3&-1&1&2&-1&-1&0&0\\
X6&16&0&0&0&-2&0&-2&0&0&1&0\\
X7&10&-2&-2&2&1&1&1&0&0&0&-1\\
X8&5&-1&1&3&-1&-1&2&1&-1&0&0\\
X9&9&-3&1&-3&0&0&0&1&1&-1&0\\
X10&5&-3&1&1&2&0&-1&-1&-1&0&1\\
X11&1&-1&1&-1&1&-1&1&-1&1&1&-1\\
\end{array}
$$

Other useful formulas are (\cite{FH} ex. 2.2)
$$\chi_{S^2W}(g)=\frac{1}{2}\left[\chi_W(g)^2+\chi_W(g^2)\right],$$

\begin{equation}\label{characterows}\chi_{S^3W}(g)=\frac{1}{6}\chi_W(g)^3+\frac{1}{2}\chi_W(g)\chi_W(g^2)+\frac{1}{3}\chi_W(g^3).
\end{equation}

We have the following table explaining how the powers of elements divide among the conjugacy classes.
\begin{equation}\label{conjugacysix}\begin{array}{r|rrrrrrrrrrr}
g&C1&C2&C3&C4&C5&C6&C7&C8&C9&C10&C11\\
g^2&C1&C1&C1&C1&C5&C5&C7&C3&C3&C10&C7\\
g^3&C1&C2&C3&C4&C1&C2&C1&C8&C9&C10&C4\\
g^4&C1&C1&C1&C1&C5&C5&C7&C1&C1&C10&C7\\
g^5&C1&C2&C3&C4&C5&C6&C7&C8&C9&C1&C11\\
\end{array}\end{equation}

\subsection{The invariant ring of six points on the line}
\begin{theorem}\label{sixordered}The Hilbert series of the invariant ring of six ordered points
on $\P^1$
is $$\frac{1-t^3}{(1-t)^5}.$$
The ring is generated by the five noncrossing matching $t_1,\ldots, t_5$ as in (\ref{t05}).
We have the unique relation in degree $3$ 

\begin{equation}\label{segrecubic}t_1t_2(-t_1-t_2+t_3+t_4+t_5)-t_3t_4t_5=0.\end{equation}
\end{theorem}

\begin{proof} The five noncrossing matching $t_1,\ldots, t_5$  listed in (\ref{t05}) generate by Theorems
\ref{kempe0} and \ref{kempe}. 
They are easily identified as a basis
for the representation $X5$ of $\Sigma_6$.
By dimensional reasons (the quotient has dimension three), only one relation is expected.

The fact that there are no relations in degree two can be proved by counting
the number of semistandard tableau $2\times 6$ of weight $2^6$, which are indeed $15$.
In order to understand the relation, we add a sixth tableau which is a linear combination of the first noncrossing $t_1,\ldots, t_5$ as

\begin{equation}\label{t00}t_0=\begin{matrix}\xymatrix@R-1.3em{&1\\
6\ar[ddrr]&&2\ar[ddll]\\
\\
5&&3\\
&4\ar[uuuu]
}\end{matrix} \end{equation}

After some computation with the straightening algorithm, as in the proof of Theorem \ref{kempe}, we get $t_0=-t_1-t_2+t_3+t_4+t_5$.

Looking at the graphs there is the obvious relation $t_0t_1t_2-t_3t_4t_5=0$ that we show in the following picture.

{\tiny
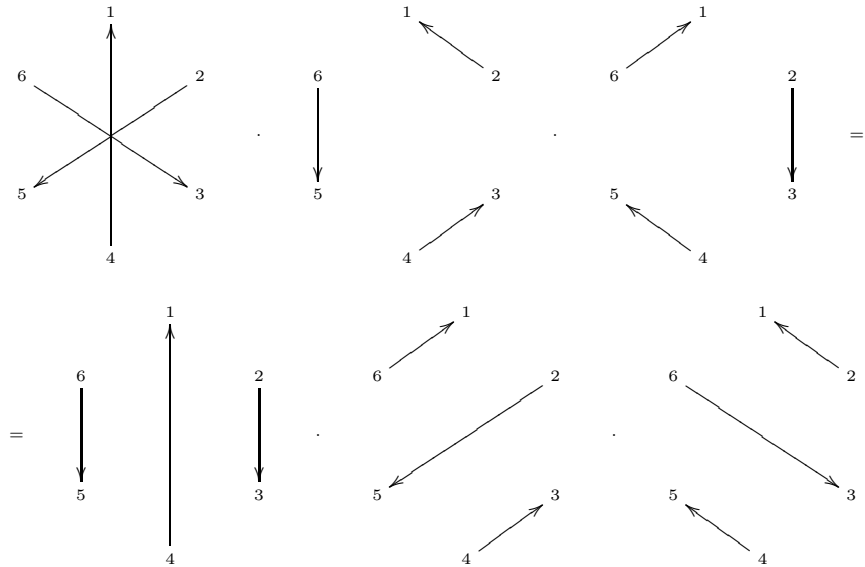
$$\begin{matrix}\xymatrix@R-1.3em{&1\\
6\ar[ddrr]&&2\ar[ddll]\\
\\
5&&3\\
&4\ar[uuuu]
}\end{matrix}\qquad\cdot\qquad\begin{matrix}{\xymatrix@R-1.3em{&1\\
6\ar[dd]&&2\ar[ul]\\
\\
5&&3\\
&4\ar[ur]
} }\end{matrix}
\qquad\cdot\qquad \begin{matrix}\xymatrix@R-1.3em{&1\\
6\ar[ur]&&2\ar[dd]\\
\\
5&&3\\
&4\ar[ul]
} \end{matrix}
\qquad=$$
$$=\qquad\begin{matrix}\xymatrix@R-1.3em{&1\\
6\ar[dd]&&2\ar[dd]\\
\\
5&&3\\
&4\ar[uuuu]
} \end{matrix}
\qquad\cdot\qquad \begin{matrix}\xymatrix@R-1.3em{&1\\
6\ar[ru]&&2\ar[ddll]\\
\\
5&&3\\
&4\ar[ur]
} \end{matrix}
\qquad\cdot\qquad\begin{matrix}\xymatrix@R-1.3em{&1\\
6\ar[ddrr]&&2\ar[ul]\\
\\
5&&3\\
&4\ar[ul]
} \end{matrix}$$
}

This concludes the proof.
 
\end{proof}

\begin{remark}\label{segreremark} The cubic $3$-fold (\ref{segrecubic}) is called the Segre cubic primal.
It has $10$ singular points and it contains $15$ planes (see \cite{Do2} example 11.6).
It can be expressed as the sum of six cubes, and not eight like the general cubic $3$-fold (see \cite{Ott, RS}).
An explicit expression as a sum of cubes will be obtained in (\ref{sixcubes}).
\end{remark}

\begin{remark}
In \cite{Do2} example 11.6 it is reported an alternative combinatorial proof of Theorem \ref{sixordered},
by counting the number of semistandard tableau $2\times 3m$ of weight $m^6$. 
\end{remark}

Since the relation (\ref{segrecubic}) is $Alt(6)$-invariant but not $\Sigma_6$-invariant,
we need a preliminary step to get the $\Sigma_6$-invariants.

By following Coble \cite{Cob} \S 3, we define the Joubert invariants
(see \cite{D} for an intrinsic way to compute these invariants).

{\scriptsize
\begin{equation}\label{AF}\begin{array}{rl}
A=(25)(13)(46)+(51)(42)(36)+(14)(35)(26)+(43)(21)(56)+(32)(54)(16)=&4t_1 + 4t_2 - 2t_3 - 2t_4 - 2t_5\\
B=(53)(12)(46)+(14)(23)(56)+(25)(34)(16)+(31)(45)(26)+(42)(51)(36)=&- 2t_3 - 2t_4 + 2t_5\\
C=(53)(41)(26)+(34)(25)(16)+(42)(13)(56)+(21)(54)(36)+(15)(32)(46)=&2t_3 - 2t_4 - 2t_4\\
D=(45)(31)(26)+(53)(24)(16)+(41)(25)(36)+(32)(15)(46)+(21)(43)(56)=&- 4t_1 + 2t_3 + 2t_4 + 2t_5\\
E=(31)(24)(56)+(12)(53)(46)+(25)(41)(36)+(54)(32)(16)+(43)(15)(26)=& - 4t_2 + 2t_3 + 2t_4 + 2t_5\\
F=(42)(35)(16)+(23)(14)(56)+(31)(52)(46)+(15)(43)(26)+(54)(21)(36)=&- 2t_3 + 2t_4 - 2t_5\\
\end{array}\end{equation}}

In terms of redundant $t_0,\ldots, t_5$ we have nicer expressions

$$\left\{\begin{array}{cccccccc}A/2&=&-t_0&+t_1&+t_2\\
B/2&=&&&&-t_3&-t_4&+t_5\\
C/2&=&&&&+t_3&-t_4&-t_5\\
D/2&=&+t_0&-t_1&+t_2\\
E/2&=&+t_0&+t_1&-t_2\\
F/2&=&&&&-t_3&+t_4&-t_5\\
\end{array}\right.$$

Note that $$A+B+C+D+E+F=0.$$ so that any five among them give a basis of the space of invariant tableau functions
for six ordered points.

A direct inspection shows that an even permutation of the points effects an even permutation of the functions
and that an odd  permutation of the points effects an odd permutation of the functions accompanied by a change of sign. For example after the permutation $(12)$
we get
$$A\to -D\qquad B\to -E\qquad C\to -F,$$
after the permutation $(13)$
we get
$$A\to -F\qquad B\to -D\qquad C\to -E.$$

 Let $b_{15}=(A-B)(A-C)\ldots(E-F)$ be the product of differences.
The previous permutation rules about $A,\ldots, F$  show that $b_{15}$ is $\Sigma_6$-invariant.

We denote by $a_i$ the $i$-th elementary symmetric function of $A,\ldots , F$.
It follows that $a_i$ are $Alt(6)$-invariant and $\Sigma_6$-invariant for even $i$.

Consider $\prod_{i<j}(x_j-x_i)=\Delta$ 
which is a $Alt(6)$-invariant, and it is the square root of the discriminant of the 
polynomial having the points as roots.

Every $Alt(6)$-invariant can be written as $a+b\Delta$ where $a, b$ are $\Sigma_6$-invariants,
see \cite{Stu} Prop. 1.1.3. It follows that $\Delta$ is equal to $a_5$ (up to scalar multiples)
and by degree reasons we get that $a_1=a_3=0$.
The last equation can be written in equivalent way as

\begin{equation}\label{sixcubes}A^3+B^3+C^3+D^3+E^3+F^3=0,\end{equation}

which is the promised expression of the Segre cubic primal as a sum of six cubes.

\begin{theorem}[Coble]\label{sixunorderedp1} The Hilbert series of the $Alt(6)$-invariant ring of six (unordered) points on $\P^1$
is $$\frac{1+t^{15}}{(1-t^2)(1-t^4)(1-t^5)(1-t^6)}.$$
The ring is generated by $a_2, a_4, a_6,\Delta, b_{15}$
with the relation in degree $30$ expressing $b_{15}^2$ as a polynomial in the other generators.
\end{theorem}
\begin{proof} 
$X5$ and $X8$ restrict to the same  representation of $Alt(6)$, that we call again $X5$.
We compute first the Hilbert series of $\oplus_p S^p(X5)^{Alt(6)}$.
The proof is a straightforward computation by using Molien formula (\ref{molienformula}).
The characteristic polynomial of any $g\in Alt(6)$ can be computed with the following trick.
The trace of the action of $g^i$ is computed by the character table and by the table (\ref{conjugacysix}).
It gives the Newton sums $\sum_{j=1}^5\lambda_j^i$, where $\lambda_j$ are the eigenvalues
of the $5\times 5$ matrix representing the action of $g$ on $X5$.
Then, by the Newton identities, we can compute the $i$-th elementary symmetric functions.

We sum in Molien formula over the six even conjugacy classes in $\Sigma_6$, as in the proof of Theorem \ref{IJgenerate2}.

The result 
is

\begin{align*}
\frac{1}{360}\left[\frac{1}{(t-1)^5}+\frac{45}{(t-1)^3(t+1)^2}+\frac{40}{(t-1)(t^2+t+1)^2}+\frac{40}{(t-1)^3(t^2+t+1)}\right.+\nonumber\\
+\left.\frac{90}{(t-1)(t+1)^2(t^2+1)}+\frac{144}{(t-1)(t^4+t^3+t^2+t+1)}\right]=\end{align*}
\begin{equation}\label{6classes}
=\frac{1+t^{15}}{(1-t^2)(1-t^3)(1-t^4)(1-t^5)(1-t^6)}.\end{equation}

By considering the relation (\ref{segrecubic}), the factor $(1-t^3)$ cancels from
the denominator in the Hilbert series.
The remaining factors in the denominator correspond to $a_2, a_4,  \Delta, a_6$.
This suggests that there is another generator of degree $15$,
which is identified with $b_{15}$. By dimensional reasons, we expect a single relation.
The square $b_{15}^2$ is $\Sigma_6$-invariant
and it corresponds to the discriminant of the polynomial having $A,\ldots, F$ as roots,
so it can be expressed as a polynomial in $a_i$.
\end{proof}

\begin{theorem}[Coble]\label{coblesix} The Hilbert series of the $\Sigma_6$-invariant ring of six (unordered) points on $\P^1$ (binary sextic)
is $$\frac{1+t^{15}}{(1-t^2)(1-t^4)(1-t^6)(1-t^{10})}.$$
The ring is generated by $a_2, a_4, a_6, \Delta^2, b_{15}$
with the relation in degree $30$ expressing $b_{15}^2$ as a polynomial in $a_2,\ldots, a_6$.
\end{theorem}
\begin{proof} The Hilbert series is obtained as in the proof of Theorem \ref{sixunorderedp1}, by adding the contribution of the remaining five conjugacy classes.
Among the generators of the $Alt(6)$-invariant ring,
$a_2, a_4, a_6, b_{15}$ are already $\Sigma_6$-invariant. 
$\Delta^2$ is another independent invariant. The ring generated by these invariant
has the claimed Hilbert series.
\end{proof}

The Hilbert function of $d$ ordered points on $\P^1$ has been found by Howe. 
\begin{theorem}
The dimension of the space of invariants for $d$ ordered points on $\P^1$
and multidegree $k^d$
with the weight $\left\{\begin{array}{l}1^d\textrm{\ if\ }d\textrm{\ is even}\\
2^d\textrm{\ if\ }d\textrm{\ is odd}\end{array}\right.$
is

$$\displaystyle\sum_{j=0}^{(d-1)/2}(-1)^j{d\choose j}{{k(d/2-j)+d-2-j}\choose d-2}.$$

In this formula it is understood that a binomial coefficient ${a\choose b}$ is zero if $a<b$.
In the case $d$ even, this formula gives the degree $k$ part of the invariant ring with respect to $1^d$.

In the case $d$ odd, the formula is meaningful for $k$ even, 
and gives the degree $k/2$ part of the invariant ring with respect to $2^d$.
\end{theorem}
\begin{proof}(\cite{Ho} 5.4.2.3)).
\end{proof}

As an application, from this formula, in \cite{FS} Theorem 1.6, it has been computed the Hilbert series of the invariant ring of $8$ ordered points on $\P^1$,
which is
{\footnotesize
$$\frac{1+8t+22t^2+8t^3+t^4}{(1-t)^6}=\frac{1-14t^2+175t^4-512t^5+700t^6-512t^7+175t^8-14t^{10}+t^{12}}{(1-t)^{14}}.$$}
The second formulation is reported because the coefficients are the Betti number of the resolution
of the GIT quotient $\P(\CC^2\otimes\CC^8)_{1^8}/SL(2)$, which is the moduli space of $8$ ordered points on a line
(see 
\cite{FS} Lemma 1.4, \cite{HMSV8} Prop. 7.2).

\subsection{The invariant ring of six points on the plane. Cremona hexahedral equation for the cubic surface.}

Let $a_1,\ldots a_6$ be six points on the projective plane.
The aim of this section is to report in representation theoretic language the results by Coble \cite{Cob}
about six points on the plane, and to apply them to the construction of Cremona hexahedral equations.

We denote by $(ijk)$ the $3\times 3$ determinant of three points labeled with $i$, $j$, $k$.
In particular $(ijx)$ is the equation in the coordinate $x$ of the line through $a_i$ and $a_j$. 
This fits particularly well with the graphical description of the previous section, indeed
the line $(ijx)$ can be seen just by prolongation of the arrow between $i$ and $j$.

Any relation between invariants $(ij)$ on $\P^1$ transfers to an analogous relation
among covariants $(ijx)$ on $\P^2$. This is ``Clebsch transfer principle''.

In general, an invariant is a function $F(a_1,\ldots , a_6)$ 
which is symmetric (classically called rational)
or skew-symmetric (classically called irrational) in the points, of degree $q$ in the coordinates of each point,
 such that
for every $g\in SL(3)$
$F(ga_1,\ldots , ga_6)=(det g)^{(q/3)}F(a_1,\ldots , a_6)$.

The covariants are polynomials in the invariants $(ijk)$ and in the $(ijx)$,
which again are  symmetric or skew-symmetric in the six points.

For example
$(12x)(34x)(56x)$ represents a cubic splitting in three lines.
There are $15$ such cubics (as symbols), which span the $4$-dimensional space of all cubics through the six points.

The six generators $t_0,\ldots, t_5$ of (\ref{t05}) and (\ref{t00}) give six cubics through the six lines.
The relation $t_0+t_1+t_2=t_3+t_4+t_5$ transfers to an analogous relation between cubics.

It is more convenient to consider the six invariants $A\ldots F$ of (\ref{AF}) which induce the following list of cubics

$${\small \begin{array}{l}
a=(25x)(13x)(46x)+(51x)(42x)(36x)+(14x)(35x)(26x)+(43x)(21x)(56x)+(32x)(54x)(16x)\\
b=(53x)(12x)(46x)+(14x)(23x)(56x)+(25x)(34x)(16x)+(31x)(45x)(26x)+(42x)(51x)(36x)\\
c=(53x)(41x)(26x)+(34x)(25x)(16x)+(42x)(13x)(56x)+(21x)(54x)(36x)+(15x)(32x)(46x)\\
d=(45x)(31x)(26x)+(53x)(24x)(16x)+(41x)(25x)(36x)+(32x)(15x)(46x)+(21x)(43x)(56x)\\
e=(31x)(24x)(56x)+(12x)(53x)(46x)+(25x)(41x)(36x)+(54x)(32x)(16x)+(43x)(15x)(26x)\\
f=(42x)(35x)(16x)+(23x)(14x)(56x)+(31x)(52x)(46x)+(15x)(43x)(26x)+(54x)(21x)(36x)
\end{array}}$$

For a representation-theoretic way to get these expressions, see \cite{D} 9.4.4.

They satisfy the  relation 

$$a+b+c+d+e+f=0$$

and indeed the six cubics span $X5$. 

Again an even permutation of the points effects an even permutation of the cubics
and an odd  permutation of the points effects an odd permutation of the cubics accompanied by a change of sign. For example after the permutation $(12)$
we get
$$a\to -d\qquad b\to -e\qquad c\to -f.$$

By the same argument given for invariants of points,
the relation (\ref{sixcubes})
transfers to the relation
$a^3+b^3+c^3+d^3+e^3+f^3=0$ 
(in equivalent way the third elementary symmetric function of $a,\ldots, f$ vanishes).

Since the space of cubics through $6$ points is four dimensional, there is a further relation
$$\overline{a}a+\overline{b}b+\overline{c}c+\overline{d}d+\overline{e}e+\overline{f}f=0,$$

which is uniquely determined considering the additional condition
$\overline{a}+\overline{b}+\overline{c}+\overline{d}+\overline{e}+\overline{f}=0$.

In order to find the additional relation between $a,\ldots ,f$, Coble considers a second interesting formulation. Define
$(ij,kl,mn)$ to be the function which represents
that the three lines $cij$, $ckl$, $cmn$ are concurrent in a point.
It is a $3\times 3$ determinant whose rows are
$i\wedge j$, $k\wedge l$, $m\wedge n$. This function can be expressed in terms of $(ijk)$
by the Lagrange identity 

$$(v\wedge w)\cdot(m\wedge n)=(v\cdot m)(w\cdot n)-(v\cdot n)(w\cdot m).$$

With $v=i\wedge j$, $w=k\wedge l$, we get
$$(ij, kl, mn)=(ijm)(kln)-(ijn)(klm)=(ijl)(kmn)-(ijk)(lmn)=(ikl)(jmn)-(jkl)(imn),$$
where the last two identities are obtained by a permutation of the rows
of the mixed product. They can be seen also as a consequence
of the Pl\"ucker relations.

Coble considers (\cite{Cob}page 170) the following expressions, obtained by formally replacing
in (\ref{AF}) $(ij)(kl)(mn)$ with $(ij,kl,mn)$.

{\small
$$\begin{array}{rl}
\overline{a}=(25,13,46)+(51,42,36)+(14,35,26)+(43,21,56)+(32,54,16)\\
\overline{b}=(53,12,46)+(14,23,56)+(25,34,16)+(31,45,26)+(42,51,36)\\
\overline{c}=(53,41,26)+(34,25,16)+(42,13,56)+(21,54,36)+(15,32,46)\\
\overline{d}=(45,31,26)+(53,24,16)+(41,25,36)+(32,15,46)+(21,43,56)\\
\overline{e}=(31,24,56)+(12,53,46)+(25,41,36)+(54,32,16)+(43,15,26)\\
\overline{f}=(42,35,16)+(23,14,56)+(31,52,46)+(15,43,26)+(54,21,36)\\
\end{array}$$}

which satisfy
\begin{equation}\label{a1zero}\overline{a}+\overline{b}+\overline{c}+\overline{d}+\overline{e}+\overline{f}=0.\end{equation}

Now any permutation of the points effects a permutation of the cubics,
without any change of sign. For example after the permutation $(12)$
we get
$$\overline{a}\to \overline{d},\qquad \overline{b}\to \overline{e},\qquad 
\overline{c}\to \overline{f}.$$

Indeed $\overline{a},\ldots ,\overline{f}$ span the representation $X8$.

\begin{proposition}
$$\overline{a}a+\overline{b}b+\overline{c}c+\overline{d}d+\overline{e}e+\overline{f}f=0.$$
\end{proposition}
\begin{proof}Consider the point at the intersection of $(12x)$ and $(34x)$.

By a direct computation, see for example \cite{D} Theor. 9.4.13, the cubic at the left-hand side of
our statement vanishes.
By the invariance,  the cubic vanishes at all the $45$ intersection points of the lines $(ijx)$ and $(klx)$,
hence it has to contain the $15$ lines and must vanish.
\end{proof}

Let $a_i$ for $i=2,\ldots, 6$ be the $i$-th elementary symmetric function in $\overline{a},\ldots, \overline{f}$
(remind that by (\ref{a1zero}) we have $a_1=0$).
Remind the expression $d_2$ in (\ref{d2conic}), which is an irrational invariant expressing
that the six points lie on a conic.
Let $\Delta$ be the product of the differences of $\overline{a},\ldots ,\overline{f}$.
It is an irrational invariant of degree $15$. Note that its square is the discriminant of the polynomial with roots $\overline{a},\ldots ,\overline{f}$
and it can be expressed as a polynomial in $a_2,\ldots, a_6$.

\begin{theorem}\label{sixorderedp2}The Hilbert series of the invariant ring of six ordered points
on $\P^2$ is $$\frac{1+t^2}{(1-t)^5}=\frac{1-t^4}{(1-t)^5(1-t^2)}.$$
The ring is generated by $\overline{a},\ldots, \overline{e}$
and by $d_2$.
We have the relation in degree $4$ $d_2^2=a_2^2-4a_4$.
\end{theorem}

\begin{proof} The Hilbert series reduces to a computation of semistandard tableau,
that can be achieved by counting the integral points in a certain polytope, see \cite{Do2} \S 11.2.
The invariants of degree $1$ are generated by $\overline{a},\ldots, \overline{f}$ with the relation
(\ref{a1zero}). In alternative, also the tableau functions $(123)(456)$, $(124)(356)$, $(125)(346)$, $(134)(256)$, $(135)(246)$
can be taken as generators.
Note that $d_2$ cannot be obtained as a polynomial in $\overline{a},\ldots, \overline{f}$ because otherwise
it should be a multiple of $\Delta$. Indeed all $Alt(6)$-invariants can be written as $\Delta\cdot p$
where $p$ is a symmetric polynomial in $\overline{a},\ldots, \overline{f}$ (see also Theorem \ref{alt6p2}).
\end{proof}

\begin{remark} This description shows that Kempe's Lemma (Theorem \ref{kempe0}) fails for points on $\P^2$, indeed 
$d_2$ is a generator which has degree $2$, while the minimal degree is $1$.
\end{remark}

\begin{theorem}[Coble]\label{alt6p2} The Hilbert series of the $Alt(6)$-invariant ring of six (unordered) points on $\P^2$
is $$\frac{1+t^{15}}{(1-t^2)^2(1-t^3)(1-t^5)(1-t^6)}.$$
The ring is generated by $a_2,d_2, a_3, a_5, a_6,\Delta$
with the relation in degree $30$ expressing $\Delta^2$ as a polynomial in the other generators.
\end{theorem}
\begin{proof} 
$X8$ is an irreducible representation of $Alt(6)$.
We already computed in the proof of Theorem \ref{sixunorderedp1} the Hilbert series of $\oplus_p S^p(X8)^{Alt(6)}$
by using Molien formula.

The result 
is
\begin{equation}\label{6classes}
\frac{1+t^{15}}{(1-t^2)(1-t^3)(1-t^4)(1-t^5)(1-t^6)}.\end{equation}
The factors in the denominator correspond to $a_2,\ldots, a_6$.
This suggests that there is another generator of degree $15$,
which is identified with $\Delta$.
Now from Theorem \ref{sixorderedp2} we have to add $d_2$ to the generators, which is already $Alt(6)$-invariant.
We get that $a_4$ can be deleted by the generators.
\end{proof}

\begin{theorem}[Coble] The Hilbert series of the $\Sigma_6$-invariant ring of six (unordered) points on $\P^2$
is $$\frac{1+t^{17}}{(1-t^2)(1-t^3)(1-t^4)(1-t^5)(1-t^6)}.$$
The ring is generated by $a_2, a_3, a_4,a_5, a_6, d_2\Delta$
with the relation in degree $34$ expressing $(d_2\Delta)^2$ as a polynomial in $a_2,\ldots, a_6$.
\end{theorem}
\begin{proof} 
Among the generators of the $Alt(6)$-invariant ring,
$a_2, a_3, a_5, a_6$ are already $\Sigma_6$-invariant. The relation $d_2^2=a_2^2-4a_4$ allows to add $a_4$ at the generators
at the place of $d_2$. Also $d_2\Delta$ is a $\Sigma_6$-invariant of degree $17$.
\end{proof}

\begin{remark}
By adding the contribution of the remaining five odd conjugacy classes to (\ref{6classes})
we get $$\frac{1}{(1-t^2)(1-t^3)(1-t^4)(1-t^5)(1-t^6)}.$$
\end{remark}

\begin{remark}
The Morley
covariant 
$$\overline{a}^2a+\overline{b}^2b+\overline{c}^2c+\overline{d}^2d+\overline{e}^2e+\overline{f}^2f$$
is a cubic quite important in the study of L\"uroth quartics. Its property is that given 
six points on $\P^2$, then a seventh point belongs to the Morley cubic if and only if the ramification quartic
obtained by the $2\colon 1$ map given linear system of cubics trough the seven points is a L\"uroth quartic
(see \S \ref{lurothsection} ) and moreover the seven points give a Aronhold system
of bitangents on the quartic which is associated to the pentalateral theta \cite{Mor} or \cite{OS1} Coroll. 4.2., Theor. 6.1 and Remark 10.7.
\end{remark}

\begin{remark} The quadratic symmetric invariant $a_2$ of six points is the 2nd elementary symmetric
function in $\overline{a},\ldots,\overline{f}$. Note that given $5$ points on $\P^2$,
the sixth point defines a covariant conic curve, not passing through the points.
What is the geometric interpretation of this conic ? \end{remark}

The previous construction can be resumed in the following steps.
Start from $6$ points in the plane.

Then the linear system of cubics through the six points embed the plane blown up at these six points
in the cubic surface with equation

$$\left\{\begin{array}{lll}a+b+c+d+e+f=0\\
\overline{a}a+\overline{b}b+\overline{c}c+\overline{d}d+\overline{e}e+\overline{f}f=0\\
a^3+b^3+c^3+d^3+e^3+f^3=0\\
\end{array}\right.$$

These
are known as {\it Cremona hexahedral equation} \cite{Cre, D}. They show that the general cubic surface is a hyperplane section of the Segre cubic primal
defined in Remark \ref{segreremark}.

The following line

$$a+b=c+d=e+f=0$$ belongs to the cubic surface. By applying permutations, we get $15$ such lines.

The remaining $12$ lines correspond to the proper transform of the six lines and of the 
six conics through five among the six points. They form a so called {\it double-six} on the cubic surface.

\bigskip

\begin{flushleft}

{\bf AMS Subject Classification: 13A50, 14N05, 15A72}\\[2ex]

%
Giorgio~OTTAVIANI,\\
Dipartimento di Matematica e Informatica ``U. Dini'', Università di Firenze\\
viale Morgagni 67/A, 50134 Firenze, ITALY\\
e-mail: \texttt{ottavian@math.unifi.it}\\[2ex]

%
\end{flushleft}

\normalsize
  

\footnotesize
  \begin{thebibliography}{999}
  
\bibitem{AC} \textsc{A. Abdessalam, J. Chipalkatti}, \textit{ On Hilbert covariants}, arXiv:1203.4761 .
   

\bibitem{BLRS} \textsc{R. Basson, R. Lercier, C. Ritzenthaler, J. Sijsling},
    \textit{An explicit expression of the L\"uroth invariant}, arXiv:1211.1327 .
 

\bibitem{Br} \textsc{M. Brion},
\textit{ Invariants et covariants des groupes algébriques réductifs}
Notes de l'école d'été "Théorie des invariants" (Monastir, 1996), available on M. Brion web page.
 
\bibitem{Brfaulkes} \textsc{M. Brion},
\textit{Sur certains modules gradués associés aux produits symétriques},  Algèbre non commutative, groupes quantiques et invariants (Reims, 1995), 157–183,
Sémin. Congr., 2, Soc. Math. France, Paris, 1997. 

\bibitem{BD} 
\textsc{ L. Báez-Duarte}, \textit{
Hardy-Ramanujan's asymptotic formula for partitions and the central limit theorem},
Adv. Math. 125 (1997), no. 1, 114--120. 

\bibitem{Be} \textsc{A. Beauville} ,  \textit{Determinantal hypersurfaces}, 
Michigan Math. J. 48 (2000), 39-64 .

\bibitem{Bed1} \textsc{L. Bedratyuk} ,   \textit{Analogue of Sylvester-Cayley formula for invariants of ternary form}, Ukrainian Mathematical Journal, 2010, 62 (11),  1810--1821.

\bibitem{Bed2} \textsc{L. Bedratyuk} ,   \textit{Analogue of Sylvester-Cayley formula for invariants of $n$-ary form},  Linear and Multilinear Algebra, 59,  (11), 2011, 1189--1199 .


\bibitem{BHORS} \textsc{G. Blekhermann, J.Hauenstein, J.C. Ottem, 
K. Ranestad, B. Sturmfels}, \textit{Algebraic boundaries of Hilbert's SOS cones },
 Comp. Math. 148 (2012) 1717--1735.


\bibitem{Cob} \textsc{A. Coble},  \textit{ Point sets and allied Cremona groups}, Trans. AMS, part I 16 (1915), part II 17 (1916).

\bibitem{Cre} \textsc{L. Cremona}, \textit{Ueber die Polar-Hexaeder bei den Fl\"achen dritter Ordnung},
Math. Ann. 13 (1877).


\bibitem{Cia} \textsc{E. Ciani}, \textit{Le curve piane di quarto ordine}, Giornale di Matematiche, 48 (1910),
259--304.

 \bibitem{DK} \textsc{H. Derksen, G. Kemper}, \textit{Computational Invariant Theory}, Encyclopaedia of Mathematical Sciences, 130, Springer, Berlin, 2002.


 \bibitem{D} \textsc{I. Dolgachev}, \textit{Classical Algebraic Geometry, a modern view} ,
Cambridge Univ. Press, 2012.

\bibitem{Do2} 
\textsc{I. Dolgachev}, \textit{Lectures on invariant theory}, London Math. Soc. Lecture Note Series, 296. Cambridge Univ. Press, Cambridge, 2003.

\bibitem{DoK}
\textsc{I. Dolgachev, V. Kanev}, \textit{Polar covariants of plane cubics and quartics}, {Advances in Math.} 98 (1993), 216-301.
 
\bibitem{EC} \textsc{F. Enriques and O. Chisini}, \textit{Lezioni sulla teoria geometrica delle equazioni e delle funzioni algebriche}, Zanichelli, Bologna 1915.

\bibitem{FS} \textsc{E. Freitag,R. Salvati Manni}, \textit{The modular variety of hyperelliptic curves of genus three}, Trans. Amer. Math. Soc. 363 (2011),  281--312.

\bibitem{FH} \textsc{W. Fulton and J. Harris}, \textit{Representation theory}, Springer GTM 129.

\bibitem{JK}  \textsc{ G. James and A. Kerber}, \textit{The representation theory of
the symmetric group}, Encyclopedia of Mathematics and its Applications, Addison Wesley 1981.

\bibitem{GY}  \textsc{J. Grace and A. Young}, \textit{The algebra of invariants}, Cambridge, 1903.

\bibitem{GS} 
\textsc{D.~Grayson and M.~Stillman},
{\sc Macaulay 2:} {\em a software system for research in algebraic
geometry}. Available at {\tt http://www.math.uiuc.edu/Macaulay2}.

\bibitem{Hilb} \textsc{D. Hilbert}, \textit{The theory of algebraic invariants},
Cambridge 1993.

\bibitem{HMSV1} \textsc{B. Howard, J. Millson, A. Snowden, R. Vakil}, \textit{The equations for the moduli space of $n$ points on the line},
Duke Math. J. 146 (2009), no. 2, 175--226.

\bibitem{HMSV8} \textsc{B. Howard, J. Millson, A. Snowden, R. Vakil}, \textit{The geometry of eight points in projective space: Representation theory, Lie theory, dualities}, to appear
in  Proc. LMS, arXiv:1103:5255. 

\bibitem{HMSV2} \textsc{B. Howard, J. Millson, A. Snowden, R. Vakil}, \textit{The ideal of relations for the ring of invariants of $n$ points on the line},
 J. Eur. Math. Soc.  14 (2012), no. 1, 1--60.

\bibitem{Ho} \textsc{R. Howe},
\textit{The classical groups and invariants of binary forms},
 in The Mathematical Heritage of Herman Weyl, 133–166, Proc. Sympos. Pure Math.
48, Amer. Math. Soc., Providence, 1988.

\bibitem{Ke} \textsc{A. Kempe}, \textit{On regular difference terms}, Proc. London Math. Soc. 25 (1894), 343--350.

\bibitem{KP} \textsc{H. Kraft, C. Procesi}, \textit{Classical invariant theory, a primer}, available on H. Kraft web page.

\bibitem{KR} \textsc{J. Kung, G. Rota}, \textit{The invariant theory of binary forms}, Bull. Amer. Math. Soc.  10 (1984), no. 1, 27--85.

\bibitem{Land} \textsc{J.M. Landsberg}, \textit{Tensors: Geometry and Applications}, Graduate Studies in Mathematics, 128, AMS, Providence, 2012.

\bibitem{LO} \textsc{J.M. Landsberg and G. Ottaviani}, {\it Equations for secant varieties of Veronese and other varieties},
 to appear in Annali di Matematica Pura e Applicata, arXiv:1111.4567 .

\bibitem{PT} \textsc{J. Le Potier and A. Tikhomirov},  \textit{Sur le morphisme de Barth},
{Ann. Sci. \'Ecole Norm. Sup.} (4) {\bf 34} (2001), no. 4, 573--629.

\bibitem{jL68}
\textsc{J. L\"uroth}, \textit{Einige Eigenshaften einer gewissen Gathung von Curves vierten Ordnung}, 
{Math. Annalen} 1 (1868), 38--53.
 
\bibitem{Man0} \textsc{L. Manivel}, \textit{Symmetric functions, Schubert polynomials and degeneracy loci}, SMF/AMS Texts and Monographs, vol. 6, 2001.

\bibitem{Man} \textsc{L. Manivel}, \textit{Prehomogeneous spaces and projective geometry}, note for School on Invariant Theory and Projective Geometry (Trento, 2012), this issue.

  \bibitem{Mor}
  \textsc{F. Morley}, \textit{On the L\"uroth Quartic Curve},   {American J. of Math.} 41 (1919), 279--282. 

\bibitem{Muk}\textsc{S. Mukai}, \textit{An introduction to invariants and moduli}, Cambridge Studies in Advanced Mathematics, 81. Cambridge University Press, Cambridge, 2003.
  
\bibitem{MN}\textsc{J. M\"uller, M.  Neunh\"offer}, \textit{ Some computations regarding Foulkes' conjecture}, Experiment. Math. 14 (2005), no. 3, 277--283.

\bibitem{Olv} \textsc{P. Olver}, \textit{Classical invariant theory}, London Mathematical Society Student Texts, 44, Cambridge University Press, Cambridge, 1999.

\bibitem{Ott}  \textsc{G. Ottaviani}, \textit{An invariant regarding Waring's problem for cubic polynomials}, Nagoya Math. J., 193 (2009), 95--110 .

 \bibitem{OS1} \textsc{G. Ottaviani and E. Sernesi}, \textit{On the hypersurface of L\"uroth quartics}, { Michigan Math. J.} 59 (2010), 365--394.

 \bibitem{Pau} \textsc{C. Pauly},
\textit{Self-duality of Coble's quartic hypersurface and applications},
Michigan Math. J. 50 (2002), no. 3, 551--574.

\bibitem{PS} \textsc{A. Popolitov, Sh. Shakirov},\textit{  
    On Undulation Invariants of Plane Curves}, arXiv:1208.5775 .
    

\bibitem{PV} \textsc{V.L.Popov, E.B. Vinberg}, \textit{Invariant Theory} in Encyclopaedia of Math. Sciences, Algebraic geometry IV, 55, Springer, Berlin, 1994.

\bibitem{Pr} \textsc{C. Procesi}, \textit{Lie groups, an approach through invariants and representations}, Universitext, Springer, New York, 2007.


\bibitem{RS} \textsc{K. Ranestad and F.O. Schreyer}, \textit{Varieties of sums of powers},
Journal f\"ur die reine und angew. Math., 525, (2000), 147--181 .


\bibitem{Ro} \textsc{G.C. Rota}, \textit{Two turning points in invariant theory},
The Mathematical Intelligencer, Volume 21(1), (1999),  20--27.


\bibitem{Sal} \textsc{G. Salmon}, {\it A treatise on higher plane curves},
 Hodges and Smith, Dublin, 1852
(reprinted from the 3d edition by Chelsea Publ. Co., New York, 1960).

\bibitem{Sc} \textsc{G. Scorza}, \textit{Un nuovo teorema sopra le quartiche piane generali}, Math. Ann.
52, (1899), 457-461.

\bibitem{Shi}\textsc{T. Shioda},  \textit{On the graded ring of invariants of binary octavics}, American J. of Math., 89(4), (1967), 1022--1046.


\bibitem{Stu} \textsc{B. Sturmfels}, \textit{Algorithms in Invariant Theory}, Springer, 1993, Wien.

\bibitem{Vac} \textsc{F. Vaccarino}, \textit{
Generalized symmetric functions and invariants of matrices},
Math. Z. 260 (2008), no. 3, 509--526. 

\bibitem{Wey} \textsc{J. Weyman}, \textit{Cohomology of vector bundles and syzygies}, Cambridge University Press, 2003.










  \end{thebibliography}
\end{document}